\documentclass{cedram-afst}

\usepackage{bm}
\usepackage[T1]{fontenc}
\usepackage[english]{babel}
\usepackage[utf8]{inputenc}
\usepackage[mathscr]{euscript}
\usepackage{textcomp,multicol,enumitem}
\usepackage[papersize={195mm,250mm},top=2cm, bottom=2cm, left=2cm , right=2cm]{geometry}

% ...

% ...

\title
[Algebraic independence]
{On the Algebraic Independence of $E$- and $G$-Functions, II: An Effective Version}

\alttitle{}

\author{\firstname{Daniel} \middlename{} \lastname{vargas-Montoya}}

\urladdr{}

\thanks{}

  
\altkeywords{}


\begin{document}

\begin{abstract}
Let $K$ be a finite extension of $\mathbb{Q}_p$ that is totally ramified over $\mathbb{Q}_p$. The set $\mathcal{M}\mathcal{F}(K)$ consists of power series in $1+zK[[z]]$ that are solutions of differential operators in $K(z)[d/dz]$ equipped with strong Frobenius structure and satisfying maximal order multiplicty (MOM) condition at zero. It turns out that this set contains an interesting class of $E$- and $G$-functions. In this work, we provide a criterion for determining the algebraic independence, over the field of analytic elements, of elements belonging to $\mathcal{M}\mathcal{F}(K)$. As an illustration of this criterion, we show the algebraic independence of some $E$- and $G$-functions over the field of analytic elements.
\end{abstract}

%% French abstract
%\begin{altabstract}
%Ceci est le r\'esum\'e fran\c cais.
%\end{altabstract}

\maketitle

%\tableofcontents

\section{Introduction}
\noindent This paper is a continuation of \cite{vargas6}, hereafter referred to as Part~I. 

In Part~I, we established a criterion for the algebraic independence of power series in $\mathcal{M}\mathcal{F}(K)$, where $K$ is a Frobenius field; that is, $K$ is a finite extension of $\mathbb{Q}_p$ such that $|x-x^p|<1$ for all $x\in\mathcal{O}_K$. Recall that the set $\mathcal{M}\mathcal{F}(K)$ consists of power series $f(z) \in 1+zK[[z]]$ such that $f(z)$ is a solution of a differential operator $\mathcal{L}\in E_p[\mathrm{d}/\mathrm{d}z]$ which is MOM (maximal order multiplicity) at zero and possesses a strong Frobenius structure. For the definitions of MOM and strong Frobenius structure, we refer the reader to Section~2.2 of Part~I. As highlighted in Section~2.4 of Part~I, $\mathcal{M}\mathcal{F}(K)$ contains an interesting class of $G$-functions and several $E$-functions. For example, for any prime number $p>2$, the classical $G$-functions
\[
\mathfrak{f}_r(z)=\sum_{n\geq0}\binom{2n}{n}^rz^n \quad \text{and} \quad \mathfrak{A}(z)=\sum_{n\geq0}\left(\sum_{k=0}^n\binom{n}{k}^2\binom{n+k}{k}^2\right)z^n
\]
belong to $\mathcal{M}\mathcal{F}(\mathbb{Q}_p)$. Note that $\mathfrak{f}_r(z)$ is the hypergeometric series ${}_rF_{r-1}((1/2,\ldots,1/2),(1,\ldots,1); 4^rz)$ and $\mathfrak{A}(z)$ is the Apéry power series. In addition, classical examples of $E$-functions, such as the exponential and Bessel functions given respectively by
\[
\exp(\pi_p z)=\sum_{n\geq0}\frac{\pi_p^n}{n!}z^n \quad \text{and} \quad J_0(\pi_p z)=\sum_{n\geq0}\frac{(-1)^n\pi_p^{2n}}{4^n(n!)^2}z^{2n},
\]
belong to $\mathcal{M}\mathcal{F}(K_p)$, where $K_p=\mathbb{Q}_p(\pi_p)$ and $\pi_p^{p-1}=-p$.

In addition, in Part~I we defined the field $E_K$ as the completion of $K(z)$ with respect to the Gauss norm. Furthermore, $K_0(z)$ denotes the ring of rational functions $A(z)/B(z)\in K(z)$ such that $B(x)\neq0$ for all $x\in D_0$, where $D_0=\{x\in\mathbb{C}_p: |x|<1\}$. Thus, the ring $E_{0,K}$ is the completion of $K_0(z)$ with respect to the Gauss norm. We then proved the following result:

\begin{theo}[Theorems~2.3 and 2.4 of Part~I]\label{theo_alg_ind}
Let $f_1(z),\ldots, f_m(z)\in\mathcal{M}\mathcal{F}(K)$, where $K$ is a Frobenius field. Then:
\begin{enumerate}
\item $f_1(z),\ldots, f_m(z)$ are algebraically dependent over $E_{K}$ if and only if there exist integers $a_1,\ldots, a_m\in\mathbb{Z}$, not all zero, such that
\[
f_1(z)^{a_1}\cdots f_m(z)^{a_m}\in E_{0,K};
\]
\item $f_1(z),\ldots, f_m(z)$ are algebraically dependent over $E_K$ if and only if, for any $(r_1,\ldots, r_m)\in\mathbb{N}^m$, the derivatives $f_1^{(r_1)},\ldots, f_m^{(r_m)}$ are algebraically dependent over $E_K$.
\end{enumerate}
\end{theo}

Inspired by this result, the goal of the present work is to study the transcendence and algebraic independence over $E_K$ of $E$- and $G$-functions that lie in $\mathcal{M}\mathcal{F}(K)$. A first illustration of Theorem~\ref{theo_alg_ind} is given in Section~\ref{sec_trans}, where we prove a conjecture due to Christol~\cite[p.~30]{C86}, which asserts that $J_0(\pi_p z)$ is transcendental over $E_{K_p}$ for $p > 2$. It is important to note that $\exp(\pi_p z)$ is algebraic over $E_{K_p}$ since $\exp(\pi_p z)^p \in E_{K_p}$.

However, in order to address the question of algebraic independence, it becomes necessary to establish a general criterion for determining when, for power series $f_1(z), \ldots, f_m(z) \in \mathcal{M}\mathcal{F}(K)$, the product $f_1(z)^{a_1} \cdots f_m(z)^{a_m} \notin E_{0,K}$ for any nonzero tuple $(a_1, \ldots, a_m) \in \mathbb{Z}^m$. Such a criterion is formulated in Theorem~\ref{theo_criterion_alg_ind}. The formulation of this result relies on the notion of residues and poles of analytic elements. As an application of this criterion, we prove in Section~\ref{sec_apli} the following result:

\begin{theo}\label{theo_appli}
Let $\pi_5 \in \overline{\mathbb{Q}}$ be such that $\pi_5^4 = -5$ and let $K = \mathbb{Q}_5(\pi_5)$. Consider the power series
\[
J_0(z)=\sum_{n\geq0}\frac{(-1)^n}{4^n(n!)^2}z^{2n}, \quad \mathfrak{F}(z)=\sum_{n\geq0}\left(\sum_{k=0}^n\binom{n}{k}^3\right)z^n,
\]
and
\[
\mathfrak{A}(z)=\sum_{n\geq0}\left(\sum_{k=0}^n\binom{n}{k}^2\binom{n+k}{k}^2\right)z^n.
\]
Then, for all integers $r,s, k\geq0$, the power series $J_0^{(r)}(\pi_5 z)$, $\mathfrak{A}^{(s)}(z)$, and $\mathfrak{F}^{(k)}(z)$ are algebraically independent over $E_{K}$.
\end{theo}

As a consequence of this theorem, we will prove in Corollary~\ref{coro_Q(z)} that, for all integers $r,s, k\geq0$, the power series $\exp(\pi_5 z)$, $J_0^{(r)}(\pi_5 z)$, $\mathfrak{A}^{(s)}(z)$, and $\mathfrak{F}^{(k)}(z)$ are algebraically independent over $\mathbb{Q}(z)$.

It is worth noting that the strategies given in \cite{ABD19} and \cite{vargas2} do not apply to the study of the algebraic independence of the power series $\mathfrak{A}(z)$, $\mathfrak{f}(z)$, $J_0(\pi_3 z)$, and $\exp(\pi_3 z)$. Indeed, while $\mathfrak{A}(z)$ is a $G$-function that is $p$-Lucas for all primes $p$, the power series $\mathfrak{f}(z)$ is a $G$-function that is not $p$-Lucas for any prime $p$. Thus, the main result of \cite{ABD19} is not applicable to this pair. However, the results of \cite{vargas2}---specifically Proposition~9.3---can be used to study the algebraic independence of $\mathfrak{A}(z)$ and $\mathfrak{f}(z)$. Nevertheless, the strategies developed in \cite{ABD19, vargas2} do not allow one to treat $E$-functions (note that $J_0(\pi_3 z)$ and $\exp(\pi_3 z)$ are $E$-functions). Finally, our method offers the additional advantage of handling the algebraic independence of their derivatives, something that cannot be addressed \textit{a priori} using the approaches of \cite{ABD19, vargas2}.

The second main result of this work is Theorem~\ref{theo_rational_function}, which establishes that, under certain conditions, $f_1(z),\ldots, f_m(z)\in\mathcal{M}\mathcal{F}(K)$ are algebraically independent over $E_p$ if and only if there exist integers $a_1,\ldots, a_m\in\mathbb{Z}$, not all zero, such that $f_1(z)^{a_1}\cdots f_m(z)^{a_m}\in\mathbb{C}_p(z)$. By combining this result with the techniques developed in \cite{ABD19}, we succeed in making this result explicit (see Corollary~\ref{theo_radius}). As a consequence, we show that the $G$-functions
\[
\mathfrak{F}(z)=\sum_{n\geq0}\left(\sum_{k=0}^n\binom{n}{k}^3\right)z^n \quad \text{and} \quad \mathfrak{A}(z)=\sum_{n\geq0}\left(\sum_{k=0}^n\binom{n}{k}^2\binom{n+k}{k}^2\right)z^n
\]
are algebraically independent over $\mathbb{Q}_5(z)$.

%The remainder of this section is devoted to establishing that the notion of residue is well-defined, to proving several of its key properties, and to discussing related properties concerning the notion of exponent. These properties will be useful in the proof of Theorem~\ref{theo_criterion_alg_ind}.

\section{Mains Results: Algebraic Independence over Analytic Elements}\label{sec_results}
In this section, we state our main result, which provides a criterion to determine when given power series $f_1(z), \ldots, f_m(z) \in \mathcal{M}\mathcal{F}(K)$ are algebraically independent over $E_K$. To state this result, we must first recall the notions of \emph{poles} and \emph{residues} for elements of \( E_{0,K} \).

\subsection{Exponents, poles, and residues}

Recall that $\mathcal{O}_{\mathbb{C}_p}$ is the set of elements in $\mathbb{C}_p$ with norm less than or equal to $1$. The maximal ideal of $\mathcal{O}_{\mathbb{C}_p}$ is denoted by $\mathfrak{m}$. For any $\alpha\in\mathcal{O}_{\mathbb{C}_p}$, we define $D_{\alpha}=\{x\in\mathcal{O}_{\mathbb{C}_p}: |x-\alpha|<1\}$. We also set $D_{\infty}=\{x\in\mathbb{C}_p:|x|\geq1\}$. Note that $\mathfrak{m}=D_0$.

\begin{rema}\label{rem_discos}
Let $\alpha,\beta \in \mathcal{O}_{\mathbb{C}_p}$. Since the norm is non-Archimedean, $D_{\alpha}\cap D_{\beta}=\emptyset$ if and only if $|\alpha-\beta|=1$, and $D_{\alpha}=D_{\beta}$ if and only if $|\alpha-\beta|<1$.
\end{rema}

Let $\alpha \in \mathcal{O}_{\mathbb{C}_p}$ and let $\mathbb{C}_p[z]_{\mathfrak{m}_{\alpha}}$ be the set of rational functions $R(z) \in \mathbb{C}_p(z)$ such that every pole of $R(z)$ belongs to $D_{\alpha}$. We let $E^{\alpha}$ denote the completion of $\mathbb{C}_p[z]_{\mathfrak{m}_{\alpha}}$ with respect to the Gauss norm. For any set $A\subset\mathbb{C}_p$, $E(A)$ denotes the completion of the ring of rational functions in $\mathbb{C}_p(z)$ that do not have poles in $A$. 

Mittag-Leffler's Theorem (see \cite[Theorem~15.1]{E95}) states that if $f\in E_{0,K}$ is bounded\footnote{We say that $f\in E_{0,K}$ is bounded if there exists a constant $C>0$ such that $|f(x)|\leq C$ for all $x\in D_0$.}, then there exists a unique sequence $\{\alpha_n\}_{n\geq1}$ in $\mathcal{O}_{\mathbb{C}_p}$ such that $|\alpha_n|=1$ for every $n$, $|\alpha_n-\alpha_m|=1$ for all $n\neq m$, and a unique sequence $\{f_n\}_{n\geq0}$ in $E_{0,K}$ such that $f_0 \in E(\mathcal{O}_{\mathbb{C}_p})$ and, for every $n>0$, $f_n \in E^{\alpha_n}$, satisfying
\begin{equation}\label{eq_mittag}
f=\sum_{n=0}^{\infty}f_n.
\end{equation}
Moreover, $|f_{n}|_{\mathcal{G}}\leq|f|_{\mathcal{G}}$ for all $n\geq0$.

We first introduce the notions of principal term and singular term.

\begin{defi}[Principal term and singular term]
Let $f \in E_{0,K}$ be bounded. The \emph{principal term} of $f$ is the element $f_0$, and the \emph{singular term} of $f$ associated with $D_{\alpha_n}$ is the element $f_n$ appearing in the decomposition~\eqref{eq_mittag}.
\end{defi}

We are now ready to give the definition of a pole and a residue for an analytic element in $E_{0,K}$.

\begin{defi}[Pole]
Let $f \in E_{0,K}$ be bounded and let $\alpha \in \mathcal{O}_{\mathbb{C}_p}$. We say that $D_{\alpha}$ is a \emph{pole} of $f$ if there exists $n \ge 1$ such that $\alpha_n \in D_{\alpha}$ and $f_n \neq 0$. We denote the set of poles of $f$ by $\bm{\mathrm{Pol}}(f)$.
\end{defi}

\begin{defi}[Residue]
Let $f \in E_{0,K}$ be bounded, and let $\alpha \in \mathcal{O}_{\mathbb{C}_p}$ be such that $D_{\alpha}$ is a pole of $f$. According to the decomposition given in Equation~\eqref{eq_mittag}, there exists $n \ge 1$ such that $\alpha_n \in D_{\alpha}$ and $f_n \neq 0$. Since $f_n \in E^{\alpha_n}$, $f_n$ is the limit of rational functions $R_j(z) \in \mathbb{C}_p(z)$ whose poles all lie in $D_{\alpha_n}$. Then, for every $j\geq0$, we have
\[
R_j(z) = \sum_{s \in \mathbb{Z}} a_{s,j} (z - \alpha)^s
\]
with $a_{s,j} \in \mathbb{C}_p$ for all $s \in \mathbb{Z}$.

We then define the residue of $f$ at $D_{\alpha}$ as
\[
\mathrm{Res}(f, D_{\alpha}) = \lim_{j \to \infty} a_{-1,j}.
\]
\end{defi}

We will see in Lemma~\ref{lem_poles} that this definition depends neither on the particular representation of $f_n$ as a limit of rational functions nor on the choice of the disk $D_{\alpha}$. The latter means that if $\beta \in D_{\alpha}$, then, by Remark~\ref{rem_discos}, $D_{\alpha} = D_{\beta}$, and thus we must check that $\mathrm{Res}(f, D_{\alpha}) = \mathrm{Res}(f, D_{\beta})$. Now, let us recall the following result from Part~I.

\begin{theo}[Theorem~3.1 of Part~I]\label{theo_analytic_element}
Let $K$ be a Frobenius field and let $f(z) \in \mathcal{M}\mathcal{F}(K)$. Then:
\begin{enumerate}[label=(\roman*)]
\item $f(z)$ belongs to $1 + z\mathcal{O}_K[[z]]$,
\item $f(z)/f(z^{p^h}) \in E_{0,K}$ for some integer $h > 0$,
\item $f'(z)/f(z) \in E_{0,K}$.
\end{enumerate}
\end{theo}

Thanks to this result, if $f(z)$ belongs to $\mathcal{M}\mathcal{F}(K)$, then $f'(z)/f(z)$ is a bounded element of $E_{0,K}$; thus, it makes sense to speak of the poles and residues of $f'(z)/f(z)$. Furthermore, we will prove in Lemma~\ref{lem_res} that for any $D_{\alpha} \in\bm{\mathrm{Pol}}(f'(z)/f(z))$, we have $\mathrm{Res}(f'(z)/f(z), D_{\alpha}) \in \mathbb{Z}_p$.

\begin{rema}
It follows from Part~I that the integer $h > 0$ appearing in (ii) is the period of the strong Frobenius structure of the operator $\mathcal{L} \in E_p[\mathrm{d}/\mathrm{d}z]$ satisfying $\mathcal{L}(f) = 0$.
\end{rema}

Another important ingredient in the statement of our main results involves the concepts of \emph{singular points} and \emph{exponents} for differential operators with coefficients in rings that are not necessarily integral domains. Our definitions are inspired by the classical setting. Therefore, let us first recall these concepts in the classical case, i.e., for differential operators with coefficients in $\mathbb{Q}(z)$.

Let $\mathcal{L} = \frac{\mathrm{d}^n}{\mathrm{d}z^n} + b_1(z)\frac{\mathrm{d}^{n-1}}{\mathrm{d}z^{n-1}} + \cdots + b_{n-1}(z)\frac{\mathrm{d}}{\mathrm{d}z} + b_n(z) \in \mathbb{Q}(z)[\mathrm{d}/\mathrm{d}z]$. We say that $\alpha \in \overline{\mathbb{Q}}$ is a regular singular point of $\mathcal{L}$ if $(z-\alpha)^j b_j(z) \in \overline{\mathbb{Q}}[[z-\alpha]]$ for all $1 \le j \le n$. Furthermore, the exponents of $\mathcal{L}$ at $\alpha$ are the roots of the polynomial
\[
X(X-1)\cdots(X-n+1) + s_1 X(X-1)\cdots(X-n+2) + \cdots + s_{n-1}X + s_n,
\]
where $s_j = [(z-\alpha)^j b_j(z)](\alpha)$.

Now, let $K$ be a finite extension of $\mathbb{Q}_p$, and let $\mathcal{O}_{K(z)}$ be the set of rational functions $R(z) \in K(z)$ with Gauss norm less than or equal to $1$. Let $\mathscr{M}_K$ denote the set of rational functions $R(z) \in K(z)$ with Gauss norm strictly less than $1$. Note that $\mathscr{M}_K$ is the maximal ideal of $\mathcal{O}_{K(z)}$. For any operator $\mathcal{L} \in \mathcal{O}_{K(z)}[\mathrm{d}/\mathrm{d}z]$ and any integer $a \ge 1$, we can reduce $\mathcal{L}$ modulo $\mathscr{M}_K^{a}$. We denote this reduction by $\mathcal{L}_{\mid\mathscr{M}_K^{a}}$.

\begin{defi}[Singular point and exponents] Let $\mathcal{L} = \frac{\mathrm{d}^n}{\mathrm{d}z^n} + b_1(z)\frac{\mathrm{d}^{n-1}}{\mathrm{d}z^{n-1}} + \cdots + b_{n-1}(z)\frac{\mathrm{d}}{\mathrm{d}z} + b_n(z)$ be an element of $\mathcal{O}_{K(z)}[\mathrm{d}/\mathrm{d}z]$, where $K$ is a finite extension of $\mathbb{Q}_p$. Let $\alpha \in \overline{\mathbb{Q}_p}$ with norm $1$, and set $\mathbf{L} = K(\alpha)$. Let $a \ge 1$ be an integer. We say that $\alpha$ is a regular singular point of $\mathcal{L}_{\mid\mathscr{M}_K^a}$ if, for every $1 \le j \le n$, we have $(z-\alpha)^j b_j(z) \equiv F_j(z) \bmod{\mathscr{M}_{\mathbf{L}}^a \mathcal{O}_{\mathbf{L}(z)}}$, where $F_j(z) \in \mathbf{L}(z) \cap \mathcal{O}_{\mathbf{L}}[[z-\alpha]]$.

Furthermore, consider the polynomial
\[
P_{\alpha, a}(X) = X(X-1)\cdots(X-n+1) + s_1 X(X-1)\cdots(X-n+2) + \cdots + s_{n-1}X + s_n,
\]
where $s_j:= F_j(\alpha) \bmod \pi^a_{\mathbf{L}}\mathcal{O}_{\mathbf{L}}$, and $\pi_{\mathbf{L}}$ is a uniformizer of $\mathbf{L}$.

Any root of $P_{\alpha, a}(X)$ is called an exponent of $\mathcal{L}_{\mid\mathscr{M}_K^a}$ at $\alpha$. The polynomial $P_{\alpha, a}(X)$ is referred to as the indicial polynomial of $\mathcal{L}_{\mid\mathscr{M}_K^a}$ at $\alpha$. Note that $P_{\alpha,a}(X)$ has coefficients in $\mathcal{O}_{\mathbf{L}}/\pi^a_{\mathbf{L}}\mathcal{O}_{\mathbf{L}}$.
\end{defi}

 \subsection{Main results}

We are now ready to state our criterion for algebraic independence.

\begin{theo}\label{theo_criterion_alg_ind}
Let $K$ be a Frobenius field, let $f_1(z),\ldots, f_m(z)\in\mathcal{M}\mathcal{F}(K)$, and let $\mathcal{L}_1,\ldots,\mathcal{L}_m\in\mathcal{O}_{K(z)}[\mathrm{d}/\mathrm{d}z]$ be such that $\mathcal{L}_i(f_i)=0$ for all $1\leq i\leq m$. Suppose that:
\begin{enumerate}
\item for every $1\leq i\leq m$, there exists $\alpha_{c_i}\in\overline{\mathbb{Q}_p}$ of norm $1$ such that $D_{\alpha_{c_i}}\in\bm{\mathrm{Pol}}(f'_i(z)/f_i(z))$, $D_{\alpha_{c_i}}\notin\bm{\mathrm{Pol}}(f'_j(z)/f_j(z))$ for all $j\neq i$, and $\mathrm{Res}\left(\frac{f'_i(z)}{f_i(z)},D_{\alpha_{c_i}}\right)\neq 0$;\smallskip

\item for every $1\leq i\leq m$, $\alpha_{c_i}$ is a regular singular point of $\mathcal{L}_i$;\smallskip

\item there exists an integer $a>0$ such that, for every $1\leq i\leq m$, the Gauss norm of the principal part of $f'_i(z)/f_i(z)$ is less than $1/p^a$.
\end{enumerate}

If, for every $1\leq i\leq m$, $\mathrm{Res}\left(\frac{f'_i(z)}{f_i(z)},D_{\alpha_{c_i}}\right)\bmod p^a$ is not an exponent of $\mathcal{L}_{i\mid\mathscr{M}^a_K}$ at $\alpha_{c_i}$, then $f_1(z),\ldots, f_m(z)$ are algebraically independent over $E_K$.
\end{theo}

If condition~(1) fails, then the conclusion of Theorem~\ref{theo_criterion_alg_ind} is no longer true. Indeed, according to Theorem~\ref{theo_h_A}, the power series
\[
\mathfrak{f}_3(z)=\sum_{n\geq0}\frac{1}{64^n}\binom{2n}{n}^3z^n \quad \text{and} \quad \mathfrak{h}(z)=\sum_{n\geq0}\frac{-1}{(2n-1)64^n}\binom{2n}{n}^3z^n
\]
are algebraically dependent over $E_{\mathbb{Q}_3}$. However, condition~(1) is not satisfied because $\bm{\mathrm{Pol}}(\mathfrak{f}_3'(z)/\mathfrak{f}_3(z))=\{D_{1}\}$ and $D_{1}\in\bm{\mathrm{Pol}}(\mathfrak{h}'(z)/\mathfrak{h}(z))$.

Using this theorem, we prove that $\mathfrak{f}_2(z)$ and $\mathfrak{A}(z)$ are algebraically independent over $E_{\mathbb{Q}_3}$. Furthermore, this theorem can be used to prove that $\mathfrak{f}_r(z)$ is transcendental over $E_{\mathbb{Q}_3}$ for all $r\ge 2$.

We observe that in many examples condition~(3) of Theorem~\ref{theo_criterion_alg_ind} is satisfied for $a = 2$. Indeed, we show that in some cases the principal part of $f'/f$ is zero. For such power series, we have the following criterion.

\begin{theo}\label{theo_rational_function}
Let $K$ be a Frobenius field and let $f_1(z),\ldots, f_m(z) \in \mathcal{M}\mathcal{F}(K)$. Suppose that condition~(1) of Theorem~\ref{theo_criterion_alg_ind} is satisfied and that, for every $1 \le i \le m$, the principal part of $f'_i/f_i$ is zero. If $f_1(z),\ldots, f_m(z)$ are algebraically dependent over $E_K$, then there exist $d_1,\ldots, d_m \in \mathbb{Z}$, not all zero, such that
\[
f^{d_1}_1(z)\cdots f^{d_m}_m(z) \in \mathbb{C}_p(z)
\]
and $d_s$ is not divisible by $p$ for some $1 \le s \le m$.

More precisely,
\[
f^{d_1}_1(z)\cdots f^{d_m}_m(z) = \prod_{i=1}^m(z-\alpha_{c_i})^{\Gamma_i} \mathfrak{a}_1(z) \in \mathbf{L}(z),
\]
where $\mathbf{L}$ is a finite extension of $K$, $\mathfrak{a}_1(z) \in \mathbf{L}(z)$, and $\Gamma_i = d_i \cdot \mathrm{Res}\left(\frac{f'_i(z)}{f_i(z)}, D_{\alpha_{c_i}}\right) \in \mathbb{Z}^*$ for all $1 \le i \le m$. We also have:
\begin{enumerate}
\item[(a)] $|\mathfrak{a}_1(z)|_{\mathcal{G}} = 1$;
\item[(b)] $\mathfrak{a}_1(z) \in E(D_{\alpha_{c_i}})$ and $\mathfrak{a}_1(\alpha_{c_i}) \neq 0$ for every $1 \le i \le m$;
\item[(c)] the poles and zeros of $\mathfrak{a}_1(z)$ have norm $1$.
\end{enumerate}

In particular, the following statements are equivalent:
\begin{enumerate}
\item[(i)] $f_1(z),\ldots, f_m(z)$ are algebraically dependent over $E_K$;
\item[(ii)] $f_1(z),\ldots, f_m(z)$ are algebraically dependent over $K(z)$.
\end{enumerate}
\end{theo}

This result is quite surprising, as it shows that, in some cases, being algebraically dependent over $E_K$ is equivalent to being algebraically dependent over $K(z)$. This equivalence does not hold in general if we drop the assumption that the principal part of $f'/f$ is zero. In fact, as we will see, for any prime number $p$, the function $\exp(\pi_p z)$ belongs to $\mathcal{M}\mathcal{F}(\mathbb{Q}_p(\pi_p))$ and is algebraic over $E_{\mathbb{Q}_p(\pi_p)}$. Nevertheless, $\exp(\pi_p z)$ is transcendental over $\mathbb{Q}_p(z)$, while the principal part of $\exp(\pi_p z)'/\exp(\pi_p z)$ is $\pi_p$, which is non-zero.

An interesting consequence of Theorem~\ref{theo_rational_function} is that one can deduce algebraic independence over $E_K$ by studying the behavior of the complex radius of convergence. More precisely, for an analytic function $f(z) \in \mathbb{C}\{z\}$, we denote its radius of convergence by $\rho_f$. As in \cite[Definition~7.3]{ABD19}, let $\mathfrak{M}$ denote the set of all analytic functions $f(z) \in \mathbb{C}\{z\}$ whose radius of convergence is finite and for which there exists $x \in \mathbb{C}$ with absolute value $\rho_f$ such that no positive integer power of $f(z)$ admits a meromorphic continuation to a neighborhood of $x$. We then have the following corollary.

\begin{coro}\label{theo_radius}
Let $f_1(z),\ldots, f_m(z) \in 1 + z\mathbb{Q}[[z]]$ be such that $f_1(z),\ldots, f_m(z) \in \mathcal{M}\mathcal{F}(\mathbb{Q}_p) \cap \mathfrak{M}$ for some prime number $p$. Suppose that condition~(1) of Theorem~\ref{theo_criterion_alg_ind} is satisfied and that, for every $1 \le i \le m$, the principal part of $f'_i/f_i$ is zero. If $\rho_{f_1},\ldots, \rho_{f_m}$ are pairwise distinct, then for any $(s_1,\ldots, s_m) \in \mathbb{N}^m$, the derivatives $f^{(s_1)}_1(z),\ldots, f^{(s_m)}_m(z)$ are algebraically independent over $E_{\mathbb{Q}_p}$.
\end{coro}

Thanks to this corollary, we are able to prove that the power series
\[
\mathfrak{F}(z) = \sum_{n \ge 0} \left( \sum_{k=0}^n \binom{n}{k}^3 \right) z^n \quad \text{and} \quad \mathfrak{A}(z) = \sum_{n \ge 0} \left( \sum_{k=0}^n \binom{n}{k}^2 \binom{n+k}{k}^2 \right) z^n
\]
are algebraically independent over $E_{\mathbb{Q}_5}$.

\section{Residues and Exponents}\label{sec_res_pol}
In this section, we show that the notion of residue for analytic elements is well defined and establish several key properties of residues and exponents that will be useful in the sequel.

\subsection{Residues}
We start with the following preliminary lemma.

\begin{lemm}\label{lem_conv}
Let $\alpha, \beta \in \mathbb{C}_p$.
\begin{enumerate}[label=(\roman*)]
\item If $|\alpha-\beta| < 1$, then 
\[
\frac{1}{z-\beta} = \frac{1}{z-\alpha} \sum_{l \ge 0} \left(\frac{\beta-\alpha}{z-\alpha}\right)^l.
\]
\item If $|\alpha-\beta| \ge 1$, then $\frac{1}{z-\beta} = \sum_{s \ge 0} a_s (z-\alpha)^s$ with $a_s \in \mathcal{O}_{\mathbb{C}_p}$ for all $s \ge 0$.
\end{enumerate}
\end{lemm}

\begin{proof}
(i)\quad Since $|\alpha-\beta| < 1$, we have $\lim_{l \to \infty} |\alpha-\beta|^l = 0$. If $|\alpha| \le 1$, then $\left|\frac{\beta-\alpha}{z-\alpha}\right|_{\mathcal{G}} = |\beta-\alpha| < 1$. If $|\alpha| > 1$, then $\left|\frac{\beta-\alpha}{z-\alpha}\right|_{\mathcal{G}} = \frac{|\beta-\alpha|}{|\alpha|} < 1$ because $|\beta-\alpha| < 1 < |\alpha|$. In both cases, $\lim_{l \to \infty} \left|\left(\frac{\beta-\alpha}{z-\alpha}\right)^l\right|_{\mathcal{G}} = 0$, and thus the series $\sum_{l \ge 0} \left(\frac{\beta-\alpha}{z-\alpha}\right)^l$ converges in $E_p$. Set $x_0 = \frac{1}{z-\alpha} \sum_{l \ge 0} \left(\frac{\beta-\alpha}{z-\alpha}\right)^l$. By summing the geometric series, we obtain
\[
x_0 = \frac{1}{z-\alpha} \cdot \frac{1}{1 - \frac{\beta-\alpha}{z-\alpha}} = \frac{1}{z-\alpha} \cdot \frac{z-\alpha}{z-\beta} = \frac{1}{z-\beta}.
\]

(ii)\quad The assertion follows directly from the Taylor expansion
\[
\frac{1}{z-\beta} = \frac{1}{\alpha-\beta} \left(1 - \frac{z-\alpha}{\beta-\alpha}\right)^{-1} = \sum_{s \ge 0} \frac{(-1)^s}{(\alpha-\beta)^{s+1}} (z-\alpha)^s,
\]
where $a_s = \frac{(-1)^s}{(\alpha-\beta)^{s+1}} \in \mathcal{O}_{\mathbb{C}_p}$ since $|\alpha-\beta| \ge 1$.
\end{proof}

\begin{rema}\label{rem_poles}
Let $R(z) \in \mathbb{C}_p(z)$ and let $\alpha_1,\ldots, \alpha_r$ be the poles of $R(z)$. Suppose that $|\alpha_i| = 1$ for all $1 \le i \le r$. Since $\mathbb{C}_p$ is algebraically closed, partial fraction decomposition yields
\[
R(z) = T(z) + F_1(z) + \cdots + F_r(z),
\]
where $T(z)$ is a polynomial and $F_i(z) = \sum_{k=1}^{n_i} \frac{a_{i,k}}{(z-\alpha_i)^k}$, with $a_{i,k} \in \mathbb{C}_p$ and $n_i$ being the order of the pole of $R(z)$ at $\alpha_i$. In particular, for all $i \in \{1,\ldots, r\}$, we have $\mathrm{Res}(R(z), \alpha_i) = a_{i,1}$.
\end{rema}

The following lemma is crucial to establish that the notion of residue is well defined.

\begin{lemm}\label{lem_poles}
We keep the notation introduced in Remark~\ref{rem_poles}.
\begin{enumerate}[label=(\roman*)]
\item Let $\alpha \in \{\alpha_1,\ldots,\alpha_r\}$. Then 
\[
R(z) = \sum_{s \in \mathbb{Z}} a_s (z-\alpha)^s,
\]
where $a_s \in \mathbb{C}_p$ for all $s \in \mathbb{Z}$, and $a_{-1} = \sum_{\alpha_i \in S} a_{i,1}$, where $S$ is the set of poles of $R(z)$ in $D_{\alpha}$.

\item Let $\alpha \in \{\alpha_1,\ldots,\alpha_r\}$ and let $\tau \in D_{\alpha}$. Then 
\[
R(z) = \sum_{s \in \mathbb{Z}} b_s (z-\tau)^s,
\]
where $b_s \in \mathbb{C}_p$ for all $s \in \mathbb{Z}$, and $b_{-1} = \sum_{\alpha_i \in S} a_{i,1}$, where $S$ is the set of poles of $R(z)$ in $D_{\alpha}$.
\end{enumerate}
\end{lemm}

\begin{proof}
(i)\quad Without loss of generality, we can assume that $S = \{\alpha_1,\ldots,\alpha_r\} \cap D_{\alpha} = \{\alpha_1,\ldots,\alpha_q\}$. We decompose $R(z)$ as $R(z) = T(z) + \sum_{i=1}^q F_i(z) + \sum_{i=q+1}^r F_i(z)$. 

By Lemma~\ref{lem_conv}~(ii), we have $F_i(z) \in \mathcal{O}_{\mathbb{C}_p}[[z-\alpha]]$ for all $q+1 \le i \le r$, so $T(z) + \sum_{i=q+1}^r F_i(z) \in \mathbb{C}_p[[z-\alpha]]$. In addition, by Lemma~\ref{lem_conv}~(i), we deduce that for all $1 \le i \le q$,
\[
\frac{1}{z-\alpha_i} = \frac{1}{z-\alpha}\sum_{l \ge 0} \left(\frac{\alpha_i-\alpha}{z-\alpha}\right)^l.
\]
Substituting this expansion into $F_i(z) = \sum_{k=1}^{n_i} \frac{a_{i,k}}{(z-\alpha_i)^k}$, we see that each $F_i(z)$ contributes $a_{i,1} (z-\alpha)^{-1}$ to the coefficient of $(z-\alpha)^{-1}$, while higher-power terms $(z-\alpha_i)^{-k}$ with $k \ge 2$ yield only powers $(z-\alpha)^s$ with $s \le -2$. Therefore, $R(z) = \sum_{s \in \mathbb{Z}} a_s (z-\alpha)^s$ with $a_s \in \mathbb{C}_p$ for all $s \in \mathbb{Z}$ and $a_{-1} = \sum_{i=1}^q a_{i,1}$.

(ii)\quad By item~(i), we know that $R(z) = \sum_{s \in \mathbb{Z}} a_s (z-\alpha)^s$, where $a_{-1} = \sum_{\alpha_i \in S} a_{i,1}$. We split the Laurent series as
\[
R(z) = \sum_{s \le -1} a_s (z-\alpha)^s + \sum_{s \ge 0} a_s (z-\alpha)^s.
\]
Since $z-\alpha = (z-\tau) + (\tau-\alpha)$, we expand $(z-\alpha)^s = \sum_{k=0}^s \binom{s}{k} (\tau-\alpha)^{s-k} (z-\tau)^k$ for $s \ge 0$. Because $|\tau-\alpha| < 1$, the series converges and yields $\sum_{s \ge 0} a_s (z-\alpha)^s \in \mathbb{C}_p[[z-\tau]]$.

For $s \le -1$, since $|\tau-\alpha| < 1$, Lemma~\ref{lem_conv}~(i) yields
\[
\frac{1}{z-\alpha} = \frac{1}{z-\tau} \sum_{l \ge 0} \left(\frac{\alpha-\tau}{z-\tau}\right)^l = \frac{1}{z-\tau} + \sum_{l \ge 1} \frac{(\alpha-\tau)^l}{(z-\tau)^{l+1}}.
\]
Taking powers of this expansion, we observe that the coefficient of $(z-\tau)^{-1}$ in $\sum_{s \le -1} a_s (z-\alpha)^s$ comes uniquely from the leading term $a_{-1} (z-\tau)^{-1}$ (since higher negative powers $(z-\alpha)^{-k}$ for $k \ge 2$ only produce powers $(z-\tau)^s$ with $s \le -2$). 

Thus, $R(z) = \sum_{s \in \mathbb{Z}} b_s (z-\tau)^s$ with $b_{-1} = a_{-1} = \sum_{\alpha_i \in S} a_{i,1}$.
\end{proof}

The following lemma shows that the notion of residue is well defined.
\begin{lemm}\label{lemm_good_def}
Let $f$ be in $E^{\alpha}$ with $|\alpha|=1$. Suppose that $$\lim\limits_{j\rightarrow\infty}R_j(z)=f(z)=\lim\limits_{j\rightarrow\infty}Q_j(z)$$ such that, for every $j\geq0$, $R_j(z), Q_j(z)\in\mathbb{C}_p(z)$ and all their poles lie in $D_{\alpha}$. Let  $\bm{S}_{R_j}$ denote the set of poles of $R_j(z)$, and let $\bm{S}_{Q_j}(z)$ denote the set of poles of $Q_j(z)$. Then, for each $j$, $$R_j(z)=\sum_{n\in\mathbb{Z}}a_{s,j}(z-\alpha)^s\quad\text{ and }\quad Q_j(z)=\sum_{n\in\mathbb{Z}}b_{s,j}(z-\alpha)^s,$$
where $$a_{-1,j}=\sum_{\tau\in\bm{S}_{R_j}}\tau\quad\text{ and }\quad b_{-1,j}=\sum_{\gamma\in\bm{S}_{Q_j}}\gamma.$$
We also have $$\lim\limits_{j\rightarrow\infty}a_{-1,j}=\lim\limits_{j\rightarrow\infty}b_{-1,j}.$$

Moreover, if $D_{\alpha}=D_{\beta}$ then, for each $j$, $$R_j(z)=\sum_{n\in\mathbb{Z}}c_{s,j}(z-\beta)^s\quad\text{ and }\quad Q_j(z)=\sum_{n\in\mathbb{Z}}c_{s,j}(z-\beta)^s,$$
where $$c_{-1,j}=\sum_{\tau\in\bm{S}_{R_j}}\tau\quad\text{ and }\quad d_{-1,j}=\sum_{\gamma\in\bm{S}_{Q_j}}\gamma.$$
Consequently, $$\mathrm{Res}(f, D_{\alpha})=\mathrm{Res}(f, D_{\beta}).$$

\end{lemm}

\begin{proof}
Let $j\geq0$ be an integer. Let $\gamma_{1}$ be a pole of $R_j(z)$ and $\gamma_2$ be a pole of $Q_j(z)$. Then, $\gamma_1,\gamma_2\in D_{\alpha}$ and thus, it follows from Remark~\ref{rem_discos} that $\alpha\in D_{\gamma_1}\cap D_{\gamma_2}$. Since $\alpha\in D_{\gamma_1}$, it follows from part (ii) of Lemma~\ref{lem_poles} that $$R_j(z)=\sum_{n\in\mathbb{Z}}a_{s,j}(z-\alpha)^s\quad\text{ with }\quad a_{-1,j}=\sum_{\tau\in\bm{S}_{R_j}}\tau.$$ 
Similarly, since $\alpha\in D_{\gamma_2}$,  it follows from part (ii) of Lemma~\ref{lem_poles} that $$Q_j(z)=\sum_{n\in\mathbb{Z}}b_{s,j}(z-\alpha)^s,\quad\text{ with }\quad b_{-1,j}=\sum_{\gamma\in\bm{S}_{Q_j}}\gamma.$$
Since $\lim\limits_{j\rightarrow\infty}R_j(z)=f(z)=\lim\limits_{j\rightarrow\infty}Q_j(z)$, it follows that, for every $s\in\mathbb{Z}$,  $\lim\limits_{j\rightarrow\infty}|a_{s,j}-b_{s,j}|=0$. In particular, $\lim\limits_{j\rightarrow\infty}a_{-1,j}=\lim\limits_{j\rightarrow\infty}b_{-1,j}$.

Since $D_{\alpha}=D_{\beta}$, and $\alpha\in  D_{\gamma_1}\cap D_{\gamma_2}$, it follows from Remark~\ref{rem_discos} that $\beta\in D_{\gamma_1}\cap D_{\gamma_2}$. Thus, from part (ii) of Lemma~\ref{lem_poles} again, we conclude that $$R_j(z)=\sum_{n\in\mathbb{Z}}c_{s,j}(z-\beta)^s\quad\text{ and }\quad Q_j(z)=\sum_{n\in\mathbb{Z}}c_{s,j}(z-\beta)^s,$$
where $c_{-1,j}=\sum_{\tau\in\bm{S}_{R_j}}\tau$ and $d_{-1,j}=\sum_{\gamma\in\bm{S}_{Q_j}}\gamma$.
By definition, $$\mathrm{Res}(f, D_{\alpha})=\lim\limits_{j\rightarrow\infty}a_{-1,j}\quad\text{ and }\quad \mathrm{Res}(f, D_{\beta})=\lim\limits_{j\rightarrow\infty}c_{-1,j}.$$
But, we know that, for every $j\geq0$, $a_{-1,j}=\sum_{\tau\in\bm{S}_{R_j}}\tau=c_{-1,j}$. Consequently,  $\mathrm{Res}(f, D_{\alpha})=\mathrm{Res}(f, D_{\beta})$.
\end{proof}

As an immediate consequence of this lemma, we conclude that the notion of residue for a bounded element in $E_{0,K}$ is well-defined. Furthermore, from Lemma~\ref{lem_poles}, we also have

\begin{rema}\label{rem_pol_rational}
Let $R(z)$ be in $\mathbb{C}_p(z)$ and let $\alpha_1,\ldots, \alpha_r$ be the poles of $R(z)$. Suppose that $|\alpha_i|=1$ for all $1\leq i\leq r$. Let $\alpha$ be in $\{\alpha_1,\ldots,\alpha_r\}$. We then have $\mathrm{Res}(R, D_{\alpha})=\sum_{\alpha_i\in D_{\alpha}}\mathrm{Res}(R,\alpha_i).$
\end{rema}
The following lemma describes several key properties of the residue of analytic elements of the form $f'(z)/f(z)\in E_{0,p}$, where $f(z)\in K[[z]]$.
\begin{lemm}\label{lem_res}
Let $K$ be a finite extension of $\mathbb{Q}_p$ and let $f(z)$ be in $1+z\mathcal{O}_{K}[[z]]$ such that $f'(z)/f(z)\in E_{0,p}$. 
\begin{enumerate}[label=(\roman*)]
\item If $D_{\alpha}$ is a pole of $f'(z)/f(z)$ then $\mathrm{Res}(f'(z)/f(z),D_{\alpha})$ belongs to $\mathbb{Z}_p$.
\item Let $R(z)=\prod_{i=1}^s(z-\beta_i)^{n_i}$ where, for each $1\leq i\leq s$, $n_i\in\mathbb{Z}\setminus\{0\}$ and $|\beta_i|=1$. If $\left|\frac{f'(z)}{f(z)}-\frac{R'(z)}{R(z)}\right|_{\mathcal{G}}\leq 1/p^n$ and  $D_{\alpha}$ is a pole of $f'(z)/f(z)$ then $$\mathrm{Res}(f'(z)/f(z),D_{\alpha})\equiv\left(\sum_{\beta_i\in D_{\alpha}}n_i\right)\bmod p^n\mathbb{Z}_p.$$

\end{enumerate}

\end{lemm}

\begin{proof}\hfill

(i) We first show by induction on $n\geq1$ that there exists a sequence $\{R_{n}(z)\}_{n\geq1}$ of rational functions with coefficients in $\mathbb{C}_p(z)$ such that, for all $n$, $R_n(x)\neq0$  for all $x\in D_0$ and 
\begin{equation}\label{eq_dl}
\left| \frac{f'(z)}{f(z)}-\sum_{i=1}^n\frac{R_i'(z)}{R_i(z)}\right|_{\mathcal{G}}<1/p^n.
\end{equation}

Since $f(x)$ converges for all $x\in D_0$ and  $f'(z)/f(z)\in E_{0,p}$, it follows from Theorem 3.1 of \cite{R74} that there exists $R_1(z)\in\mathbb{C}_p(z)\cap E_{0,p}$ such that $R_1(x)\neq0$ for all $x\in D_0$ and $\left|\frac{f'(z)}{f(z)}-\frac{R_1'(z)}{R_1(z)}\right|_{\mathcal{G}}<1/p$.
Now, let us suppose that there are $R_1(z),\ldots, R_n(z)\in\mathbb{C}_p(z)$ such that, for all $1\leq j\leq n$, $R_j(x)\neq0$ for all $x\in D_0$ and that Equation~\eqref{eq_dl} is satisfied. Now, we put $g(z)=\frac{f(z)}{R_1(z)R_2(z)\cdots R_n(z)}$. Then  $$\frac{g'(z)}{g(z)}=\frac{f'(z)}{f(z)}-\sum_{i=1}^n\frac{R_i'(z)}{R_i(z)}.$$ Notice that $g(x)$ converges for all $x\in D_0$ and that $g'(z)/g(z)$ belongs to $E_{0,p}$. Then, by applying  Theorem 3.1 of \cite{R74} to $g(z)$, there exists $R_{n+1}(z)\in\mathbb{C}_p(z)\cap E_{0,p}$ such that, $R_{n+1}(x)\neq0$ for all $x\in D_0$ and $\left|\frac{g'(z)}{g(z)}-\frac{R_{n+1}'(z)}{R_{n+1}(z)}\right|_{\mathcal{G}}<1/p^{n+1}$. In particular, we obtain $$\left| \frac{f'(z)}{f(z)}-\sum_{i=1}^{n+1}\frac{R_i'(z)}{R_i(z)}\right|_{\mathcal{G}}<1/p^{n+1}.$$
So, by induction we conclude that there is a sequence  $\{R_{n}(z)\}_{n\geq1}$ of rational functions with coefficients in $\mathbb{C}_p$ such that, for all $n$, Equation~\eqref{eq_dl} is satisfied. 

We now put $Q_j(z)=\sum_{i=1}^j\frac{R'_i(z)}{R_i(z)}$. Then, it follows from \eqref{eq_dl} that $f'(z)/f(z)=\lim_{j\rightarrow\infty}Q_j(z)$. Further, $Q_j(z)$ is a rational function of the shape $\sum\frac{n_i}{z-\alpha_i}$, where the $n_i$'s are nonzero integers. Write $$Q_j(z)=\sum_{\alpha_i\in D_{\alpha}}\frac{n_i}{z-\alpha_i}+\sum_{\alpha_i\notin D_{\alpha}}\frac{n_i}{z-\alpha_i}.$$ 
%Here, we use the fact that if $\alpha_i\in D_{\alpha}$ then $$\frac{1}{z-\alpha_i}=\frac{1}{z-\alpha}\sum_{l\geq0}\left(\frac{\alpha_i-\alpha}{z-\alpha}\right)^l. $$
By (ii) of Lemma~\ref{lem_conv}, we deduce that $\sum_{\alpha_i\notin D_{\alpha}}\frac{n_i}{z-\alpha_i}\in\mathbb{C}_p[[z-\alpha]]$.
Thus, from Remark~\ref{rem_pol_rational}, we deduce that $\mathrm{Res}(Q_j(z),D_{\alpha})=\sum_{\alpha_i\in D_{\alpha}}n_i\in\mathbb{Z}$. So, $\mathrm{Res}(f'/f,D_{\alpha})=\lim_{j\rightarrow\infty}\mathrm{Res}(Q_j(z),D_{\alpha})$ and since, for every $j$, $\mathrm{Res}(Q_j(z),D_{\alpha})$ is an integer, we get that $\mathrm{Res}(f'/f, D_{\alpha})$ belongs to $\mathbb{Z}_p$.

(ii) Since $\left|\frac{f'(z)}{f(z)}-\frac{R'(z)}{R(z)}\right|_{\mathcal{G}}\leq 1/p^n$. It is clear that $\frac{R'(z)}{R(z)}=\sum_{i=1}^s\frac{n_i}{z-\beta_i}$ and thus, $$\frac{R'(z)}{R(z)}=\sum_{\beta_i\in D_{\alpha}}\frac{n_i}{z-\beta_i}+\sum_{\beta_i\notin D_{\alpha}}\frac{n_i}{z-\beta_i}.$$
Since $f'/f\in E_{0,p}$ and is bounded, by Mittag-Lefflet's Theorem, we know that there exists a unique sequence $\{\alpha_n\}_{n\geq0}$ in $\mathcal{O}_{\mathbb{C}_p}$ such that, for every $n$, $|\alpha_n|=1$, $|\alpha_n-\alpha_m|=1$ for every $n\neq m$, and there exists a unique sequence ${\omega_n}_{n\geq0}$ in $E_{0,K}$ such that $\omega_0\in E(\mathcal{O}_{\mathbb{C}_p})$ and, for every $n>0$, $\omega_n\in E^{\alpha_n}$, and  $\frac{f'}{f}=\sum_{n\geq0}\omega_n$

We have $\alpha\in D_{\alpha_n}$ for some $n>0$ because $D_{\alpha}$ is a pole of $f'(z)/f(z)$. So, the singular term of $\frac{f'(z)}{f(z)}-\frac{R'(z)}{R(z)}$ associated twith$D_{\alpha_n}$ is $\omega_n-\sum_{\beta_i\in D_{\alpha}}\frac{n_i}{z-\beta_i}$ and since, by assumption, $\left|\frac{f'(z)}{f(z)}-\frac{R'(z)}{R(z)}\right|_{\mathcal{G}}\leq 1/p^n$, Mittag-Lefflet's Theorem implies $\left|\omega_n-\sum_{\beta_i\in D_{\alpha}}\frac{n_i}{z-\beta_i}\right|_{\mathcal{G}}\leq 1/p^n$.

Since $\omega_n\in E^{\alpha_n}$, it follows from Lemma~\ref{lemm_good_def} that $\omega_n=\lim\limits_{j\rightarrow\infty}R_j(z)$ where $R_j(z)=\sum_{s\in\mathbb{Z}}a_{s,j}(z-\alpha)^s$. So, $\omega_n=\sum_{s\in\mathbb{Z}}b_s(z-\alpha)^s$, where $b_s=\lim\limits_{j\rightarrow\infty}a_{s,j}$. Notice that, by definition $res\left(\frac{f'}{f}, D_{\alpha}\right)=b_{-1}$.

In addition, it follows from Lemma~\ref{lem_conv} that $$\sum_{\beta_i\in D_{\alpha}}\frac{n_i}{z-\beta_i}=\sum_{\beta_{i}\in D_{\alpha}}\frac{n_i}{z-\alpha}\sum_{s\geq0}\left(\frac{\beta_i-\alpha}{z-\alpha}\right)^s=\sum_{s\geq0}\frac{\sum_{\beta_i\in D_{\alpha}}n_i(\beta_{i}-\alpha)^s}{(z-\alpha)^{s+1}}.$$
Given that $\left|\omega_n-\sum_{\beta_i\in D_{\alpha}}\frac{n_i}{z-\beta_i}\right|_{\mathcal{G}}\leq 1/p^n$, it follows that $|b_{-1}-\sum_{\beta_i\in D_{\alpha}}n_i|<1/p^n$. Therefore, we obtain $$\mathrm{Res}(f'(z)/f(z),D_{\alpha})\equiv\left(\sum_{\beta_i\in D_{\alpha}}n_i\right)\bmod p^n\mathbb{Z}_p.$$  

%we deduce that $$\frac{f'(z)}{f(z)}\equiv\frac{R'(z)}{R(z)}\bmod\pi_L^n\mathcal{O}_{L}[[z]],$$ 
%where $L=K(\beta_1,\ldots,\beta_s)$ and $\pi_L$ is a uniformizer of $L$. 

%t is clear that $\frac{R'(z)}{R(z)}=\sum_{i=1}^s\frac{n_i}{z-\beta_i}$. So, Remark~\ref{rem_pol_rational} implies that $\mathrm{Res}(R'(z)/R(z),D_{\alpha})=\sum_{\beta_i\in D_{\alpha}}n_i$. Thus, we have $\mathrm{Res}(f'(z)/f(z), D_{\alpha})\equiv \mathrm{Res}(R'(z)/R(z),D_{\alpha})\bmod p^n$. Therefore, $$\mathrm{Res}(f'(z)/f(z),D_{\alpha})=\left(\sum_{\beta_i\in D_{\alpha}}n_i\right)\bmod p^n.$$
\end{proof}

\subsection{Exponents and Regular Singular Points}

It is well known that if $\alpha \in \mathbb{C}_p$ is a regular singular point of a differential operator $\mathcal{L} \in \mathbb{Q}(z)[d/dz]$, and if $f(z) = (z-\alpha)^{\mu} \sum_{n \ge 0} a_n(z-\alpha)^n \in \mathbb{C}[[z-\alpha]]$ with $a_0 \neq 0$ is a solution of $\mathcal{L}(f) = 0$, then $\mu$ is an exponent of $\mathcal{L}$ at $\alpha$. The following lemma shows that a similar property holds when reducing modulo $p^a$ for an integer $a > 0$.

\begin{lemm}\label{lem_exponents}
Let $K$ be a finite extension of $\mathbb{Q}_p$ and let
\[
\mathcal{L} = \frac{d^n}{dz^n} + b_1(z)\frac{d^{n-1}}{dz^{n-1}} + \cdots + b_{n-1}(z)\frac{d}{dz} + b_n(z) \in \mathcal{O}_{K(z)}[d/dz].
\]
Let $\alpha \in \overline{\mathbb{Q}}_p$ with norm $|\alpha| = 1$, let $\mathbf{L} = K(\alpha)$, and let
\[
h(z) = (z-\alpha)^{\mu} \sum_{j \ge 0} a_j(z-\alpha)^j \in (z-\alpha)^{\mu} \mathcal{O}_{\mathbf{L}}[[z-\alpha]],
\]
with $\mu \in \mathbb{Z}_{(p)}$ and $|a_0| = 1$. Suppose that $\alpha$ is a regular singular point of $\mathcal{L}_{\mid\mathscr{M}_K^a}$ for some integer $a \ge 0$. If $h(z) \bmod{\pi_{\mathbf{L}}^a}$ is a solution of $\mathcal{L}_{\mid\mathscr{M}_{\mathbf{L}}^a}$, then $\mu \bmod{\pi_{\mathbf{L}}^a}$ is an exponent of $\mathcal{L}_{\mid\mathscr{M}_K^a}$ at $\alpha$, where $\pi_{\mathbf{L}}$ denotes a uniformizer of $\mathbf{L}$.
\end{lemm}

\begin{proof}
By assumption, for all $1 \le j \le n$, we have $(z-\alpha)^j b_j(z) \equiv F_j(z) \bmod{\mathscr{M}_{\mathbf{L}}^a \mathcal{O}_{\mathbf{L}(z)}}$, where $F_j(z) \in \mathbf{L}(z) \cap \mathcal{O}_{\mathbf{L}}[[z-\alpha]]$. Notice that $z-\alpha$ is a unit in $\mathcal{O}_{\mathbf{L}(z)}$ because $|\alpha| = 1$. Thus, $b_j(z) \equiv \frac{F_j(z)}{(z-\alpha)^j} \bmod{\mathscr{M}_{\mathbf{L}}^a \mathcal{O}_{\mathbf{L}(z)}}$.

Since $h(z) \bmod{\pi_{\mathbf{L}}^a}$ is a solution of $\mathcal{L}_{\mid\mathscr{M}_{\mathbf{L}}^a}$, we have
\begin{equation}\label{eq_reduction_m}
h^{(n)}(z) + \frac{F_1(z)}{z-\alpha} h^{(n-1)}(z) + \cdots + \frac{F_n(z)}{(z-\alpha)^n} h(z) \equiv 0 \bmod{\pi_{\mathbf{L}}^a \mathcal{O}_{\mathbf{L}}\left[\left[z-\alpha, (z-\alpha)^{\mu}, \frac{1}{z-\alpha}\right]\right]}.
\end{equation}
Let $\overline{x}$ denote the reduction of $x \in \mathcal{O}_{\mathbf{L}}$ modulo $\pi_{\mathbf{L}}^a$. Reducing $h(z)$ modulo $\pi_{\mathbf{L}}^a$ yields
\[
h(z) \equiv \overline{a_0}(z-\overline{\alpha})^{\mu} + \overline{a_1}(z-\overline{\alpha})^{\mu+1} + \cdots + \overline{a}_j(z-\overline{\alpha})^{\mu+j} + \cdots \bmod{\pi_{\mathbf{L}}^a}.
\]
Evaluating the leading-term coefficient in Equation~\eqref{eq_reduction_m} gives
\[
a_0 \Big[ \mu(\mu-1)\cdots(\mu-n+1) + s_1 \mu(\mu-1)\cdots(\mu-n+2) + \cdots + s_{n-1}\mu + s_n \Big] \equiv 0 \bmod{\pi_{\mathbf{L}}^a},
\]
where $s_j = F_j(\alpha) \bmod{\pi_{\mathbf{L}}^a}$. Since $|a_0| = 1$, the reduction $a_0 \bmod{\pi_{\mathbf{L}}^a}$ is a unit in $\mathcal{O}_{\mathbf{L}}/\pi_{\mathbf{L}}^a$. Therefore, $\mu \bmod{\pi_{\mathbf{L}}^a}$ is a root of the indicial polynomial $P_{\alpha,a}(X)$ associated with $\mathcal{L}_{\mid\mathscr{M}_K^a}$ at $\alpha$.
\end{proof}
\section{A refined version of Theorem~\ref{theo_alg_ind}}\label{sec_rifi}

The purpose of this section is to prove Theorem~\ref{theo_power}, which provides more detailed information than Theorem~\ref{theo_alg_ind}. Furthermore, this theorem is crucial in the proof of Theorem~\ref{theo_criterion_alg_ind}. Recall that, for any set $A \subset \mathbb{C}_p$, $E(A)$ denotes the completion of the ring of rational functions in $\mathbb{C}_p(z)$ that do not have poles in $A$.

\begin{theo}\label{theo_power}
Let $K$ be a Frobenius field and let $f_1(z),\ldots, f_m(z) \in \mathcal{M}\mathcal{F}(K)$. Suppose that:
\begin{enumerate}
\item[(1)] for every $1 \le i \le m$, there exists $\alpha_{c_i} \in \overline{\mathbb{Q}}_p$ with norm $1$ such that $D_{\alpha_{c_i}} \in \mathbf{Pol}(f'_i(z)/f_i(z))$, $D_{\alpha_{c_i}} \notin \mathbf{Pol}(f'_j(z)/f_j(z))$ for all $j \neq i$, and $\mathrm{Res}\left(\frac{f'_i(z)}{f_i(z)}, D_{\alpha_{c_i}}\right) \neq 0$;

\item[(2)] for every $1 \le i \le m$, the Gauss norm of the principal part of $f'_i(z)/f_i(z)$ is less than $|\pi_K|^a$, where $\pi_K$ is a uniformizer of $\mathcal{O}_K$ and $a \in \mathbb{Z}_{>0}$.
\end{enumerate}
If $f_1(z),\ldots, f_m(z)$ are algebraically dependent over $E_K$, then there exist integers $d_1,\ldots, d_m \in \mathbb{Z}$ not all zero, a finite extension $\mathbf{L}$ of $K$, a rational function $\mathfrak{a}_1(z) \in \mathbf{L}(z)$, and a power series $\mathfrak{j}_1(z) \in \mathcal{O}_{\mathbf{L}}[[z]]$ such that:
\begin{enumerate}[label=(\roman*)]
\item 
\begin{equation}\label{eq_factor}
f^{d_1}_1(z)\cdots f^{d_m}_m(z) = \prod_{i=1}^m (z-\alpha_{c_i})^{\Gamma_i} \mathfrak{a}_1(z) \mathfrak{j}_1(z),
\end{equation}
where $\Gamma_i = d_i \cdot \mathrm{Res}\left(\frac{f'_i(z)}{f_i(z)}, D_{\alpha_{c_i}}\right) \in \mathbb{Z}^*$ for all $1 \le i \le m$.
\end{enumerate}
Moreover:
\begin{enumerate}
\item[(ii)] there exists $d_s \in \{d_1,\ldots, d_m\}$ such that $p$ does not divide $d_s$;
\item[(iii)] $|\mathfrak{a}_1(z)|_{\mathcal{G}} = 1$, $\mathfrak{a}_1(z) \in E(D_{\alpha_{c_i}})$ and $\mathfrak{a}_1(\alpha_{c_i}) \neq 0$ for every $1 \le i \le m$, and all poles and zeros of $\mathfrak{a}_1(z)$ have norm $1$;
\item[(iv)] $|\mathfrak{j}_1(z)|_{\mathcal{G}} = 1 = |\mathfrak{j}_1(0)|$ and $\mathfrak{j}_1'(z) \in \pi_{\mathbf{L}}^a \mathcal{O}_{\mathbf{L}}[[z]]$, where $\pi_{\mathbf{L}}$ is a uniformizer of $\mathcal{O}_{\mathbf{L}}$;
\item[(v)] there exists an integer $N > 0$ such that $(f^{d_1}_1(z)\cdots f^{d_m}_m(z))^N \in E_{0,K}$.
\end{enumerate}
\end{theo}

The proof of this theorem is derived from Theorem~\ref{theo_key}. This theorem is a direct consequence of Theorems~59.8 and 59.11 of \cite{E95}. To state it properly, we need to recall the notion of a \emph{strongly copiercing sequence associated to $D_0$}. Following \cite[p. 160]{E95}, we say that an infinite sequence $(a_n,b_n)_{n\geq0}$ is a strongly copiercing sequence associated to $D_0$ if, for all $n\geq0$, $a_n,b_n\in\mathbb{C}_p\setminus D_0$, $|a_n-b_n|<\delta(a_n, D_0)$\footnote{$\delta(a,D_0)=\inf\{|a-x| : x\in D_0\}$.} and $\lim_{n\rightarrow\infty}\frac{|a_n-b_n|}{\delta(a_n,D_0)}=0$.

Next, we prove some properties of the strongly copiercing sequence associated to $D_0$.

\begin{lemm}\label{rem_copiercing}
If $(a_n,b_n)_{n\geq0}$ is a strongly copiercing sequence associated to $D_0$ then $\mathfrak{t}(z)=\prod_{i=0}^{\infty}\left(\frac{z-a_n}{z-b_n}\right)$ belongs to $E_{0,p}$, $|\mathfrak{t}(z)-1|_{\mathcal{G}}<1$, and $|\mathfrak{t}(0)|=1$. 
\end{lemm}
\begin{proof}
For every $r\geq1$, we consider $\mathfrak{t}_r(z)=\prod_{n=1}^r\left(\frac{z-a_n}{z-b_n}\right)$. Then $$\mathfrak{t}_r(z)-\mathfrak{t}_{r+1}(z)=\mathfrak{t}_r(z)\left(1-\frac{z-a_{r+1}}{z-b_{r+1}}\right).$$
 Since, for all $n\geq1$, $|a_n|\geq1$ and the norm is non\nobreakdash-Archimedean, we get that $\delta(a_n,D_0)=|a_n|$ and for this reason $|a_n-b_n|<|a_n|$ for all $n\geq1$. So, using the fact that norm is non\nobreakdash-Archimedean again, we get that $|b_n|=|a_n|$ for all $n\geq1$. 
 
 Hence, for every $r\geq1$, $|\mathfrak{t}_r(z)|_{\mathcal{G}}=1$, and 
 \begin{equation}\label{eq_norme}
 \left|1-\frac{z-a_{r}}{z-b_{r}}\right|_{\mathcal{G}}=\frac{|a_{r}-b_{r}|}{|b_{r}|}=\frac{|a_{r}-b_{r}|}{\delta(a_{r},D_0)}<1.
 \end{equation}
 As $\lim_{n\rightarrow\infty}\frac{|a_n-b_n|}{\delta(a_n,D_0)}=0$ then $\lim_{r\rightarrow\infty}\left|1-\frac{z-a_{r}}{z-b_{r}}\right|_{\mathcal{G}}=0$. Whence, $\mathfrak{t}(z)$ is the limit of the rational functions $\mathfrak{t}_r(z)$. Note that this rational function does have poles in the $D_0$ because $b_n\notin D_0$ for all $n\geq1$. 
 
 Further, it follows from Equation~\eqref{eq_norme} that, for all $n\geq1$, $\frac{z-a_n}{z-b_n}\equiv1\bmod\mathscr{M}\mathcal{O}_{\mathbb{C}_p(z)}$, where $\mathscr{M}$ is the maximal ideal of $\mathcal{O}_{\mathbb{C}_p(z)}$. Thus, for every $r\geq1$, $\mathfrak{t}_r(z)\equiv1\bmod\mathscr{M}\mathcal{O}_{\mathbb{C}_p(z)}$, and $\mathfrak{t}(z)\equiv1\bmod\mathscr{M}\mathcal{O}_{\mathbb{C}_p(z)}$. Consequently, $|\mathfrak{t}(z)-1|_{\mathcal{G}}<1$. Finally, by Equation~\eqref{eq_norme} again, we get that, for ever $r\geq1$, $\frac{a_{r}}{b_{r}}\equiv1\bmod\mathfrak{m}\mathcal{O}_{\mathbb{C}_p}$ and hence $\mathfrak{t}_r(0)\equiv1\bmod\mathfrak{m}\mathcal{O}_{\mathbb{C}_p}$ for every $r\geq1$. So $\mathfrak{t}(0)\equiv1\bmod\mathfrak{m}\mathcal{O}_{\mathbb{C}_p}$. Whence, $|\mathfrak{t}(0)|=1$.
\end{proof}
We are now ready to state Theorem~\ref{theo_key}.
\begin{theo}[Theorems 59.8 and 59.11~\cite{E95}]\label{theo_key}
 Let $K$ be a finite extension of $\mathbb{Q}_p$, and let $f(z)\in 1+z\mathcal{O}_K[[z]]$ such that $f'(z)/f(z)\in E_{0,K}$.  If $f(z)\in E_{0,K}$ then
 \begin{equation}\label{eq_factorr}
 f(z)=\prod_{i=1}^q(z-\theta_i)^{\mu_i}\mathfrak{t}(z),
 \end{equation}
where,  $|\theta_i|=1$, $|\theta_i-\theta_j|=1$ for $j\neq i$, $\mu_i\in\mathbb{Z}^{*}$ for all $1\leq i\leq q$, and $\mathfrak{t}(z)=\prod_{i=1}^{\infty}\left(\frac{z-a_n}{z-b_n}\right)$ with $(a_n,b_n)_{n\geq1}$ a strongly copiercing sequence associated to $D_0$. 

Moreover:
\begin{enumerate}
\item [(i)]$\mu_i=\mathrm{Res}\left(\frac{f'(z)}{f(z)}, D_{\theta_i}\right)\in\mathbb{Z}^{*}$ for all $1\leq i\leq q$,
\item [(ii)]$\mathrm{Res}\left(\frac{f'(z)}{f(z)}, D_{\theta}\right)=0$ for any $\theta$ with norm $1$ and $D_{\theta}\notin\{D_{\theta_1},\ldots, D_{\theta_q}\}$,
\item [(iii)]$\frac{f'(z)}{f(z)}-\sum_{i=1}^q\frac{\mu_i}{z-\theta_i}$ is integrable in $E_{0,p}$. That is, there exists $G\in E_{0,p}$ such that $G'(z)=\frac{f'(z)}{f(z)}-\sum_{i=1}^q\frac{\mu_i}{z-\theta_i}$.
\end{enumerate}

%\item If $\frac{f'(z)}{f(z)}-\sum_{i=1}^q\frac{\mu_i}{z-\tau_i}$ is integrable, with $\tau_i\in D_{\alpha_i}$ then $f(z)=\prod_{i=1}^q(z-\tau_i)^{\mu_i}\mathfrak{j}(z)$, wehere  $\mathfrak{j}(z)\in E_{p}$ and $|\mathfrak{j}(z)-1|_{\mathcal{G}}<1$.
 
\end{theo}

\begin{rema}
We would like to point out that Theorem~59.8 of \cite{E95} is stated in a slightly different way. Let us be more precise. In \cite{E95}, the author considers $f'/f\in E(D)$, where $D$ is a well pierced set\footnote{Following \cite[p.106]{E95}, we say that a set $D\subset\mathbb{C}_p$ is well pierced if $\delta(\overline{D},\mathbb{C}_p\setminus D)>0$, where  $\overline{D}$ is the topological closure of $D$ and $\delta(\overline{D},\mathbb{C}_p\setminus D):=\inf\{|b-a|: b\in\mathbb{C}_p\setminus D, a\in\overline{D}\}$. For example, $D_0$ is well pierced because $\overline{D_0}=D_0$ and $\delta(D_0,\mathbb{C}_p\setminus D_0)=1$.}. Then, he shows that $f(z)$ is invertible in $E(D)$ if and only if $f(z)$ satisfies a decomposition of the shape \eqref{eq_factorr}. Thus, Theorem~\ref{theo_key} can be viewed as a specialization of this result for $D=D_0$, since $D_0$ is a well pierced set. Moreover, if $f(z)\in 1+z\mathcal{O}_K[[z]]\cap E_{0,K}$, then $f(z)$ is an invertible in $E_{0,K}$, and hence also invertible in $E(D_0)$, because $E(D_0)=E_{0,p}$. Finally, the fact that $f'(z)/f(z)-\sum_{i=1}^q\frac{\mu_i}{z-\beta_i}$ is integrable in $E_{0,p}$ is a direct consequence of Theorem~59.8 and Theorem~59.11 of \cite{E95}.
\end{rema}

\subsection{$a$-rooth of a power series}

In this section, we recall the definition of a $a$-th root of a power series with coefficients in $\mathbb{C}_p$.

\begin{defi}[$a$-th root of a power series]
Let $f(z)=\sum_{j\geq0}a_j(z-\alpha)^j\in\mathbb{C}_p[[z-\alpha]]$ with $\alpha\in\mathbb{C}_p$, and let $a>0$ be an integer. Let us suppose that $a_0\neq0$. Since $\mathbb{C}_p$ is algebraically closed, the polynomial $X^{a}-a_0$ has all its roots in $\mathbb{C}_p$. We denote by $a_{0}^{1/a}$ one of these roots. So, the roots of $X^{a}-a_0$ are $a_{0}^{1/a}\exp(\frac{2\pi i k}{a})$ with $k=0,1,\ldots, a-1$. We then define an $a$\nobreakdash-th root of $f(z)$ as follows $$f(z)^{1/a}:=a_0^{1/a}\left(\sum_{s\geq0}\binom{1/a}{s}\frac{1}{a_0^s}T^s\right),\text{ where }\binom{1/a}{s}=\frac{(1/a)(1/a-1)\cdots(1/a-s+1)}{s!}$$
and $T(z)=\sum_{j\geq1}a_j(z-\alpha)^j$. 

Then $f(z)^{1/a}=\sum_{j\geq0}b_j(z-\alpha)^j\in\mathbb{C}_p[[z-\alpha]]$, where $b_0=a_0^{1/a}$ and for any $j>0$, $$b_j\in \mathbb{Z}\left[a_1,\ldots,a_j,\binom{1/a}{1},\ldots,\binom{1/a}{j},\frac{1}{a_0},a_0^{1/a}\right].$$
\end{defi}

Next, we prove some properties concerning the $a$-th root of a power series, which are very useful in the rest of the paper.

 \begin{lemma}\label{lemm_rational} Suppose that $K$ is a finite extension of $\mathbb{Q}_p$.
 
 (i) Let $f(z)=\sum_{n\geq0}a_n(z-\alpha)^n\in\mathcal{O}_{K}[[z-\alpha]]$ with $|a_0|=1$. If $a$ is an integer not divisible by $p$ then $$f(z)^{1/a}=\sum_{n\geq0}b_n(z-\alpha)^n\in\mathcal{O}_{\bm{L}}[[z-\alpha]]\quad\text{ and }|b_0|=1,$$
 where $\bm{L}$ is a finite extension of $K$ containing all the roots of $X^{a}-a_0$.
 
 (ii) Let $P(z) \in K(z)$ have Gauss norm $|P(z)|_{\mathcal{G}} = 1$. If $P'(z)/P(z) \in E(D_{\alpha})$ with $\alpha \in \mathcal{O}_{\mathbb{C}_p}$ of norm $|\alpha| = 1$, then $P(z) = \sum_{n \ge 0} a_n (z-\alpha)^n \in \mathcal{O}_{\mathbf{L}}[[z-\alpha]]$ with $|a_0| = 1$, where $\mathbf{L}$ is a finite extension of $K$.

Moreover, if $p \nmid a$, then
\[
P(z)^{1/a} = \sum_{n \ge 0} b_n (z-\alpha)^n \in \mathcal{O}_{\mathbf{F}}[[z-\alpha]] \quad \text{and} \quad |b_0| = 1,
\]
where $\mathbf{F}$ is the finite extension of $\mathbf{L}$ generated by all roots of $X^a - P(\alpha)$.

% Let $P(z), Q(z)$ be in $\mathbb{C}_p[z]$ such that $|P(z)/Q(z)|_{\mathcal{G}}=1$ and let $\alpha$ be in $\mathbb{C}_p$ such that $|\alpha|=1$. If $Q(z)$ doe not have pole in $D_{\alpha}$ and  $|P(\alpha)/Q(\alpha)|=1$ then $$\frac{P(z)}{Q(z)}=\sum_{n\geq0}a_n(z-\alpha)^n,$$
 %where, for all $n\geq0$, $|a_n|\leq 1$ with $a_0=P(\alpha)/Q(\alpha)$. Moreover, if $a$ is an integer such that $p$ does not divide $a$ then, for all integers $k\geq1$ $$\left(\frac{P(z)}{Q(z)}\right)^{1/a}=\sum_{n\geq0}b_n(z-\alpha)^n,$$
% where, for all $n\geq0$, $|b_n|\leq1$ with $|b_0|=1$.
 \end{lemma}
 
 \begin{proof}
(i)  By definition, 
 $$f(z)^{1/a}=a_{0}^{1/a}\left(\sum_{s\geq0}\binom{1/a}{s}\frac{1}{a_0^s}T^s\right)=\sum_{j\geq0}b_j(z-\alpha)^j,$$
 where $b_0=a_0^{1/a}$ and for $j>0$, $$b_j\in \mathbb{Z}\left[a_1,\ldots,a_j,\binom{1/a}{1},\ldots,\binom{1/a}{j},\frac{1}{a_0},a_0^{1/a}\right].$$
 Thus, $b_j\in\bm{L}$ for all $j\geq0$. Finally, since $p$ does not divide $a$,  $\binom{1/a}{s}\in\mathbb{Z}_p$  and we know that $|a_n|\leq1$ for any $n\geq0$. Thus, for all integers $n$, $|b_n|\leq 1$.  We then obtain that $b_j\in\mathcal{O}_{\bm{L}}$ for all $j\geq0$. In addition, $|b_0|=1$ as $b_0=a_0^{1/a}$ and $|a_0|=1$.
 
 (ii) Suppose that $P(z) = R(z)/S(z)$, where $R(z), S(z) \in K[z]$. Since $|P(z)|_{\mathcal{G}} = 1$, without loss of generality we can assume that $|R(z)|_{\mathcal{G}} = 1 = |S(z)|_{\mathcal{G}}$. Since $P'(z)/P(z) \in E(D_{\alpha})$, it follows from \cite[Proposition~1]{Motzkin} that $P(z) \in E(D_{\alpha})$. In particular, no root of $S(z)$ lies in $D_{\alpha}$, and since $|S(z)|_{\mathcal{G}} = 1$, we have $\overline{S}(\overline{\alpha}) \neq 0$. Thus, $|S(\alpha)| = 1$. 

Now, we observe that
\[
\frac{P'(z)}{P(z)} = \frac{R'(z)}{R(z)} - \frac{S'(z)}{S(z)}.
\]
Because no root of $S(z)$ lies in $D_{\alpha}$, we have $\frac{S'(z)}{S(z)} \in E(D_{\alpha})$. Since $P'(z)/P(z) \in E(D_{\alpha})$ by assumption, it follows that $\frac{R'(z)}{R(z)} \in E(D_{\alpha})$. Hence, no root of $R(z)$ belongs to $D_{\alpha}$. In particular, $\overline{R}(\overline{\alpha}) \neq 0$, which yields $|R(\alpha)| = 1$, and consequently $|P(\alpha)| = 1$.

Now, let $\mathbf{L}$ be the finite extension of $K$ generated by the coefficients of $P(z)$ and $\alpha$. Since $P(z) \in E(D_{\alpha})$, we can expand $P(z)$ in its Taylor series around $\alpha$ as $P(z) = \sum_{n \ge 0} a_n (z-\alpha)^n \in \mathbf{L}[[z-\alpha]]$. It then follows from \cite[Chapter~IV, p.~128]{Dworkgfunciones} that $a_n \in \mathcal{O}_{\mathbf{L}}$ for all $n \ge 0$. In particular, $|a_0| = 1$ because $a_0 = P(\alpha)$ and $|P(\alpha)| = 1$.

Finally, from item~(i) we deduce that $P(z)^{1/a} = \sum_{n \ge 0} b_n (z-\alpha)^n \in \mathcal{O}_{\mathbf{F}}[[z-\alpha]]$, where $\mathbf{F}$ is a finite extension of $\mathbf{L}$ containing all roots of $X^a - P(\alpha)$, with $|b_0| = 1$.
 
 %that Let us write $P(z)=\frac{\prod_{i=1}^n(z-\alpha_i)^{\mu_i}}{\prod_{j=1}^m(z-\beta_j)^{\nu_j}}$, where $\mu_1,\ldots,\mu_n,\nu_1,\ldots,\nu_m$ belong to $\mathbb{N}_{>0}$. Let $\bm{L}=K(\alpha,\alpha_1,\ldots,\alpha_n,\beta_1,\ldots,\beta_m)$. By assumption, $\beta_1,\ldots,\beta_m\in\mathcal{O}_{\mathbb{C}_p}$ and $P(z)\in E(D_{\alpha})$. Then, Remark~\ref{rem_discos} implies that, for every $1\leq j\leq m$, $|\beta_j-\alpha|=1$. Similarly, by assumption, $\alpha_1,\ldots, \alpha_n\in\mathcal{O}_{\mathbb{C}_p}$ and $P(\alpha)\neq0$ and thus, Remark~\ref{rem_discos} implies that, for every $1\leq i\leq n$, $|\alpha_i-\alpha|=1$. Consequently, $|P(\alpha)|=1$. 
 
 %Now, it is clear that, for all $(i,j)\in\{1,\ldots,n\}\times\{1,\ldots,m\}$, $$(z-\alpha_i)^{\mu_i}=\sum_{s=0}^{\mu_i}\binom{\mu_i}{s}(\alpha-\alpha_i)^s(z-\alpha)^{\mu_i-s}\in\mathcal{O}_{\bm{L}}[z-\alpha]$$ $$(z-\beta_j)^{\nu_j}=\sum_{s=0}^{\nu_j}\binom{\nu_j}{s}(\alpha-\beta_j)^s(z-\alpha)^{\nu_j-s}\in\mathcal{O}_{\bm{L}}[z-\alpha].$$ Since $|\alpha-\beta_j|=1$, we have $|\alpha-\beta_j|^{\nu_j}=1$. So, for all $1\leq j\leq m$, $(z-\beta_j)^{\nu_j}$ is a unit of the ring $\mathcal{O}_L[[z-\alpha]]$ and thus,  $\frac{1}{(z-\beta_j)^{\nu_j}}\in\mathcal{O}_{\bm{L}}[[z-\alpha]]$. Consequently, $P(z)=\sum_{n\geq0}a_n(z-\alpha)^n\in\mathcal{O}_{\bm{L}}[[z-\alpha]]$ and $|a_0|=1$ because $P(\alpha)=a_0$ and we have already seen that $|P(\alpha)|=1$. 

  \end{proof}
 From the previous proof, we have the following remark
 \begin{rema}\label{rem_a_p}
 Let $f(z)=\sum_{n\geq0}a_nz^n$ be in $\mathcal{O}_K[[z]]$, $K$ an extension of $\mathbb{Q}_p$, such that $|f|_{\mathcal{G}}=1=|a_0|$. It follows from the proof of previous lemma that if $p$ does not divide $a$ then $f(z)^{1/a}=\sum_{n\geq0}b_nz^n\in \mathcal{O}_L[[z]]$ and $|f^{1/a}|_{\mathcal{G}}=1=|b_0|$, where $L$ is a finite extension of $K$ containing all the roots of $X^a-a_0$.
 \end{rema}
 
 \begin{lemm}\label{lem_mod}
Let $K$ be a finite extension of $\mathbb{Q}_p$. Suppose that $f(z)\in 1+z\mathcal{O}_K[[z]]$ such that $f'/f\in E_{0,K}$ and $f(z)/f(z^{p^h})\in E_{0,K}$ for some integer of $h>0$. Then, for any integer $a>0$ there exists $P_{a}\in K_0(z)$ with $|P_a|_{\mathcal{G}}=1$ such that $$f(z)\equiv P_a(z)f(z^{p^{ah}})\bmod\pi^a_K\mathcal{O}_K[[z]].$$
Moreover, for any $D_{\alpha}\notin\bm{Pol}(f'/f)$, $P'_a(z)/P_a(z)$ belongs to $E(D_{\alpha})$.

And, if $d>0$ is an integer not divisible by $p$ then $$f(z)^{1/d}\equiv (P_a(z)f(z^{p^{ah}}))^{1/d}\bmod\pi_{K}^a\mathcal{O}_{K}[[z]].$$
 \end{lemm}
 
 \begin{proof}
 We set $H=f(z)/f(z^{p^{h}})$. It is clear that $$\frac{H'}{H}=\frac{f'(z)}{f(z)}-p^{h}z^{p^{h}-1}\frac{f'(z^{p^{h}})}{f(z^{p^{h}})}.$$

So, $H'/H\in E_{0,K}$ because, by assumption, $f'/f\in E_{0,K}$. In addition, if $D_{\alpha}\notin\bm{Pol}(f'/f)$ then it is clear that $D_{\alpha}\notin\bm{Pol}(H'/H)$ and thus, from Proposition~1 of \cite{Motzkin}, $H$ belongs to $E(D_{\alpha})$. Thus, for any integer $a>0$, there exists $B_{a}\in K_0(z)$ such that $$H\equiv B_{a}(z)\bmod\pi^a_K\mathcal{O}_K[[z]]$$
and,  for any $D_{\alpha}\notin\bm{Pol}(f'/f)$, $B_{a}$ belongs to $E(D_{\alpha})$. Further, since  $H$ has Gauss norm $1$, the previous congruence implies that  $B_{a}$ has Gauss norm $1$.

%Further, since the principal part $H$ is zero, any pole of $B_{a}(z)$ has norm $1$ and $B_{a}(z)\in E(D_{\alpha})$ for all $D_{\alpha}\notin\bm{Pol}(f'/f)$. 
Notice that $$\frac{f(z)}{f(z^{p^{ah}})}=H(z)H(z^{p^{h}})\cdots H(z^{p^{(a-1)h}})\equiv B_{a}(z)B_{a}(z^{p^{h}})\cdots B_{a}(z^{p^{(a-1)h}})\bmod\pi^a_K\mathcal{O}_K[[z]].$$
We put $P_{a}:=B_{a}(z)B_{a}(z^{p^{h}})\cdots B_{a}(z^{p^{(a-1)h}})$. Hence, $P_{a}(z)$ has Gauss norm $1$ and belongs to $E(D_{\alpha})$ for all $D_{\alpha}\notin\bm{Pol}(f'/f)$, and $$f(z)\equiv P_{a}(z)f(z^{p^{ah}})\bmod\pi^a_K\mathcal{O}_K[[z]].$$
Clearly, we have $$\frac{f'(z)}{f(z)}\equiv\frac{P'_{a}(z)}{P_{a}(z)}\bmod\pi^a_K\mathcal{O}_K[[z]].$$
Thus, if $D_{\alpha}\notin\bm{Pol}(f'/f)$ then, from the previous congruence, we have $D_{\alpha}\notin\bm{Pol}(P'_{a}(z)/P_{a}(z))$. That is, $P'_{a}(z)/P_{a}(z)$ belongs to $E(D_{\alpha})$. 

We now show that $$f(z)^{\frac{1}{d}}\equiv(P_{a}(z)f(z^{p^{ah}}))^{\frac{1}{d}}\bmod\pi^a_K\mathcal{O}_K[[z]].$$

We know that $f(z)\equiv P_{a}(z)f(z^{p^{ah}})\bmod\pi_K^{a}\mathcal{O}_K[[z]]$. Thus, there exists  $g_{a}(z)\in\mathcal{O}_K[[z]]$ such that $f(z)=P_{a}(z)f(z^{p^{ah}})+\pi_K^ag_{a}(z)$. So, it is clear that $f(z)=P_{a}(z)f(z^{p^{ah}})\left(1+\pi_K^a\frac{g_{a}(z)}{P_{a}(z)f(z^{p^{ah}})}\right)$. Whence $$f(z)^{\frac{1}{d}}=(P_{a}(z)f(z^{p^{ah}}))^{\frac{1}{d}}\left(1+\pi_K^a\frac{g_{a}(z)}{P_{a}(z)f(z^{p^{ah}})}\right)^{\frac{1}{d}}.$$  
 We have \[\left(1+\pi_K^a\frac{g_{a}(z)}{P_{a}(z)f(z^{p^{ah}})}\right)^{\frac{1}{d}}=\sum_{j\geq0}\binom{1/d}{j}\left(\pi_K^a\frac{g_a(z)}{P_{a}(z)f(z^{p^{ah}})}\right)^j.\]
 Given that $|P_{a}(z)f(z^{p^{ah}})|_{\mathcal{G}}=1$ and that $g_{a}\in\mathcal{O}_K[[z]]$, we have $\left|\frac{g_{a}(z)}{P_{a}(z)f(z^{p^{ah}})}\right|_{\mathcal{G}}\leq1$. In addition, for all $j\geq0$, $\binom{1/d}{j}\in\mathbb{Z}_p$ because $p$ does not divide $d$. Therefore, $$\left(1+\pi_K^a\frac{g_a(z)}{P_{a}(z)f(z^{p^{ah}})}\right)^{\frac{1}{d}}\equiv1\bmod\pi_{\bm{F}}^a\mathcal{O}_{\textbf{F}}[[z]].$$
 
Hence, we deduce that $f(z)^{\frac{1}{d}}=(P_{a}(z)f(z^{p^{ah}}))^{\frac{1}{d}}\bmod\pi_{K}^a\mathcal{O}_{K}[[z]]$.

 \end{proof}

We also need the following remark.
\begin{rema}\label{rem_k_0}
Let $K$ be an extension of $\mathbb{Q}_p$ and let $P(z)$ be $K_0(z)$. If $|P(z)|_{\mathcal{G}}\leq1$ then $P(z)\in\mathcal{O}_K[[z]]$. Indeed, since $0$ is not a pole of $P(z)$,  $P(z)=\sum_{n\geq0}a_nz^n$ with $a_n\in K$ for all $n\geq0$. But, $|a_n|\leq 1$ for all $n\geq0$ because  $|P(z)|_{\mathcal{G}}\leq1$ and thus, $1\geq\sup\{|a_n|\}_{n\geq0}$.
\end{rema}
\subsection{Key lemma}

The next lemma is derived from Theorems~\ref{theo_alg_ind} and \ref{theo_key} and is the main ingredient in the proof of Theorem~\ref{theo_power}.

\begin{lemm}\label{lem_inicio}
Let $K$ be a Frobenius field and let $f_1(z),\ldots, f_m(z)$ be in $\mathcal{M}\mathcal{F}(K)$. Suppose that, for every $1\leq i\leq m$, there exists $\alpha_{c_i}\in\overline{\mathbb{Q}_p}$ with norm $1$ such that $D_{\alpha_{c_i}}\in\textbf{Pol}(f'_i(z)/f_i(z))$, $D_{\alpha_{c_i}}\notin\textbf{Pol}(f'_j(z)/f_j(z))$ for all $j\neq i$, and $\mathrm{Res}\left(\frac{f'_i(z)}{f_i(z)}, D_{\alpha_{c_i}}\right)\neq0$.

If $f_1(z),\ldots, f_m(z)$ are algebraically dependent over $E_K$ then there exist $a_1,\ldots, a_m\in\mathbb{Z}$, not all zero, such that $h(z)=f_1(z)^{a_1}\cdots f_m(z)^{a_m}\in E_{0,K}$ and
\begin{equation*}
h(z)=\mathfrak{q}(z)\mathfrak{a}(z)\mathfrak{j}(z)\text{ with }\mathfrak{q}(z)=\prod_{i=1}^m(z-\alpha_{c_i})^{\mu_i},\text{ }\mathfrak{a}(z)=\prod_{i=m+1}^q(z-\beta_i)^{\mu_i}\text{ and }\mathfrak{j}(z)\in\mathcal{O}_L[[z]],
\end{equation*}
where for all $1\leq i\leq q$, $\mu_i\in\mathbb{Z}^{*}$, and $L$ is a finite extension of $K$.  Moreover:
\begin{enumerate}
%\item [(i)] for all $1\leq i\leq m$, $\mu_i=res\left(\frac{f'_i(z)}{f_i(z)}, D_{\alpha_{c_i}}\right)$,
\item [(i)] for any $m+1\leq i\leq q$, $\beta_i\in\overline{\mathbb{Q}_p}$, $|\beta_i|=1$, $|\beta_i-\beta_j|=1$ for $j\neq i$, and $|\alpha_{c_i}-\beta_j|=1$ for any $1\leq i\leq m$ and, any $m+1\leq j\leq q$,\smallskip

\item [(ii)] $\mu_i=\mathrm{Res}(h'/h, D_{\alpha_{c_i}})$ for all $1\leq i\leq m$, and $\mu_i=\mathrm{Res}(h'/h, D_{\beta_i})$ for all $m+1\leq i\leq q$,\smallskip

\item [(iii)]  $\mathrm{Res}(h'/h, D_{\beta})=0$ for any $\beta$ with norm $1$ and $D_{\beta}\notin\{D_{\alpha_{c_1}},\ldots,D_{\alpha_{c_m}}, D_{\beta_{m+1}},\ldots, D_{\beta_q}\}$,\smallskip

\item [(iv)] $\mathfrak{j}(z)=\mathfrak{o}(z)^{1/p^e}$ where $e\in\mathbb{Z}_{\geq0}$ and $\mathfrak{o}(z)=\prod_{n=1}^{\infty}\left(\frac{z-c_n}{z-d_n}\right)$ with $(c_n,d_n)_{n\geq1}$ a strongly copiercing sequence associated to $D_0$,\smallskip

\item [(v)]  $\mathfrak{j}(z)\equiv1\bmod\pi_{L}\mathcal{O}_{L}[[z]]$, where $\pi_L$ is a uniformizer of $\mathcal{O}_L$.
\end{enumerate}
\end{lemm}

\begin{proof}
  By Theorem~\ref{theo_alg_ind}, there exist $a_1,\ldots, a_m\in\mathbb{Z}$, not all zero, such that $$h(z):=f_1(z)^{a_1}\cdots f_m(z)^{a_m}\in E_{0,K}.$$
Notice that $h(z)\in 1+z\mathcal{O}_K[[z]]$ because, by (i) of Theorem~\ref{theo_analytic_element}, for all $1\leq i\leq m$, $f_i(z)\in 1+z\mathcal{O}_K[[z]]$. By (iii) of Theorem~\ref{theo_analytic_element}, we also know that, for all $1\leq i\leq m$, $f'_i(z)/f_i(z)\in E_{0,K}$ and thus, $$\frac{h'(z)}{h(z)}=\sum_{i=1}^ma_i\frac{f'_i(z)}{f_i(z)}\in E_{0,K}.$$
Moreover, $h'(z)/h(z)$ is bounded because, for all $1\leq i\leq m$, $f_i(z)\in 1+z\mathcal{O}_K[[z]]$. 

Then, from Theorem~\ref{theo_key}, we get 
 \begin{equation}\label{eq_h_deco1}
 h(z)=\prod_{i=1}^q(z-\theta_i)^{\mu_i}\mathfrak{t}(z),
 \end{equation}
where $|\theta_i|=1$ for all $1\leq i\leq q$, $|\theta_i-\theta_j|=1$ for $j\neq i$, and $\mathfrak{t}(z)=\prod_{i=0}^{\infty}\left(\frac{z-a_n}{z-b_n}\right)$ with $(a_n,b_n)_{n\geq0}$ a strongly copiercing sequence associated to $D_0$.  

By Lemma~\ref{rem_copiercing}, we know that $\mathfrak{t}(z)$ belongs to $E_{0,p}$, $|\mathfrak{t}(z)-1|_{\mathcal{G}}<1$, and $|\mathfrak{t}(0)|=1$.

Furthermore, it follows from (i) and (iii) of Theorem~\ref{theo_key} that:
\begin{enumerate}
\item $\text{ with }\mu_i=\mathrm{Res}\left(\frac{h'}{h},D_{\theta_i}\right)\in\mathbb{Z}^{*}$ for all $1\leq i\leq q$,
\item  $\mathrm{Res}\left(\frac{h'}{h}, D_{\theta}\right)=0$ for any $\theta$ with norm 1 and $D_{\theta}\notin\{D_{\theta_1},\ldots,D_{\theta_q}\}$.
\end{enumerate}

%Let $D_{\alpha_1},\ldots, D_{\alpha_q}$ be the poles of $h'(z)/h(z)$ such that $res\left(\frac{h'(z)}{h(z)},D_{\alpha_i}\right)\neq0$ for all $1\leq i\leq q$. Notice $D_0$ is not a pole of $h'(z)/h(z)$ because $h'(z)/h(z)\in E_{0,K}$. Then, by Equation~\eqref{eq_mittag}, we conclude that, for all $1\leq i\leq q$, $|\alpha_i|=1$.

% Given that $h(z)$ and $h'(z)/h(z)$ belong to $E_{0,K}$, from Theorem~\ref{theo_key}, we get that there exist $\beta_1,\ldots, \beta_q$ such that $\beta_i\in D_{\alpha_i}$ and 
% \begin{equation}\label{eq_h_deco1}
% h(z)=\prod_{i=1}^q(z-\beta_i)^{\mu_i}\mathfrak{t}(z)\text{ with }\mu_i=res\left(\frac{h'}{h},D_{\alpha_i}\right)\in\mathbb{Z},
% \end{equation}
%$\mathfrak{t}(z)=\prod_{i=0}^{\infty}\left(\frac{z-a_n}{z-b_n}\right)$ with $(a_n,b_n)_{n\geq0}$ a strongly copiercing sequence associated to $D_0$. 

%Since $\beta_i\in D_{\alpha_i}$ and $|\alpha_{i}|=1$ for all $1\leq i\leq q$, it follows from Remark~\ref{rem_discos} that $|\beta_i|=1$ for all $1\leq i\leq q$.
Now, by assumption, for every $1\leq i\leq m$, we have $D_{\alpha_{c_i}}\in\textbf{Pol}(f'_i(z)/f_i(z))$ and $D_{\alpha_{c_i}}\notin\textbf{Pol}(f'_j(z)/f_j(z))$ for every $j\neq i$. Thus, it follows that $D_{\alpha_{c_i}}$ is a pole of $h'(z)/h(z)$ and $$\mathrm{Res}\left(\frac{h'(z)}{h(z)},D_{\alpha_{c_i}}\right)=a_i\cdot\mathrm{Res}\left(\frac{f'_i(z)}{f_i(z)},D_{\alpha_{c_i}}\right)$$
for every $1\leq i\leq m$. 

By assumption, we know that, for all $i\in\{1,\ldots, m\}$, $\mathrm{Res}\left(\frac{f'_i(z)}{f_i(z)},D_{\alpha_{c_i}}\right)\neq0$. Further, we know that there are  $j\in\{1,\ldots, m\}$ such $a_j\neq0$. So, for these indices $j$, we have $\mathrm{Res}\left(\frac{h'(z)}{h(z)},D_{\alpha_{c_j}}\right)\neq0$. Thus, without loss of generality, we assume that $\mathrm{Res}\left(\frac{h'(z)}{h(z)},D_{\alpha_{c_i}}\right)\neq0$ for all $1\leq i\leq m$. Thus, it follows from (1) and (2) that $D_{\alpha_{c_1}},\ldots, D_{\alpha_{c_m}}\in\{D_{\theta_1},\ldots, D_{\theta_q}\}$. Hence, we can also assume that $D_{\theta_1}=D_{\alpha_{c_1}},\ldots, D_{\theta_m}=D_{\alpha_{c_m}}$.

%Now, for every $1\leq i\leq q$,  $|\alpha_i|=1$. Moreover, since $\alpha_{c_i}\in D_{\alpha_i}$ for every $1\leq i\leq m$, Remark~\ref{rem_discos} implies that $|\alpha_{c_i}-\alpha_i|<1$. Then, using the non-Archimedean property, we deduce that $|\alpha_{c_i}|=1$. Similarly, we prove that $|\beta_i|=1$ for all $1\leq i\leq q$.
 
 We therefore have $D_{\theta_i}=D_{\alpha_{c_{i}}}$ for all $1\leq i\leq m$. As $\mathbb{C}_p$ is the completion of of $\overline{\mathbb{Q}_p}$ then, for every $m+1\leq i\leq q$, there exists $\beta_i\in\overline{\mathbb{Q}_p}$ such that $|\theta_i-\beta_i|<1$. In addition, $|\beta_i|=1$ for every $m+1\leq i\leq q$ because $|\theta_i|=1$. So, by Remark~\ref{rem_discos}, $D_{\theta_i}=D_{\beta_i}$ for all $m+1\leq i\leq q$. Thanks to part Part (iii) of Theorem~\ref{theo_key}, $\frac{h'(z)}{h(z)}-\sum_{i=1}^q\frac{\mu_i}{z-\theta_i}$ is integrable in $E_{0,p}$ and thus, according to Lemma 59.2 of \cite{E95}, $\frac{h'(z)}{h(z)}-\sum_{i=1}^m\frac{\mu_i}{z-\alpha_{c_i}}-\sum_{i=m+1}^q\frac{\mu_i}{z-\beta_i}$ is also integrable in $E_{0,p}$. Then, by Theorem~59.11 of \cite{E95}, there exits an integer $e\geq0$ such that 
 \begin{equation}\label{eq_deco_h'_h}
 \frac{h'(z)}{h(z)}-\sum_{i=1}^m\frac{\mu_i}{z-\alpha_{c_i}}-\sum_{i=m+1}^q\frac{\mu_i}{z-\beta_i}=\frac{1}{p^e}\left(\sum_{i=1}^{\infty}\frac{1}{z-c_n}-\frac{1}{z-d_n}\right),
 \end{equation}
 where $(c_n,d_n)_{n\geq1}$ is a strongly copiercing sequence associated to $D_0$. So, by Lema~\ref{rem_copiercing}, we have $\mathfrak{o}(z)=\prod_{n=1}^{\infty}\left(\frac{z-c_n}{z-d_n}\right)\in E_{0,p}$ and $|\mathfrak{o}(0)|=1$. 
 
Consequently, Equation~\eqref{eq_deco_h'_h} implies that there exists  $\lambda\in\mathbb{C}_p\setminus\{0\}$ such that
\begin{equation}\label{eq_h_deco22}
h(z)=\mathfrak{q}(z)\mathfrak{a}(z)\mathfrak{j}(z)\text{ with }\mathfrak{q}(z)=\prod_{i=1}^m(z-\alpha_{c_i})^{\mu_i},\text{ }\mathfrak{a}(z)=\prod_{i=m+1}^q(z-\beta_i)^{\mu_i}\text{ and }\mathfrak{j}(z)=\lambda(\mathfrak{o}(z))^{1/p^e},
\end{equation}
We want to prove that the above decomposition is the desired decomposition. 

First, thanks to the construction, the conditions (i)-(iv) are already satisfied. Finally, we need to check that $\mathfrak{j}(z)\in\mathcal{O}_L[[z]]$ ( with $L$ a finite extension of $K$) and that the condition (v) is satisfied.
%We proceed to show that $|\lambda|=1$. We know that, for all $1\leq i\leq q$, $|\alpha_i|=1$ and given that $D_{\alpha_1}=D_{\alpha_{c,1}},\ldots, D_{\alpha_m}=D_{\alpha_{c,m}}$ and $\beta_{m+1}\in D_{\alpha_m},\ldots, \beta_q\in D_{\alpha_q}$, Remark~\ref{rem_discos} implies that, for all $1\leq i\leq m$, $|\alpha_{c_i}|=1$ and that, for all $m+1\leq i\leq q$, $|\beta_i|=1$. Thus, $|\mathfrak{q}(0)|=1$ and $|\mathfrak{a}(0)|=1$. As $h(z)\in 1+z\mathcal{O}_K[[z]]$ then Equation~\eqref{eq_h_deco2} gives $1=\lambda\mathfrak{q}(0)\mathfrak{a}(0)\mathfrak{o}(0)^{1/p^e}$ and $|\mathfrak{o}(0)^{1/p^e}|=1$ because $|\mathfrak{o}(0)|=1$. As a consequence, we then have $|\lambda|=1$.  

Let $L$ be the field $K(\alpha_{c_1},\ldots,\alpha_{c_m},\beta_1,\ldots,\beta_q)$. We proceed to show that $\mathfrak{j}(z)$ belongs to $\mathcal{O}_{L}[[z]]$. Indeed, we know that $h(z)\in 1+z\mathcal{O}_K[[z]]$ and it is clear that $\mathfrak{q}(z)\mathfrak{a}(z)\in L(z)$. However, we know that $|\alpha_{c_i}|=1$ for any $1\leq i\leq m$ and $|\beta_i|=1$ for any $m+1\leq i\leq q$. Thus, we get that $|\mathfrak{q}(z)\mathfrak{a}(z)|_{\mathcal{G}}=1$ and that $\mathfrak{q}(z)\mathfrak{a}(z)\in E(D_0)$. Hence, by Remark~\ref{rem_k_0}, $\mathfrak{q}(z)\mathfrak{a}(z)\in\mathcal{O}_{L}[[z]]$ and since $|\mathfrak{q}(0)\mathfrak{a}(0)|=|\prod_{i=1}^m(-\alpha_{c_i})^{\mu_i}||\prod_{i=m+1}^q(-\beta_{i})^{\mu_i}|=1$, we get that $\mathfrak{q}(z)\mathfrak{a}(z)$ is a unit of $\mathcal{O}_{L}[[z]]$.  As consequence, it follows from \eqref{eq_h_deco22} that $\mathfrak{j}(z)$ belongs to $\mathcal{O}_{L}[[z]]$.
 
We finally show that $\mathfrak{j}(z)\equiv1\bmod\pi_{L}\mathcal{O}_{L}[[z]]$. We prove, by using Equation~\eqref{eq_h_deco1}, in a similar way as above that $\mathfrak{t}(z)$ belongs to $\mathcal{O}_{\mathbb{C}_p}[[z]]$. Since $|\mathfrak{t}(z)-1|_{\mathcal{G}}<1$, we have $\mathfrak{t}(z)\equiv1\bmod\mathfrak{m}_{\mathbb{C}_p}[[z]]$, where $\mathfrak{m}_{\mathbb{C}_p}$ is the maximal ideal of $\mathcal{O}_{\mathbb{C}_p}$. Thus, by reducing Equation~\eqref{eq_h_deco1} modulo $\mathfrak{m}_{\mathbb{C}_p}$ we get $h(z)\equiv\prod_{i=1}^q(z-\theta_i)^{\mu_i}\bmod\mathfrak{m}_{\mathbb{C}_p}[[z]]$. In addition, for all $1\leq i\leq m$, $|\theta_i-\alpha_{c_i}|<1$ and, for all $m+1\leq i\leq q$, $|\theta_i-\beta_i|<1$ and thus, $\prod_{i=1}^q(z-\theta_i)^{\mu_i}\equiv\prod_{i=1}^m(z-\alpha_{c_i})^{\mu_i}\prod_{i=m+1}^q(z-\beta_i)^{\mu_i}\bmod\mathfrak{m}_{\mathbb{C}_p}[[z]]$. So, by reducing Equation~\eqref{eq_h_deco22} modulo $\mathfrak{m}_{\mathbb{C}_p}$ we conclude that $\mathfrak{j}(z)\equiv1\bmod \mathfrak{m}_{\mathbb{C}_p} [[z]]$. But, we already know that $\mathfrak{j}(z)$ belongs to $\mathcal{O}_{L}[[z]]$ and hence, $\mathfrak{j}(z)\equiv1\bmod\pi_{L}\mathcal{O}_{L}[[z]]$

 \end{proof}
 
 \begin{lemm}\label{lem_zeros_unit}
 Let $K$ be a finite extension of $\mathbb{Q}_p$, and let $f(z)\in\mathcal{O}_K[[z]]$ be a unit such that $f'(z)/f(z)\in E_{0,K}$. If the Gauss norm of the principal term of $f'(z)/f(z)$ is less than $|\pi_K|^a$ (with $a\in\mathbb{Z}_{>0}$) then there exists $P(z)\in\mathbb{C}_p(z)$ such that $P(z)$, $P'(z)/P(z)$ belong to $\mathcal{O}_{\mathbb{C}_p}[[z]]$ and $$|f'(z)/f(z)-P'(z)/P(z)|_{\mathcal{G}}<|\pi_K|^a.$$
 Moreover, the zeros and poles of $P(z)$ have norm equal to $1$.
 \end{lemm}

\begin{proof}
As $f(z)\in\mathcal{O}_K[[z]]$ is a unit then $f'(z)/f(z)\in\mathcal{O}_K[[z]]$, and since $f'(z)/f(z)\in E_{0,K}$, it follows from Theorem~3.1 of \cite{R74} that there exists $R(z)\in\mathbb{C}_p(z)$ such that $R(z)$ has no zeros or poles in $D_0$ and $\left|\frac{f'(z)}{f(z)}-\frac{R'(z)}{R(z)}\right|_{\mathcal{G}}<|\pi_K|^a$.  We now show that $|R'(z)/R(z)|_{\mathcal{G}}\leq 1$. Indeed, we have 
$$\left|\frac{R'(z)}{R(z)}\right|_{\mathcal{G}}=\left|\frac{R'(z)}{R(z)}-\frac{f'(z)}{f(z)}+\frac{f'(z)}{f(z)}\right|_{\mathcal{G}}\leq\max\left\{\left|\frac{f'(z)}{f(z)}-\frac{R'(z)}{R(z)}\right|_{\mathcal{G}}, \left|\frac{f'(z)}{f(z)}\right|_{\mathcal{G}}\right\}\leq1.$$

Furthermore, $R'(z)/R(z)\in E(D_0)$ because $R(z)$ has no zeros or poles in $D_0$. Then, from Remark~\ref{rem_k_0}, we conclude that $R'(z)/R(z)\in\mathcal{O}_{\mathbb{C}_p}[[z]]$. 

%We know that $\pi_K\mathcal{O}_K\subset\pi_F\mathcal{O}_F$ and therefore, we have $\left|\frac{f'(z)}{f(z)}-\frac{R'(z)}{R(z)}\right|_{\mathcal{G}}<|\pi_K|^a\leq|\pi_F|^a$. Given that $f'(z)/f(z)-R'(z)/R(z)\in\mathcal{O}_F[[z]]$, we then deduce that that $$\frac{f'(z)}{f(z)}\equiv\frac{R'(z)}{R(z)}\bmod\pi_F^a\mathcal{O}_{F}[[z]].$$
%So, there exists  $\mathfrak{u}\in\mathcal{O}_F[[z]]$ such that $$\frac{f'(z)}{f(z)}=\frac{R'(z)}{R(z)}+\pi^a_F\mathfrak{u}.$$

%Notice that $\mathfrak{u}\in E_{0,F}$ because $R'(z)/R(z)\in F_0(z)$, $f'(z)/f(z)\in E_{0,K}$, and $E_{0,K}\subset E_{0,F}$. In addition, $\mathfrak{u}$ is bounded since $\mathfrak{u}\in\mathcal{O}_F[[z]]$. Then, by Mittag-Lefflet's Theorem, it makes to consider the principal part of $\mathfrak{u}$, which is denoted by $\mathfrak{u}_0$.

It is clear that $\frac{R'(z)}{R(z)}=\sum_{\tau\in S}\frac{c_{\tau}}{z-\tau},$ where $c_{\tau}\in\mathbb{Z}^*$ for all $\tau\in S$, where $S$ is the set of poles and zero of $R(z)$. So,
 $$\frac{R'(z)}{R(z)}=\sum_{\tau\in S, |\tau|=1}\frac{c_{\tau}}{z-\tau}+\sum_{\tau\in S, |\tau|>1}\frac{c_{\tau}}{z-\tau}.$$
On the one hand, if $|\tau|>1$ then $\frac{1}{z-\tau}\in E(\mathcal{O}_{\mathbb{C}_p})$. On the other hand, if $|\tau|=1$, then $\frac{1}{z-\tau}\in E^{\tau}$. Therefore, if $\omega_0$ is the principal part of $f'(z)/f(z)$ then the principal part of $f'(z)/f(z)-R'(z)/R(z)$ is $\omega_0-\sum_{\tau\in S, |\tau|>1}\frac{c_{\tau}}{z-\tau}$. Furthermore, thanks to Mittag-Lefflet's Theorem, we have $$\left|\omega_0-\sum_{\tau\in S, |\tau|>1}\frac{c_{\tau}}{z-\tau}\right|_{\mathcal{G}}\leq\left|\frac{f'(z)}{f(z)}-\frac{R'(z)}{R(z)}\right|_{\mathcal{G}}<|\pi_K|^a.$$
By assumption, we know that $|\omega_0|<|\pi_K|^a$. Thus, $$\left|\sum_{\tau\in S, |\tau|>1}\frac{c_{\tau}}{z-\tau}\right|_{\mathcal{G}}<|\pi_K|^a.$$
Taking $P(z)=\prod_{\tau\in S, |\tau|=1}(z-\tau)^{c_{\tau}}$, we obtain that $P(z)$ and $P'(z)/P(z)$ belong to $\mathcal{O}_{\mathbb{C}_p}[[z]]$, and 
\begin{align*}
\left|\frac{f'(z)}{f(z)}-\frac{P'(z)}{P(z)}\right|_{\mathcal{G}}&=\left|\frac{f'(z)}{f(z)}-\frac{R'(z)}{R(z)}+\sum_{\tau\in S, |\tau|>1}\frac{c_{\tau}}{z-\tau}\right|_{\mathcal{G}}\\
&\leq\max\left\{\left|\frac{f'(z)}{f(z)}-\frac{R'(z)}{R(z)}\right|_{\mathcal{G}},\left|\sum_{\tau\in S, |\tau|>1}\frac{c_{\tau}}{z-\tau}\right|_{\mathcal{G}}\right\}\\
&<|\pi_K|^a.
\end{align*}
\end{proof}
 
 We are now ready to prove Theorem~\ref{theo_power}.
 
 \subsection{Proof of Theorem~\ref{theo_power}}

It follows from Lemma~\ref{lem_inicio} that there exist $a_1,\ldots, a_m\in\mathbb{Z}$, not all zero, such that $h(z)=f_1(z)^{a_1}\cdots f_m(z)^{a_m}\in E_{0,K}$ and
\begin{equation}\label{eq_h_deco2}
h(z)=\mathfrak{q}(z)\mathfrak{a}(z)\mathfrak{j}(z)\text{ with }\mathfrak{q}(z)=\prod_{i=1}^m(z-\alpha_{c_i})^{\mu_i},\text{ }\mathfrak{a}(z)=\prod_{i=m+1}^q(z-\beta_i)^{\mu_i}\text{ and }\mathfrak{j}(z)\in\mathcal{O}_L[[z]],
\end{equation}
where for all $1\leq i\leq q$, $\mu_i\in\mathbb{Z}^{*}$, and $L$ is a finite extension of $K$. Moreover:
\begin{enumerate}
\item for all $m+1\leq i\leq q$, $\beta_i\in\overline{\mathbb{Q}_p}$ and $|\beta_i|=1$, $|\beta_i-\beta_j|=1$ for $j\neq i$, and $|\alpha_{c_i}-\beta_j|=1$ for any $1\leq i\leq m$ and, any $m+1\leq j\leq q$,
\item  $\mu_i=\mathrm{Res}(h'/h, D_{\alpha_{c_i}})$ for all $1\leq i\leq m$, and $\mu_i=\mathrm{Res}(h'/h, D_{\beta_i})$ for all $m+1\leq i\leq q$,
\item  $\mathrm{Res}(h'/h, D_{\beta})=0$ for any $\beta$ with norm $1$ and $D_{\beta}\notin\{D_{\alpha_{c_1}},\ldots,D_{\alpha_{c_m}}, D_{\beta_{m+1}},\ldots, D_{\beta_q}\}$,
\item   $\mathfrak{j}(z)\equiv1\bmod\pi_{L}\mathcal{O}_{L}[[z]]$, where $\pi_L$ is a uniformizer of $\mathcal{O}_L$.
\end{enumerate}

 For every $1\leq i\leq m$, we set $D_{\alpha_i}:=D_{\alpha_{c_i}}$ and for all $m+1\leq i\leq q$, we set $D_{\alpha_i}:=D_{\beta_i}$. We put $s=\min\{v_p(a_1),\ldots, v_p(a_m)\}$ and let $\bm{L}$ be the field $L(\zeta_1,\ldots,\zeta_{p^s}, \alpha_{c_1},\ldots,\alpha_{c_m},\beta_1,\ldots,\beta_q)$, where $\zeta_1,\ldots, \zeta_{p^s}$ are the $p^s$\nobreakdash-th roots of unity.

For every $i\in\{1,\ldots, m\}$, we put $d_i=\frac{a_i}{p^s}$. Then $d_i\in\mathbb{Z}$ for all $i\in\{1,\ldots, m\}$ because, by definition, $s=\min\{v_p(a_1),\ldots, v_p(a_m)\}$.  We then put $$h_1(z)=f_1(z)^{d_1}\cdots f_m(z)^{d_r}.$$

Therefore, for all $\alpha\in\mathcal{O}_{\mathbb{C}_p}$, we have $p^s\mathrm{Res}\left(\frac{h'_1(z)}{h_1(z)},D_{\alpha}\right)=\mathrm{Res}\left(\frac{h'(z)}{h(z)},D_{\alpha}\right)$.
%Notice that, for all $1\leq i\leq q$, $$\Gamma_i=\sum_{j=1}^m\frac{a_j}{p^s}\cdot res\left(\frac{f'_i(z)}{f_i(z)},D_{\alpha_{i}}\right)=\sum_{j=1}^md_j\cdot res\left(\frac{f'_j(z)}{f_j(z)},D_{\alpha_{i}}\right)=res\left(\frac{h'_1(z)}{h_1(z)},D_{\alpha_{i}}\right).$$

%By assumption, we know that, for all $1\leq i\leq m$, $D_{\alpha_{c_i}}\in\textbf{Pol}(f'_i(z)/f_i(z))$ and that $D_{\alpha_{c_i}}\notin\textbf{Pol}(f'_j(z)/f_j(z))$ for all $j\neq i$. Then $$\Gamma_i=d_i\cdot res\left(\frac{f'_i(z)}{f_i(z)}, D_{\alpha_{c_i}}\right)$$
%for all $1\leq i\leq m$.

For all $i\in\{1,\ldots, q\}$, we put $\Gamma_i=\frac{\mu_i}{p^s}$. We now proceed to prove that, for any $1\leq i\leq q$,  $\Gamma_i$ belongs to $\mathbb{Z}^{*}$. First of all, for all $1\leq i\leq q$,  we have
\begin{equation}\label{eq_gamma}
\Gamma_i=\sum_{j=1}^m\frac{a_j}{p^s}\cdot\mathrm{Res}\left(\frac{f'_i(z)}{f_i(z)},D_{\alpha_{i}}\right)=\sum_{j=1}^md_j\cdot\mathrm{Res}\left(\frac{f'_j(z)}{f_j(z)},D_{\alpha_{i}}\right)=\mathrm{Res}\left(\frac{h'_1(z)}{h_1(z)},D_{\alpha_{i}}\right).
\end{equation}
By Theorem~\ref{theo_analytic_element}, we know that $f_1(z),\ldots, f_m(z)$ belong to $1+z\mathcal{O}_K[[z]]$ and, by Theorem~\ref{theo_analytic_element} again, we also know that  $f'_1(z)/f_1(z)$, $\ldots, f'_m(z)/f_m(z)$ belong to $E_{0,K}$. So, we are able to apply Lemma~\ref{lem_res} and thus, we conclude that $\mathrm{Res}\left(\frac{f'_j(z)}{f_j(z)},D_{\alpha_{i}}\right)$ belongs to $\mathbb{Z}_p$ for all $(i,j)\in\{1,\ldots,q\}\times\{1,\ldots, m\}$. Then $\Gamma_i\in\mathbb{Z}_p$ because $d_j\in\mathbb{Z}$ for all $1\leq j\leq m$. As $\mu_i/p^s=\Gamma_i$ and $\mu_i\in\mathbb{Z}$, and $\Gamma_i\in\mathbb{Z}_p$ then $p^{s}$ divides $\mu_i$ and thus, we get that $\Gamma_i\in\mathbb{Z}$. Finally, $\Gamma_i\neq0$ for all $1\leq i\leq q$ because $\mu_i\neq0$ for all $1\leq i\leq q$. 

So, Equation~\eqref{eq_h_deco2} gives us
%\[\mu_i=res\left(\sum_{i=1}^ra_i\frac{f'_i(z)}{f_i(z)},D_{\alpha_i}\right)=\sum_{i=1}^rp^{v_p(a_i)}l_i\cdot res\left(\frac{f'_i(z)}{f_i(z)},D_{\alpha_i}\right).\]
\begin{equation}\label{eq_deco_h1}
h_1(z)=\zeta\mathfrak{q}_1(z)\mathfrak{a}_1(z)\mathfrak{j}_1(z)\text{ with }\mathfrak{q}_1(z)=\prod_{i=1}^m(z-\alpha_{c_i})^{\Gamma_i},\text{ }\mathfrak{a}_1(z)=\prod_{i=m+1}^q(z-\beta_i)^{\Gamma_i} \text{ and }\mathfrak{j}_1(z)=(\mathfrak{j}(z))^{1/p^{s}},
\end{equation}
where $\zeta$ is a $p^s$\nobreakdash-th root of unity.

We want to see that the decomposition given in \eqref{eq_deco_h1} is the desired decomposition. To do so, we verify conditions (i)–(v).

\textit{Proof of $(i)$} It is clear that $\mathfrak{q}_1(z)$ and $\mathfrak{a}_1(z)$ belong to $\bm{L}(z)$.  We now prove that $\mathfrak{q}_1(z)$ belongs to $\mathcal{O}_{\bm{L}}[[z]]$ and that it is a unit of this ring. Indeed, we know that $|\alpha_{c_i}|=1$ for any $1\leq i\leq m$ and therefore, $|\mathfrak{q}_1(z)|_{\mathcal{G}}=1=|\mathfrak{q}_1(0)|$. Further, we have $\mathfrak{q}_1(z)\in E(D_0)$ because $|\alpha_{c_i}|=1$ for all $1\leq i\leq m$. So, Remark~\ref{rem_k_0} implies that $\mathfrak{q}_1(z)\in\mathcal{O}_{\bm{L}}[[z]]$ and given that $|\mathfrak{q}_1(0)|=1$ then $\mathfrak{q}_1(z)$ is a unit of $\mathcal{O}_{\bm{L}}[[z]]$. In a similar fashion, we prove that $\mathfrak{a}_1(z)$ belongs to $\mathcal{O}_{\bm{L}}[[z]]$ and that it is a unit of this ring. We also have $h_1(z)\in1+z\mathcal{O}_{\bm{L}}[[z]]$ because, by Theorem~\ref{theo_analytic_element}, $f_1(z),\ldots, f_m(z)\in 1+z\mathcal{O}_K[[z]]$ and for all $1\leq i\leq m$, $d_i\in\mathbb{Z}$. Therefore, from Equation~\eqref{eq_deco_h1}, we obtain $$\mathfrak{j}_1(z)=\frac{h_1(z)}{\zeta\mathfrak{q}_1(z)\mathfrak{a}_1(z)}\in\mathcal{O}_{\bm{L}}[[z]]\text{ and } |\mathfrak{j}_1(z)|_{\mathcal{G}}=1=|\mathfrak{j}_1(0)|.$$
Whence, $\mathfrak{j}_1(z)$ is a unit of  $\mathcal{O}_{\bm{L}}[[z]]$.

 Let $1\leq i\leq m$. By assumption, we know that $D_{\alpha_{c_i}}\notin\bm{Pol}(f'_j/f_j)$ for $j\neq i$. Then, we have $\mathrm{Res}\left(\frac{f'_j(z)}{f_j(z)}, D_{\alpha_{c_i}}\right)=0$ for $j\neq i$. Consequently, it follows from Equation~\eqref{eq_gamma} that $$\Gamma_i=d_i\cdot\mathrm{Res}\left(\frac{f'_i(z)}{f_i(z)}, D_{\alpha_{c_i}}\right).$$
Finally, we already know that $\Gamma_i\in\mathbb{Z}^{*}$.

\textit{Proof of (ii)} It is clear, by construction, that $p$ does not divide $d_s$.

\textit{Proof of (iii)} Notice that $|\mathfrak{a}_1(z)|_{\mathcal{G}}=1$ because for all $m+1\leq j\leq q$, $|\beta_j|=1$. Since the poles and zeros of $\mathfrak{a}_1(z)$ are in $\{\beta_{m+1},\ldots,\beta_q\}$, we get that all the poles and zeros of $\mathfrak{a}_1(z)$ have norm 1. Further, for all $1\leq i\leq m$, $\beta_{m+1},\ldots, \beta_q$ do not belong to $D_{\alpha_{c_i}}$ because we know that $|\beta_j-\alpha_{c_i}|=1$ for any $1\leq i\leq m$ and any $m+1\leq j\leq q$. So, for every $1\leq i\leq m$, $\mathfrak{a}_1(z)\in E(D_{\alpha_{c_i}})$ and $\mathfrak{a}_1(\alpha_{c_i})\neq0$. 

\textit{Proof of (iv)} We have already seen that $|\mathfrak{j}_1(0)|=1$. We now prove that $\mathfrak{j}_1(z)\equiv\mathfrak{j}_1(0)\bmod\pi_{\bm{L}}\mathcal{O}_{\bm{L}}[[z]]$. Write $\mathfrak{j}_1(z)=\sum_{j\geq0}\kappa_jz^{j}$. Then $\mathfrak{j}_1(z)^{p^{s}}\equiv\sum_{j\geq0}\kappa_j^{p^{s}}z^{jp^{s}}\bmod\pi_{\bm{L}}\mathcal{O}_{\bm{L}}[[z]]$. Since $\mathfrak{j}(z)=\mathfrak{j}_1(z)^{p^{s}}$ and, by (4) we know that $\mathfrak{j}(z)\equiv1\bmod\pi_{L}\mathcal{O}_{L}[[z]]$, we obtain $\kappa_j\in\pi_{\bm{L}}\mathcal{O}_{\bm{L}}$ for all $j\geq1$.

Finally, we proceed to prove that $\mathfrak{j}'_1(z)\in\pi_{\bm{L}}^a\mathcal{O}_{\bm{L}}[[z]]$. 

%Since  $\mathfrak{j}_1(z)\equiv\mathfrak{j}_1(0)\bmod\pi_{F}\mathcal{O}_{F}[[z]]$, by reducing Equation~\eqref{eq_deco_h1} modulo $\pi_{F}$, we have $$h_1(z)=\zeta\mathfrak{j}_1(0)\mathfrak{q}_1(z)\mathfrak{a}_1(z)\bmod\pi_{F}\mathcal{O}_{F}[[z]].$$ Notice that $\mathfrak{j}_1(0)$ is a unit of $\mathcal{O}_{F}$ because $|\mathfrak{j}_1(0)|=1$. In addition, we have already seen that $\mathfrak{q}_1(z)$ and $\mathfrak{a}_1(z)$ are units of $\mathcal{O}_{F}[[z]]$ but $h_1(z)$ is also a unit of this ring because $h_1(z)\in 1+z\mathcal{O}_K[[z]]$. So, from the equality $h_1(z)=\zeta\mathfrak{j}_1(0)\mathfrak{q}_1(z)\mathfrak{a}_1(z)\bmod\pi_{F}\mathcal{O}_{F}[[z]]$, we obtain 
%\begin{equation}\label{eq_red_pi_l}
%\frac{h_1'(z)}{h_1(z)}=\frac{\mathfrak{q}_1'(z)}{\mathfrak{q}_1(z)}+\frac{\mathfrak{a}'_1(z)}{\mathfrak{a}_1(z)}\bmod\pi_{F}\mathcal{O}_{F}[[z]]
%\end{equation}
Let us show that $\frac{\mathfrak{j}'_1(z)}{\mathfrak{j}_1(z)}$ belongs to $E_{0,p}$ and the Gauss norm of its principal term is less than $|\pi_{\bm{L}}|^a$. From Equation~\eqref{eq_deco_h1}, we deduce that 
\begin{equation}\label{eq_decom_j_1}
\frac{\mathfrak{j}'_1(z)}{\mathfrak{j}_1(z)}=\frac{h'_1(z)}{h_1(z)}-\frac{\mathfrak{q}'_1(z)}{\mathfrak{q}_1(z)}-\frac{\mathfrak{a}'_1(z)}{\mathfrak{a}_1(z)}.
\end{equation}
Given that $\frac{h'_1(z)}{h_1(z)}=\sum_{i=1}^md_i\frac{f'_i(z)}{f_i(z)}$ and $\frac{f'_i(z)}{f_i(z)}\in E_{0,K}$ for all $1\leq i\leq m$, we have $\frac{h'_1(z)}{h_1(z)}\in E_{0,K}$. Further, $\frac{\mathfrak{q}'_1(z)}{\mathfrak{q}_1(z)}=\sum_{i=1}^m\frac{\Gamma_i}{z-\alpha_{c_i}}$ and since $|\alpha_{c_i}|=1$ for all $1\leq i\leq m$, we get $\frac{\mathfrak{q}'_1(z)}{\mathfrak{q}_1(z)}\in E_{0,p}$. We also have $\frac{\mathfrak{a}'_1(z)}{\mathfrak{a}_1(z)}=\sum_{i=m+1}^q\frac{\Gamma_i}{z-\beta_i}$ and since $|\beta_i|=1$ for all $m+1\leq i\leq q$, we get $\frac{\mathfrak{a}'_1(z)}{\mathfrak{a}_1(z)}\in E_{0,p}$. Whence, $\mathfrak{j}'_1(z)/\mathfrak{j}_1(z)$ belongs to $E_{0,p}$.  

Furthermore, since $|\alpha_{c_i}|=1$ for all $1\leq i\leq m$ and $|\beta_i|=1$ for all $m+1\leq i\leq q$, it follows from \eqref{eq_decom_j_1} that the principal terms of $\frac{\mathfrak{j}'_1(z)}{\mathfrak{j}_1(z)}$ and $\frac{h'_1(z)}{h_1(z)}$ are equal. Given that $\frac{h'_1(z)}{h_1(z)}=\sum_{i=1}^md_i\frac{f'_i(z)}{f_i(z)}$ and, by assumption, the Gauss norm of the principal term of $\frac{f'_i(z)}{f_i(z)}$ is less than $1/p^a$ for every $1\leq i\leq m$, we conclude that the Gauss norm of the principal term of $\frac{h'_1(z)}{h_1(z)}$ is less than $1/p^a$. Thus, the Gauss norm of the principal term of $\frac{\mathfrak{j}'_1(z)}{\mathfrak{j}_1(z)}$ is less than $1/p^a$. %But $|\pi_K|^a\leq|\pi_F|^a$ because $\pi_K\mathcal{O}_K\subset\pi_F\mathcal{O}_F$.

As we have already seen $\mathfrak{j}_1(z)\in\mathcal{O}_{\bm{L}}[[z]]$ is a unit and therefore, we are in a position to apply Lemma~\ref{lem_zeros_unit} and hence, there exists $R(z)\in\mathbb{C}_p(z)\cap\mathcal{O}_{\mathbb{C}_p}[[z]]$ such that the zeros and poles of $R(z)$ have norm equal to $1$ and
\begin{equation}\label{eq_j_1}
\left|\frac{\mathfrak{j}'_1(z)}{\mathfrak{j}_1(z)}-\frac{R'(z)}{R(z)}\right|_{\mathcal{G}}<1/p^a.
\end{equation}

 Let us write $\frac{R'(z)}{R(z)}=\sum_{\tau\in S_1}\frac{c_{\tau}}{z-\tau},$
 where $S_1$ is the set of poles and zero of $R(z)$ and $c_{\tau}\in\mathbb{Z}^*$. 
 
 In addition, $|\tau|=1$ for all $\tau\in S_1$ because we know that the zeros and poles of $R(z)$ have norm equal to $1$.
 
We now proceed to prove that $\left|\frac{R'(z)}{R(z)}\right|_{\mathcal{G}}\leq1/p^a$. Clearly, we have $$\frac{R'(z)}{R(z)}=\sum_{\tau\in S_1\cap D_{\alpha_1}}\frac{c_{\tau}}{z-\tau}+\cdots+\sum_{\tau\in S_1\cap D_{\alpha_q}}\frac{c_{\tau}}{z-\tau}+\sum_{\substack{\tau\in S_1 \\ D_{\tau}\notin\{D_{\alpha_1},\ldots, D_{\alpha_q}\}}}\frac{c_{\tau}}{z-\tau}.$$
Let $S_1'=\{\tau\in S_1: D_{\tau}\notin\{D_{\alpha_1},\ldots, D_{\alpha_q}\}\}$. For any $\tau,\tau'\in S'_1$, we say that $\tau\sim\tau'$ if and only if $D_{\tau}=D_{\tau'}$. It is clear that $\sim$ is an equivalence relation on $S'_1$. If $T_1,\ldots, T_r$ is the partition given by this equivalence relation then $$\sum_{\substack{\tau\in S_1 \\ D_{\tau}\notin\{D_{\alpha_1},\ldots, D_{\alpha_q}\}}}\frac{c_{\tau}}{z-\tau}=\sum_{\tau\in T_1}\frac{c_{\tau}}{z-\tau}+\cdots+\sum_{\tau\in T_r}\frac{c_{\tau}}{z-\tau}.$$
By definition, $D_{\alpha_j}=D_{\alpha_{c_j}}$ for any $1\leq j\leq m$. Then, Lemma~\ref{lem_conv} implies that, for any $\tau\in D_{\alpha_j}$ with $1\leq j\leq m$, we have $$\frac{1}{z-\tau}=\frac{1}{z-\alpha_{c_j}}\sum_{l\geq0}\left(\frac{\tau-\alpha_{c_j}}{z-\alpha_{c_j}}\right)^l.$$
Again, by definition, $D_{\alpha_j}=D_{\beta_j}$ for any $m+1\leq j\leq q$. Then, Lemma~\ref{lem_conv} implies that, for any $\tau\in D_{\alpha_j}$ with $m+1\leq j\leq q$, we have $$\frac{1}{z-\tau}=\frac{1}{z-\beta_j}\sum_{l\geq0}\left(\frac{\tau-\beta_j}{z-\beta_j}\right)^l.$$

Now, if $\lambda_j$ is in $T_j$ then, for any $\tau\in T_j$, $|\tau-\lambda_j|<1$ because $D_{\tau}=D_{\theta_j}$. Hence, by Lemma~\ref{lem_conv}, we obtain that, for any $\tau\in T_j$, $$\frac{1}{z-\tau}=\frac{1}{z-\lambda_j}\sum_{l\geq0}\left(\frac{\tau-\lambda_j}{z-\lambda_j}\right)^l.$$
Thus,
\begin{align}\label{eq_r_2}
\begin{split}
\frac{R'(z)}{R(z)}=&\frac{\sum_{\tau\in S_1\cap D_{\alpha_1}}c_{\tau}}{z-\alpha_{c_1}}\sum_{l\geq0}\left(\frac{\tau-\alpha_{c_1}}{z-\alpha_{c_1}}\right)^l+\cdots+\frac{\sum_{\tau\in S_1\cap D_{\alpha_m}}c_{\tau}}{z-\alpha_{c_m}}\sum_{l\geq0}\left(\frac{\tau-\alpha_{c_m}}{z-\alpha_{c_m}}\right)^l\\
&+\frac{\sum_{\tau\in S_1\cap D_{\alpha_{m+1}}}c_{\tau}}{z-\beta_m}\sum_{l\geq0}\left(\frac{\tau-\beta_m}{z-\beta_m}\right)^l+\cdots+\frac{\sum_{\tau\in S_1\cap D_{\alpha_{q}}}c_{\tau}}{z-\beta_q}\sum_{l\geq0}\left(\frac{\tau-\beta_q}{z-\beta_q}\right)^l\\
&+\frac{\sum_{\tau\in T_1}c_{\tau}}{z-\lambda_1}\sum_{l\geq0}\left(\frac{\tau-\lambda_1}{z-\lambda_1}\right)^l+\cdots+\frac{\sum_{\tau\in T_r}c_{\tau}}{z-\lambda_r}\sum_{l\geq0}\left(\frac{\tau-\lambda_r}{z-\lambda_r}\right)^l
\end{split}
\end{align}
%Since $|\alpha_j-\alpha_i|=1$ for any $j\neq i$ and $|\alpha_j-\theta_s|=1$ for any $s$, it follows that 
%Since $|\alpha_j-\tau|<1$, we have $\alpha_j-\tau\in\pi_M\mathcal{O}_M$. So, from the previous equality, we get $$\frac{1}{z-\tau}=\frac{1}{z-\alpha_j}+\frac{\tau-\alpha_j}{(z-\alpha_j)^2}\bmod\pi_M^2\mathcal{O}_M[[z]].$$
 %But $c_{\tau}\in\pi_M\mathcal{O}_M$ because  $\frac{R'_2(z)}{R_2(z)}\in\pi_M\mathcal{O}_{M}[[z]].$ So, from the previous equality, we get 
 %\begin{equation}\label{eq_tau}
 %\frac{c_{\tau}}{z-\tau}=\frac{c_{\tau}}{z-\alpha_j}\bmod\pi_M^2\mathcal{O}_M[[z]].
% \end{equation}
 %So, 
%\begin{equation}\label{eq_r_2}
%\frac{R'_2(z)}{R_2(z)}\equiv\frac{\sum_{\tau\in S_1\cap D_{\alpha_1}}c_{\tau}}{z-\alpha_1}+\cdots+\frac{\sum_{\tau\in S_1\cap D_{\alpha_q}}c_{\tau}}{z-\alpha_q}+\sum_{\substack{\tau\in S_1 \\ D_{\tau}\notin\{D_{\alpha_1},\ldots, D_{\alpha_q}\}}}\frac{c_{\tau}}{z-\tau}\bmod\pi^2_M\mathcal{O}_M[[z]]
%\end{equation}
Put $P(z)=\mathfrak{q}_1(z)\mathfrak{a}_1(z)R(z)$. Then, Equation~\eqref{eq_j_1} implies $$\left|\frac{h'_1(z)}{h_1(z)}-\frac{P'(z)}{P(z)}\right|_{\mathcal{G}}<1/p^a.$$

Since $\frac{\mathfrak{q}'_1(z)}{\mathfrak{q}_1(z)}=\sum_{i=1}^m\frac{\Gamma_i}{z-\alpha_{c_i}}$ and $\frac{\mathfrak{a}'_1(z)}{\mathfrak{a}_1(z)}=\sum_{i=m+1}^q\frac{\Gamma_i}{z-\beta_{i}}$, it follows from Equation~\eqref{eq_r_2} that 
\begin{align}\label{eq_p}
\begin{split}
\frac{P'(z)}{P(z)}&=\frac{\sum_{\tau\in S_1\cap D_{\alpha_1}}c_{\tau}+\Gamma_1}{z-\alpha_{c_1}}+\sum_{l\geq1}\frac{(\tau-\alpha_{c_1})^l}{(z-\alpha_{c_1})^{l+1}}+\cdots+\frac{\sum_{\tau\in S_1\cap D_{\alpha_m}}c_{\tau}+\Gamma_m}{z-\alpha_{c_m}}+\sum_{l\geq1}\frac{(\tau-\alpha_{c_m})^l}{(z-\alpha_{c_m})^{l+1}}\\
&+\frac{\sum_{\tau\in S_1\cap D_{\alpha_{m+1}}}c_{\tau}+\Gamma_{m+1}}{z-\beta_{m+1}}+\sum_{l\geq1}\frac{(\tau-\beta_{m+1})^l}{(z-\beta_{m+1})^{l+1}}+\cdots+\frac{\sum_{\tau\in S_1\cap D_{\alpha_q}}c_{\tau}+\Gamma_q}{z-\beta_q}+\sum_{l\geq1}\frac{(\tau-\beta_q)^l}{(z-\beta_q)^{l+1}}\\
&+\frac{\sum_{\tau\in T_1}c_{\tau}}{z-\lambda_1}\sum_{l\geq0}\left(\frac{\tau-\lambda_1}{z-\lambda_1}\right)^l+\cdots+\frac{\sum_{\tau\in T_r}c_{\tau}}{z-\lambda_r}\sum_{l\geq0}\left(\frac{\tau-\lambda_r}{z-\lambda_r}\right)^l.
\end{split}
\end{align}
Since $D_{\lambda_1},\ldots, D_{\lambda_r}\notin\{D_{\alpha_1},\ldots, D_{\alpha_q}\}$ and $D_{\alpha_j}\neq D_{\alpha_i}$ for $i\neq j$, it follows from Equation~\eqref{eq_p} that $$\mathrm{Ref}\left(\frac{P'(z)}{P(z)}, D_{\alpha_j}\right)=\sum_{\tau\in S_1\cap D_{\alpha_j}}c_{\tau}+\Gamma_j$$
for all $1\leq j\leq q.$

We know that $\left|\frac{h'_1(z)}{h_1(z)}-\frac{P'(z)}{P(z)}\right|_{\mathcal{G}}<1/p^a$ and thus, from (ii) of Lemma~\ref{lem_res}, we deduce that, for every $1\leq j\leq q$, 
 \begin{equation*}
\mathrm{Res}\left(\frac{h'_1(z)}{h_1(z)}, D_{\alpha_j}\right)\equiv(\Gamma_j+\sum_{\tau\in S_1\cap D_{\alpha_j}}c_{\tau})\bmod p^a.
 \end{equation*}
But, by Equation~\eqref{eq_gamma}, $\Gamma_j=\mathrm{Res}\left(\frac{h'_1(z)}{h_1(z)}, D_{\alpha_j}\right)$ for all $1\leq j\leq q$. Therefore, $\sum_{\tau\in S_1\cap D_{\alpha_j}}c_{\tau}\in p^a\mathbb{Z}$ for any $1\leq j\leq q$. 

Since $D_{\lambda_j}\notin\{D_{\alpha_1},\ldots, D_{\alpha_q}\}$ and $D_{\lambda_j}\neq D_{\lambda_i}$ for $i\neq j$, it follows from Equation~\eqref{eq_p} that 
 \begin{equation*}
\mathrm{Res}\left(\frac{P'(z)}{P(z)}, D_{\lambda_j}\right)=\sum_{\tau\in T_j}c_{\tau}
 \end{equation*}
 for all $1\leq j\leq r$.
 
 In addition, given that $\left|\frac{h'_1(z)}{h_1(z)}-\frac{P'(z)}{P(z)}\right|_{\mathcal{G}}<1/p^a$, by (ii) of Lemma~\ref{lem_res}, we deduce that, for every $1\leq j\leq q$, 
 \begin{equation*}
\mathrm{Res}\left(\frac{h'_1(z)}{h_1(z)}, D_{\lambda_j}\right)\equiv\sum_{\tau\in T_j}c_{\tau}\bmod p^a.
 \end{equation*}
It follows from (3) that $\mathrm{Res}\left(\frac{h'_1(z)}{h_1(z)},D_{\lambda_j}\right)=0$ for all $1\leq j\leq r$. So, we obtain $\sum_{\tau\in T_j}c_{\tau}\in p^a\mathbb{Z}$ for all $1\leq j\leq r$.
%But, by using (ii) of Lemma~\ref{lem_res}, we deduce from Equation~\eqref{eq_h_1_red_2} that, for every $1\leq j\leq q$, 

Consequently, $\sum_{\tau\in S_1\cap D_{\alpha_j}}c_{\tau}\in p^a\mathbb{Z}$ for any $1\leq j\leq q$ and $\sum_{\tau\in T_j}c_{\tau}\in p^a\mathbb{Z}$ for all $1\leq j\leq r$. Hence, from Equation~\eqref{eq_r_2} we get $\left|\frac{R'(z)}{R(z)}\right|_{\mathcal{G}}\leq 1/p^a$. Therefore, from Equation~\eqref{eq_j_1}, we obtain $\left|\frac{\mathfrak{j}_1'(z)}{\mathfrak{j}_1(z)}\right|_{\mathcal{G}}\leq 1/p^a$. We know that $\mathfrak{j}_1(z)\in\mathcal{O}_{\bm{L}}[[z]]$ is a unit and hence, $|\mathfrak{j}'_1(z)|_{\mathcal{G}}\leq 1/p^a\leq |\pi_{\bm{L}}|^a$. So, $\mathfrak{j}'_1(z)\in\pi_{\bm{L}}^a\mathcal{O}_{\bm{L}}[[z]]$. 
 
\textit{Proof of (v)} Since, for all $1\leq i\leq r$, $p^sd_i=a_i$, we have $(f^{d_1}_1(z)\cdots f^{d_r}_r(z))^{p^s}=f^{a_1}_1(z)\cdots f^{a_r}_{r}(z)$. But, we know that $f^{a_1}_1(z)\cdots f^{a_r}_{r}(z)\in E_{0,K}$.
 
 $\hfill\square$

   \section{Proof of Main results}\label{sec_proof}

 We are now ready to prove Theorems~\ref{theo_criterion_alg_ind} and \ref{theo_rational_function}, and Corollary~\ref{theo_radius}.
 
 \subsection{Proof of  Theorem~\ref{theo_criterion_alg_ind}}
 
 By assumption (2), we know that, for any $1\leq s\leq m$, $\alpha_{c_s}$ is a regular singular point of $\mathcal{L}_s$. We want to prove that if $\mathrm{Res}\left(\frac{f'_s(z)}{f_s(z)},D_{\alpha_{c_s}}\right)$ is not an exponent of $\mathcal{L}_s$ at $\alpha_{c_s}$ for any $1\leq s\leq r$, then $f_1(z),\ldots, f_m(z)$ are algebraically independent over $E_K$. For this purpose, we proceed by proving its contrapositive. That is, if $f_1(z),\ldots, f_m(z)$ are algebraically dependent over $E_K$ then there exists $1\leq s\leq m$ such that $\mathrm{Res}\left(\frac{f'_s(z)}{f_s(z)},D_{\alpha_{c_s}}\right)$ is an exponent of $\mathcal{L}_{i\mid\mathscr{M}^a_K}$ at $\alpha_{c_i}$.
 
 Let us write $\textbf{Pol}(f'_i(z)/f_i(z))=\{D_{\alpha_{1,i}},\ldots, D_{\alpha_{g_i,i}}\}$. By Theorem~\ref{theo_analytic_element}, we know that $f'_i(z)/f_i(z)\in E_{0,K}$ and thus, by using Equation~\eqref{eq_mittag}, we conclude that $|\alpha_{l,i}|=1$ for every $1\leq l\leq g_i$.

According to Theorem~\ref{theo_power}, there exist $d_1,\ldots, d_m\in\mathbb{Z}$, not all zero, such that $$f^{d_1}_1(z)\cdots f^{d_m}_m(z)=\prod_{i=1}^m(z-\alpha_{c_i})^{\Gamma_i}\mathfrak{a}_1(z)\mathfrak{j}_1(z),$$
where $\mathfrak{a}_1(z)\in\bm{L}(z)$, $\mathfrak{j}_1(z)\in\mathcal{O}_{\bm{L}}[[z]]$ with $\bm{L}$ a finite extension of $K$, and $\Gamma_i=d_i\cdot\mathrm{Res}\left(\frac{f'_i(z)}{f_i(z)}, D_{\alpha_{c_i}}\right)\in\mathbb{Z}^*$ for all $1\leq i\leq m$.

 Moreover, thanks to (iii) of Theorem~\ref{theo_power} we know that 
 \begin{enumerate}
 \item [(a)] $|\mathfrak{a}_1(z)|_{\mathcal{G}}=1$,
 \item [(b)] $\mathfrak{a}_1(z)\in E(D_{\alpha_{c_i}})$ and $\mathfrak{a}_1(\alpha_{c_i})\neq0$ for every $1\leq i\leq m$;
 \item [(c)] the poles and zeros of $\mathfrak{a}_1(z)$ have norm 1.
 \end{enumerate}

Further, by (ii) of Theorem~\ref{theo_power}, we know that for some $1\leq s\leq m$, $d_s$ is not divisible by $p$.

Now, it is clear that
 \begin{equation}\label{eq_f_s}
 f_{s}(z)=\prod_{i=1}^m(z-\alpha_{c_i})^{\frac{\Gamma_i}{d_{s}}}\mathfrak{a}_1(z)^{1/d_{s}}\mathfrak{j}_1(z)^{1/d_s}\prod_{i=1, i\neq s}^mf_i(z)^{-\frac{d_i}{d_{s}}}.
 \end{equation}
By assumption, we know that $\mathcal{L}_s=\frac{d}{dz^n}+b_1(z)\frac{d}{dz^{n-1}}+\cdots+b_{n-1}(z)\frac{d}{dz}+b_n(z)\in\mathcal{O}_K(z)[d/dz]$ and that $\alpha_{c_s}$ is a regular singular point of $\mathcal{L}_s$. So for all $1\leq j\leq n$, $(z-\alpha_{c_s})^jb_j(z)\in\mathbb{C}_p[[z-\alpha]]$. Furthermore, the indicial polynomial of $\mathcal{L}_s$ at $\alpha_{c_s}$ is given by $$I(X)=X(X-1)\cdots(X-n+1)+s_1X(X-1)\cdots(X-n+2)+\cdots+s_{n-1}X+s_n,$$
 where $s_j=[(z-\alpha)^jb_j(z)](\alpha)$. 

Note that $s_j\in\mathcal{O}_F$ for all $1\leq j\leq n$ because $|b_j(z)|_{\mathcal{G}}=1$ and $|\alpha_{c_s}|=1$.  Thus, $I(X)\in\mathcal{O}_F[X]$.

 We want to prove that  $\Gamma_{s}/d_s$ is an exponent of $\mathcal{L}_{i\mid\mathscr{M}^a_K}$ at $\alpha_{c_i}$. That is, we want to see that $I(\Gamma_{s}/d_s)\equiv0\bmod\pi^a_K$ for all integers $a>0$.

From Theorem~\ref{theo_analytic_element}, we know that, for all $1\leq i\leq m$, $f_i(z)/f_i(z^{p^{h_i}})\in E_{0,K}$ for some integer $h_i>0$. Further, by assumption (3), we know that the Gauss norme of principal part of $f'_i(z)/f_i(z)$ is less than $1/p^a$ for all $1\leq i\leq m$. Since $d_s$ is not divisible by $p$, we conclude, from Lemma~\ref{lem_mod} that there exists $P_{i,a}\in K_0(z)$ with $|P_{i,a}|_{\mathcal{G}}=1$ such that $$f_i(z)^{-\frac{d_i}{d_{s}}}=(P_{i,a}(z)f_i(z^{p^{ah_i}}))^{-\frac{d_i}{d_{s}}}\bmod\pi_{K}^a\mathcal{O}_{K}[[z]].$$
By Lemma~\ref{lem_mod} again, we know that, for every $1\leq i\leq m$, $P'_{i,a}(z)/P_{i,a}(z)$ belongs to $E(D_{\alpha_{c_j}})$ for every $j\neq i$.

Therefore, reducing Equation~\eqref{eq_f_s} modulo $\pi^a_K$ we get $$f_{s}(z)\equiv\mathfrak{r}(z)\mathfrak{c}(z)\bmod\pi^a_{K}\mathcal{O}_{K}[[z]],$$
 where $$\mathfrak{r}(z)=\prod_{i=1}^m(z-\alpha_i)^{\frac{\Gamma_i}{d_{s}}}\mathfrak{a}_1(z)^{1/d_{s}}\left(\prod_{i=1, i\neq s}^mP_{i,a}(z)^{-\frac{d_i}{d_{s}}}\right)\text{ and }\mathfrak{c}(z)=\mathfrak{j}_1(z)^{\frac{1}{d_s}}\prod_{i=1,i\neq s}^mf_i(z^{p^{ah_i}})^{-\frac{d_i}{d_s}}.$$

By assumption, $\mathcal{L}_s(f_s)=0$. So $\mathcal{L}_s(f_s)\equiv0\bmod\pi_K^a\mathcal{O}_{K}[[z]]$ and hence, $$\mathcal{L}_{s}(\mathfrak{r}(z)\mathfrak{c}(z))\equiv0\bmod\pi_K^a\mathcal{O}_{K}[[z]].$$ By (iv) of Theorem~\ref{theo_power}, we know that $\mathfrak{j}'_1(z)$ belongs to $\pi^a_{\bm{L}}\mathcal{O}_{\bm{L}}[[z]]$. Hence, $\mathfrak{c}'(z)\in\pi_{\bm{L}}^a\mathcal{O}_{\bm{L}}[[z]]$. So, $$0\equiv\mathcal{L}_{s}(\mathfrak{r}(z)\mathfrak{c}(z))\equiv\mathfrak{c}(z)\mathcal{L}_{s}(\mathfrak{r}(z))\bmod\pi_{\bm{L}}^a\mathcal{O}_{\bm{L}}[[z]].$$
 Note that $\mathfrak{c}(z)\bmod\pi_{K}^a\mathcal{O}_{K}[[z]]$ is a unit of the ring $\frac{\mathcal{O}_{K}}{\pi_{K}^a}[[z]]$ because $\mathfrak{c}(0)=\mathfrak{j}_1(0)^{1/d_s}$ and we know, by Theorem~\ref{theo_power}, that $|\mathfrak{j}_1(0)|=1$. Therefore, $$\mathcal{L}_{s}(\mathfrak{r}(z))\equiv0\bmod\pi_{\bm{L}}^a\mathcal{O}_{\bm{L}}[[z]].$$
 
 By assumption (1), we know that $D_{\alpha_{c_{s}}}\notin\textbf{Pol}(f'_{i}(z)/f_i(z))$ for all $i\neq s$ and $D_{\alpha_{c_i}}\in\textbf{Pol}(f'_i(z)/f_i(z))$ for all $1\leq i\leq m$. Then, $D_{\alpha_{c_i}}\neq D_{\alpha_{c_{s}}}$ for all $i\neq s$. So, Remark~\ref{rem_discos} implies that $|\alpha_{c_{s}}-\alpha_{c_i}|=1$ for all $i\neq s$. Since $p$ does not divide $d_{s}$ and $\Gamma_i$ is an integer, from Lemma~\ref{lemm_rational}, we get that, for all $i\neq s$, $$(z-\alpha_{c_i})^{\frac{\Gamma_i}{d_{s}}}=\sum_{j\geq0}\zeta_{j,i}(z-\alpha_{c_{s}})^j\in\mathcal{O}_{\bm{F}_i}[[z-\alpha_{c_{s}}]]$$ with $|\zeta_{0,i}|=1$ and $\bm{F}_i$ a finite extension of $\bm{L}$.

Consequently, $$\left(\prod_{i=1,i\neq {s}}^m(z-\alpha_{c_i})^{\Gamma_i/d_{s}}\right)=\left(\prod_{i=1, i\neq s}^m\left(\sum_{j\geq0}\zeta_{j,i}(z-\alpha_{c_{s}})^j\right)\right)=\sum_{j\geq0}\xi_j(z-\alpha_{c_{s}})^j\in\mathcal{O}_{\bm{F}}[[z-\alpha_{c_{s}}]]$$
with $|\xi_0|=1$ and $\bm{F}$ a finite extension of $\bm{L}$ containing the fields $\{\bm{F}_i\}_{i\neq s}$. 

We know that $\mathfrak{a}_1(z)\in E(D_{\alpha_{c_{s}}})$, $\mathfrak{a}_1(\alpha_{c_{s}})\neq0$ and all the poles and zeros of $\mathfrak{a}_1(z)$ have norm $1$. So, by (ii) of Lemma~\ref{lemm_rational}, we get that $$\mathfrak{a}_1(z)^{1/d_{s}}=\sum_{j\geq0}\epsilon_j(z-\alpha_{c_{s}})^j\in\mathcal{O}_{\bm{T}}[[z-\alpha_{c_{s}}]]\text{ with } |\epsilon_0|=1$$ and $\bm{T}$ a finite extension of $\bm{L}$.

%Further, we know that $\mathfrak{j}_1(z)\in\mathcal{A}_p$ and that $|\mathfrak{j}_1(z)|_{\mathcal{G}}|=1=|\mathfrak{j}_1(0)|$. Thus, Remark~\ref{rem_a_p} implies $$\mathfrak{j}_1(z)^{1/d_s}=\sum_{j\geq0}\epsilon_jz^j,$$
% where, for all $j\geq0$, $|\epsilon_j|\leq1$ with $|\epsilon_0|=1$. 

 We have seen that if $i\ne s$ then $P'_{i,a}(z)/P_{i,a}(z)\in E(D_{\alpha_{c_{s}}})$ and the Gauss norm of $P_{i,a}$ is 1. Consequently, by Lemma~\ref{lemm_rational}, we deduce that, for any $i\neq s$, $$P_{i,a}(z)^{-\frac{d_i}{d_{s}}}=\sum_{j\geq0}\gamma_{i,j}(z-\alpha_{c_{s}})^j\in\mathcal{O}_{\bm{F}_{i,a}}[[z-\alpha_{c_{s}}]]\text{ with } |\gamma_{i,0}|=1$$
 and $\bm{F}_{i,a}$ a finite extension of $\bm{L}$. 
 
Then $$\left(\sum_{j\geq0}\xi_j(z-\alpha_{c_{s}})^j\right)\left(\sum_{j\geq0}\epsilon_j(z-\alpha_{c_{s}})^j\right)\prod_{i=1,i\neq s}^m\left(\sum_{j\geq0}\gamma_{i,j}(z-\alpha_{c_{s}})^j\right)\in\mathcal{O}_{\bm{M}}[[z-\alpha_{c_{s}}]]$$
and $\left|\xi_0\epsilon_0\prod_{i=1,i\neq s}^m\gamma_{i,0}\right|=1$ and $\bm{M}$ is a finite extension of $K$ containing the fields $\bm{T}$, $\bm{F}$, and $\{\bm{F}_{i,a}\}_{i\neq s}$.

Further, it is clear that $$\mathfrak{r}(z)=(z-\alpha_{c_{s}})^{\frac{\Gamma_{s}}{d_{s}}}\left(\sum_{j\geq0}\xi_j(z-\alpha_{c_{s}})^j\right)\left(\sum_{j\geq0}\epsilon_j(z-\alpha_{c_{s}})^j\right)\prod_{i=1,i\neq s}^m\left(\sum_{j\geq0}\gamma_{i,j}(z-\alpha_{c_{s}})^j\right).$$
Since $\mathcal{L}_{s}(\mathfrak{r}(z))\equiv0\bmod\pi^a_{\bm{L}}\mathcal{O}_{\bm{L}}$ and $\left|\xi_0\epsilon_0\prod_{i=1,i\neq s}^m\gamma_{i,0}\right|=1$, it follows that Lemma~\ref{lem_exponents} that  $I_s(\Gamma_{s}/d_{s})\equiv0\bmod\pi_K^a$.

Finally, we already know that $$\frac{\Gamma_{s}}{d_{s}}=\mathrm{Res}\left(\frac{f'_{s}(z)}{f_{s}(z)},D_{\alpha_{c_{s}}}\right).$$
$\hfill\square$
 
 \subsection{Proof of Theorem~\ref{theo_rational_function}}
 
According to Theorem~\ref{theo_power}, there exist $d_1,\ldots, d_m\in\mathbb{Z}$, not all zero, such that $$f^{d_1}_1(z)\cdots f^{d_m}_m(z)=\prod_{i=1}^m(z-\alpha_{c_i})^{\Gamma_i}\mathfrak{a}_1(z)\mathfrak{j}_1(z),$$
where $\mathfrak{a}_1(z)\in\bm{L}(z)$, $\mathfrak{j}_1(z)\in\mathcal{O}_{\bm{L}}[[z]]$ with $\bm{L}$ a finite extension of $K$, and $\Gamma_i=d_i\cdot\mathrm{Res}\left(\frac{f'_i(z)}{f_i(z)}, D_{\alpha_{c_i}}\right)\in\mathbb{Z}$ for all $1\leq i\leq m$.

By assumption, we know that the Gauss norm of the principal part of $f'_i/f_i$ is zero for every $1\leq i\leq m$. In particular, the Gauss norm of the principal part of the  $f'_i/f_i$ is less than $|\pi_K|^a$ for any integer $a>0$. Then, according to (iv) of Theorem~\ref{theo_power}, $\mathfrak{j}'_1(z)$ belongs to $\pi_{\bm{L}}^a\mathcal{O}_{\bm{L}}[[z]]$ for any integer $a>0$. Consequently, $\mathfrak{j}'_1(z)=0$. Thus, $\mathfrak{j}_1(z)=\mathfrak{j}_1(0)$. Moreover, by Theorem~\ref{theo_power} again, we know that $|\mathfrak{j}_1(0)|=1$. So, we conclude that $$f_1(z)^{d_1}\cdots f^{d_m}_m(z)=\mathfrak{j}_1(0)\prod_{i=1}^m(z-\alpha_{c_i})^{\Gamma_i}\mathfrak{a}_1(z)\in\bm{L}(z).$$

$(i)\Rightarrow(ii)$ We have already seen that $f^{d_1}_1(z)\cdots f^{d_m}_m(z)\in\bm{L}(z)$. Since $f_1,\ldots, f_m(z)\in K[[z]]$ and $\bm{L}$ is a finite extension of $K$, we deduce that $f_1(z),\ldots, f_m(z)$ are algebraically dependent over $K(z)$.

$(ii)\Rightarrow(i)$ It is clear.
 
$\hfill\square$

\subsection{Proof of Corollary~\ref{theo_radius}}

Suppose that $f_1(z),\ldots, f_m(z)$ are algebraically dependent over $E_{\mathbb{Q}_p}$. Since $f_1(z),\ldots, f_m(z)\in\mathcal{MF}(\mathbb{Q}_p)$ and the principal part of $f'_i/f_i$ is zero for every $1\leq i\leq m$, it follows from Theorem~\ref{theo_rational_function} that there are integers $d_1,\ldots, d_m$, not all zero, such that $f^{d_1}_1(z)\cdots f^{d_m}_m(z)\in\mathbb{C}_p(z)$. Moreover, we know that $\mathbb{C}_p$ and $\mathbb{C}$ are isomorphic as fields. Let $\varphi:\mathbb{C}_p\rightarrow\mathbb{C}$ be an isomorphism. Naturally, $\varphi$ extends to the set of Laurent series with coefficients in $\mathbb{C}_p$ to the set of  Laurent series with coefficients in $\mathbb{C}$ by defining $\varphi(\sum_{n\geq -N}a_nz^n)=\sum_{n\geq-N}\varphi(a_n)z^n$.  By assumption, $f_1(z),\ldots, f_m(z)\in 1+z\mathbb{Q}[[z]]$ and thus, after applying $\varphi$ to $f^{d_1}_1(z)\cdots f^{d_m}_m(z)$ , we deduce that $f^{d_1}_1(z)\cdots f^{d_m}_m(z)\in\mathbb{C}(z)$. Let us say $f^{d_1}_1(z)\cdots f^{d_m}_m(z)=r(z)$ with $r(z)\in\mathbb{C}(z)$. By assumption, $\rho_{f_1},\ldots, \rho_{f_m}$ are distinct pairwise.  Without loss
of generality we assume that $\rho_{f_1}<\ldots<\rho_{f_m}$. Let $j$ be the smallest index for which $d_j\neq0$. If $d_j>0$ then $$f_j(z)^{d_j}=r(z)f_{j+1}(z)^{-d_{j+1}}\cdots f_m(z)^{-d_m}.$$
If $d_j<0$ then $$f_j(z)^{-d_j}=r(z)f_{j+1}(z)^{d_{j+1}}\cdots f_m(z)^{d_m}.$$

 By assumption, we know that $f_{j}(z)\in\mathfrak{M}$. Thus, $f_j(z)$ has a singularity at a point $x\in\mathbb{C}$ such that $|x|=\rho_{f_j}$ and hence, $f_j^{d_j}$ has no meromorphic continuation to a neighborhood
of $x$. But, it is clear that $r(z)f_{j+1}(z)^{-d_{j+1}}\cdots f_m(z)^{-d_m}$ is meromorphic in a neighborhood of $x$. Therefore, we obtain a contradiction. So,  $f_1(z),\ldots, f_m(z)$ are algebraically independent over $E_{\mathbb{Q}_p}$.

$\hfill\square$
 
 \section{The algebraic independence of $E$- and $G$-functions: some examples}\label{sec_apli}
  \subsection{The Bessel Function}\label{sec_trans}
Recall that the Bessel function is the power series $$J_0(z)=\sum_{n\geq0}\frac{(-1)^n}{4^n(n!)^2}z^{2n}.$$ The main goal of this section is to prove that $\mathfrak{B}(z)=J_0(\pi_p z)$ is transcendental over $E_K$ for $p>2$ with $K = \mathbb{Q}_p(\pi_p)$, where $\pi_p^{p-1}=-p$. Note that $K$ is totally ramified because $\pi_p$ is a solution of the Eisenstein polynomial $Y^{p-1}+p$. Furthermore, a uniformizer of $\mathcal{O}_K$ is $\pi_p$. 
A first remark is that we are not able to apply Theorems~\ref{theo_criterion_alg_ind} or \ref{theo_rational_function} because, as Remark~\ref{rem_no_poles} shows, the principal part of $\mathfrak{B}'(z)/\mathfrak{B}(z)$ is the function itself. In consequence, $\mathfrak{B}'(z)/\mathfrak{B}(z)$ has no poles, which is a fundamental requerement for applying Theorems~\ref{theo_criterion_alg_ind} or \ref{theo_rational_function}. Therefore, in order to prove that $J_0(\pi_p z)$ is transcendental over $E_K$, we argue by contradiction and assume that it is algebraic. Then, by Theorem~\ref{theo_alg_ind}, $\mathfrak{B}(z)^N \in E_{0,K}$ for some integer $N \geq 1$. To obtain a contradiction, we begin by estimating the $p$-adic norm of its coefficients.

\begin{lemm}\label{lem_3_adic}
Let $p$ be an odd prime. Then
\begin{enumerate}[label=(\roman*)]
\item For an integer $n\geq1$, let $n=n_0+n_1p+\cdots+n_lp^l$ be its $p$\nobreakdash-adic expansion and let $\Sigma_n=n_0+n_1+\cdots+n_l$. Then 
\begin{equation}\label{eq_3_adic}
\left|\frac{\pi_p^{2n}}{4^n(n!)^2}\right|_p=(1/p)^{\frac{2\Sigma_n}{p-1}},
\end{equation}
\item for every $1<s\leq N$ ($N$ a positive integer), 
\begin{equation}\label{eq_ine_0}
\left|\binom{N}{s}f_1(z)^s\right|_{\mathcal{G}}<\left|Nf_1(z)\right|_{\mathcal{G}},
\end{equation}
where $f_1(z)=\sum_{n\geq1}\frac{\pi_p^{2n}}{4^n(n!)^2}z^{2n}.$
\end{enumerate}
\end{lemm}
\begin{proof}\hfill\\
(i) It is well-known that, for every integer $n\geq1$, $|n!|_p=(1/p)^{(n-\Sigma_n)/(p-1)}$. Whence, we immediately obtain the equality~\eqref{eq_3_adic} for every integer $n\geq1$.\smallskip

(ii) It follows from (i) that $|f_1(z)|_{\mathcal{G}}=(1/p)^{2/p-1}$. Thus, the inequality~\eqref{eq_ine_0} is equivalent to proving that, for every $1<s\leq N$, 
\begin{equation}\label{eq_inequality}
\left|\binom{N}{s}\right|_p <\left|N\right|_p p^{\frac{2(s-1)}{p-1}}.
\end{equation}
Let us suppose that $p$ does not divide $N$. Then $|N|_p=1$. In addition, $p^{2(s-1)/(p-1)}>1$ since $s>1$, and the $p$\nobreakdash-adic norm of $\binom{N}{s}$  is less than or equal to 1 given that $\binom{N}{s}$ is an integer. Hence, the  inequality~\eqref{eq_inequality} holds.

Let us now suppose that $p$ divides $N$ and let $r$ be the $p$\nobreakdash-adic valuation of $N$. Then, we have $\left|N\right|_p p^{2(s-1)/(p-1)}=p^{\frac{2(s-1)}{p-1}-r}$. If $\frac{2(s-1)}{p-1}>r$ then it is clear that the inequality~\eqref{eq_inequality} is true. Now, we suppose that $\frac{1}{p-1}\leq \frac{2(s-1)}{p-1}\leq r$. We first observe that $(N-1)(N-2)(N-3)(N-4)\cdots(N-s+1)=(-1)^{s-1}(s-1)!+Nt$ for some integer $t>0$. Hence, $$(N-1)(N-2)(N-3)(N-4)\cdots(N-s+1)=(-1)^{s-1} p^{\frac{s-1-\Sigma_{s-1}}{p-1}}a+p^{r}bt,$$
where $a,b$ are integers not divisible by $p$. 

Since $s>1$, we have  $\frac{1}{p-1}\leq\frac{s-1}{p-1}<\frac{2(s-1)}{p-1}\leq r$. In addition, it is clear that $s-1-\Sigma_{s-1}<s-1$ because $s>1$. So $\frac{s-1-\Sigma_{s-1}}{p-1}<\frac{s-1}{p-1}< r$. Thus $$|(N-1)(N-2)(N-3)(N-4)\cdots(N-s+1)|_p=(1/p)^{\frac{s-1-\Sigma_{s-1}}{p-1}}.$$
Consequently, $$|N(N-1)(N-2)\cdots(N-s+1)|_p=(1/p)^{r+\frac{s-1-\Sigma_{s-1}}{p-1}}.$$
So, by using the fact that $v_p(s!)=(s-\Sigma_s)/(p-1)$, we get $$\left|\binom{N}{s}\right|_p=\left(\frac{1}{p}\right)^{r+\frac{\Sigma_s-\Sigma_{s-1}-1}{p-1}}.$$
It is not hard to see that, for every $s\geq2$, $s-\Sigma_{s-1}>1$ and that $\Sigma_s\geq1$. Thus, $$\frac{\Sigma_s-\Sigma_{s-1}-1}{p-1}+\frac{s-1}{p-1}=\frac{s-\Sigma_{s-1}-1+\Sigma_s-1}{p-1}>0.$$
So $$\frac{\Sigma_s-\Sigma_{s-1}-1}{p-1}>-\frac{s-1}{p-1}>\frac{-2(s-1)}{p-1}.$$
Whence, $$r+\frac{\Sigma_s-\Sigma_{s-1}-1}{p-1}>r-\frac{2(s-1)}{p-1}.$$
Notice that $0\leq r-2(s-1)/(p-1)$ because we have assumed that $2(s-1)/(p-1)\leq r$. Then $$\left|\binom{N}{s}\right|_p=(1/p)^{r+\frac{\Sigma_s-\Sigma_{s-1}-1}{p-1}}\leq(1/p)^{r-\frac{2(s-1)}{p-1}}=|N|_pp^{\frac{2(s-1)}{p-1}}.$$

That completes the proof of the inequality~\eqref{eq_inequality}. 
\end{proof}

We have the following remark.

\begin{rema}\label{rem_no_poles}
Let $p>2$ be a prime, let $\mathfrak{B}(z)=J_0(\pi_p z)$, and let $K=\mathbb{Q}_p(\pi_p)$. Then, the principal part of $\mathfrak{B}'(z)/\mathfrak{B}(z)$ is the function itself. The latter implies that $D_{\alpha}$ is not a pole of $\mathfrak{B}'(z)/\mathfrak{B}(z)$ for any $\alpha\in\mathcal{O}_{\mathbb{C}_p}$. We set $T(z)=\mathfrak{B}(z)/\mathfrak{B}(z^p)$. It follows from Lemma~\ref{lem_3_adic} that $T(z)\in 1+z\pi_p\mathcal{O}_K[[z]]$. In addition, from \cite[§ 6]{Bessel}, we obtain that $T$ converges on $D(0,1^{+})=\{x\in\mathbb{C}_p: |x|\leq 1\}$. Since 
\begin{equation}\label{eq_b}
\frac{\mathfrak{B}'(z)}{\mathfrak{B}(z)}\equiv\frac{T'_m(z)}{T_m(z)}\bmod{\pi_p^{m+1}\mathcal{O}_{K}[[z]]}, \quad \text{where } T_m(z)=T(z)T(z^p)\cdots T(z^{p^m}),
\end{equation}
and because $T(z)\in 1+z\pi_p\mathcal{O}_{K}[[z]]$, it follows that $|T(\alpha)|=1$ for any $\alpha\in\mathcal{O}_{\mathbb{C}_p}$. Consequently, $|T_m(\alpha)|=1$ for all $m\geq0$. Thus, Equation~\eqref{eq_b} implies that $D_{\alpha}$ is not a pole of $\mathfrak{B}'(z)/\mathfrak{B}(z)$.
\end{rema}

\begin{theo}\label{theo_bessel}
Let $p$ be an odd prime. Then, for all integers $s\geq0$, $J_0^{(s)}(\pi_pz)$ is transcendental over $E_K$.
\end{theo}

The idea of the proof of this theorem is to assume, for the shake of contradiction, that $J_0(\pi_p z)$ is algebraic over $E_K$. Then, for some integer $N > 0$, $J_0^N(\pi_p z) = \sum_{n \geq 0} b_n z^n$ belongs to $E_K$. Using the estimates provided by the previous lemma, we show that there exists an integer $l > 0$ such that the sequence $\{ b_n \bmod \pi_p^l \}_{n \geq 0}$ is lacunary. This implies that $J_0^N(\pi_p z)$ cannot lie in $E_K$, leading to a contradiction.
\begin{proof}
 We put $\mathfrak{B}(z)=J_0(\pi_pz)$. Aiming for a contradiction, suppose that $\mathfrak{B}(z)$ is algebraic over $E_K$. Given that $J_0(z)$ is solution of $\delta^2+z^2$ then $\mathfrak{B}(z)$ is solution of  $\delta^2-\pi_p^2z^2$. It is clear that this differential operator is MOM at zero and given that by assumption $p\neq2$ then, according to Dwork~\cite[§4]{Bessel}, this differential operator has a strong Frobenius structure. Hence, $\mathfrak{B}(z)\in\mathcal{M}\mathcal{F}(K)$ and, since we suppose that $\mathfrak{B}(z)$ is algebraic over $E_K$, by part (i) of Theorem~\ref{theo_alg_ind}, there exists an integer $N>0$ such that $\mathfrak{B}(z)^N\in E_{0,K}$. By definition, we have $$\mathfrak{B}(z)=\sum_{n\geq0}\frac{(-1)^n\pi_p^{2n}}{4^n(n!)^2}z^{2n}.$$
We set $f_1(z)=\sum_{n\geq1}\frac{(-1)^n\pi_p^{2n}}{4^n(n!)^2}z^{2n}.$  Then  $\mathfrak{B}(z)^N=(1+f_1(z))^N=\sum_{s=0}^N\binom{N}{s}f_1(z)^s.$ 

According to (ii) of Lemma~\ref{lem_3_adic} we know that, for every $1<s\leq N$, 
\begin{equation*}
\left|\binom{N}{s}f_1(z)^s\right|_{\mathcal{G}}<\left|Nf_1(z)\right|_{\mathcal{G}}.
\end{equation*}

From (i) of Lemma~\ref{lem_3_adic}, we deduce that $|f_1(z)|_{\mathcal{G}}=(1/p)^{2/p-1}$. Thus, for all $1<s\leq N$, 
\begin{equation}\label{eq_N_s}
\left|\binom{N}{s}f_1(z)^s\right|_{\mathcal{G}}<(1/p)^{v_p(N)+\frac{2}{p-1}}.
\end{equation} 
By definition, $\pi_p$ is a solution of the Eisenstein polynomial $X^{p-1}+p$, and thus $\pi_p$ is a uniformizer of $K=\mathbb{Q}_p(\pi_p)$. Given that, for any \( s \geq 1 \), we have \( \binom{N}{s} f_1(z)^s \in\mathcal{O}_K[[z]] \), Inequality~\eqref{eq_N_s} implies that, for every \( 1 < s \leq N \), \( \binom{N}{s} f_1(z)^s \in (\pi_p)^{v_p(N)(p-1)+3} \mathcal{O}_K[[z]] \). Consequently, for every $1< s\leq N$, $$\binom{N}{s}f_1(z)^s\equiv0\bmod\pi_p^{(p-1)v_p(N)+3}\mathcal{O}_{K}[[z]].$$ 
Therefore, we obtain 
\begin{equation}\label{eq_b_n}
\mathfrak{B}(z)^{N}\equiv1+Nf_1(z)\bmod\pi_p^{(p-1)v_p(N)+3}\mathcal{O}_{K}[[z]].
\end{equation}
Again, from (i) of Lemma~\ref{lem_3_adic},  we conclude that, for all $n\geq0$, 
\begin{equation}\label{eq_p_norm}
\left|\frac{\pi_p^{2p^n}}{4^{p^n}(p^n)!^2}\right|_p=(1/p)^{2/p-1}.
\end{equation}

Consequently, for all $n\geq0$, $N\frac{\pi_p^{2p^n}}{4^{p^n}(p^n)!^2}\notin\pi_p^{(p-1)v_p(N)+3}\mathcal{O}_{K}$. In other words, for all $n\geq0$, $$N\frac{\pi_p^{2p^n}}{4^{p^n}(p^n)!^2}\not\equiv0\bmod\pi_p^{(p-1)v_p(N)+3}\mathcal{O}_{K}$$

Now, it is clear that if $m\notin\{p^n\}_{n\geq0}$ then $\Sigma_m>1$ and thus, by using (i) of Lemma~\ref{lem_3_adic} again,  we obtain
\begin{equation*}%\label{eq_p_norm_1}
\left|\frac{\pi_p^{2m}}{4^{m}(m)!^2}\right|_p=(1/p)^{\frac{2\Sigma_m}{p-1}}<(1/p)^{\frac{2}{p-1}}.
\end{equation*}
 Therefore, if $m\notin\{p^n\}_{n\geq0}$ then $$N\frac{\pi_p^{2m}}{4^{m}(m)!^2}\equiv0\bmod\pi_p^{(p-1)v_p(N)+3}\mathcal{O}_{K}.$$ Hence, $$f_1(z)=\sum_{n\geq1}\frac{(-1)^n\pi_p^{2n}}{4^n(n!)^2}z^{2n}\equiv\sum_{n\geq1}\frac{\pi_p^{2p^n}}{4^{p^n}(p^n)!^2}z^{2p^n}\bmod\pi_p^{(p-1)v_p(N)+3}\mathcal{O}_{K}[[z]].$$ 
 So, from Equation~\eqref{eq_b_n}, we get 
 \begin{equation}\label{eq_f_1}
\mathfrak{B}(z)^N\equiv1+Nf_1(z)\equiv1-N\sum_{n\geq1}\frac{\pi_p^{2p^n}}{4^{p^n}(p^n)!^2}z^{2p^n}\bmod\pi_p^{(p-1)v_p(N)+3}\mathcal{O}_{K}[[z]].
 \end{equation}
Write $\mathfrak{B}(z)^N=\sum_{n\geq0}b_nz^n$. Then, Equality~\eqref{eq_f_1} implies that the sequence $\{b_n\bmod\pi_{p}^{v_p(N)(p-1)+3}\}_{n\geq0}$ is lacunary.  That fact will be crucial in deriving a contradiction.

Since $\mathfrak{B}(z)^{N}$ belongs to $E_{0,K}$  there exists a rational function $R(z)\in K_0(z)$ such that $$\mathfrak{B}(z)^N\equiv R(z)\bmod\pi_p^{(p-1)v_p(N)+3}\mathcal{O}_{K}[[z]].$$ Let us write $R(z)=\sum_{n\geq0}a_nz^n$. Then, it follows from the above equality and Equations~\eqref{eq_f_1} that 
\begin{equation}\label{eq_rational}
\sum_{n\geq0}a_nz^n\equiv1-N\sum_{n\geq1}\frac{\pi_p^{2p^n}}{4^{p^n}(p^n)!^2}z^{2p^n}\bmod \pi_p^{(p-1)v_p(N)+3}\mathcal{O}_{K}[[z]].
\end{equation} 
On the one hand, since $R(z)$ is a rational function, the sequence $\{a_n\}_{n\geq0}$ is almost periodic. That is, for any $\epsilon>0$, there is an integer $n_{\epsilon}$ such that, for all $n\geq n_{\epsilon}$, $|a_{n+n_{\epsilon}}-a_{n}|_p<\epsilon$. On the other hand, it follows from Equation~\eqref{eq_rational} that, for all $m\notin\{2p^k: k\geq1\}$, $a_m\in\pi_p^{(p-1)v_p(N)+3}\mathcal{O}_{K}$ and thus, $|a_m|_p\leq (1/p)^{v_p(N)+\frac{3}{p-1}}$. 

Take $\epsilon'= (1/p)^{v_p(N)+\frac{3}{p-1}}$. Then there is an integer $n_{\epsilon'}$ such that, for all $n\geq n_{\epsilon'}$, $|a_{n+n_{\epsilon'}}-a_n|_p<\epsilon'$. In addition, it is clear that there exists $n_0\in\{2p^{k}: k\geq1\}$ such that $n_0\geq n_{\epsilon'}$ and $n_0+n_{\epsilon'}\notin\{2p^{k}: k\geq1\}$. If $n_0=2p^{k_0}$, it follows from \eqref{eq_rational} that $$a_{n_0}\equiv\lambda\bmod\pi_p^{(p-1)v_p(N)+3}\mathcal{O}_{K}\text{ with }\lambda=-N\frac{\pi_p^{2p^{k_0}}}{4^{p^{k_0}}(p^{k_0})!^2}.$$ Thus $|a_{n_0}-\lambda|_p\leq\epsilon'$. Further, from Equation~\eqref{eq_p_norm}, we have $|\lambda|_p=(1/p)^{v_p(N)+\frac{2}{p-1}}>\epsilon'$. Since $|a_{n_0}|_p=|a_{n_0}-\lambda+\lambda|_p$ and $|a_{n_0}-\lambda|_p\leq\epsilon'<|\lambda|_p$, and since moreover the norm is non\nobreakdash-Archimedean, we get that $|a_{n_0}-\lambda+\lambda|_p=|\lambda|_p$. Hence, $|a_{n_0}|_p=|\lambda|_p$. Given that $n_0+n_{\epsilon'}\notin\{2p^{k}: k\geq1\}$ we know that $|a_{n_0+n_{\epsilon'}}|_p\leq\epsilon'$. Therefore, $|a_{n_0}|_p>|a_{n_0+n_{\epsilon'}}|_p$. Again, given that the norm is non\nobreakdash-Archimedean, we get that $|a_{n_0+n_{\epsilon'}}-a_{n_0}|_p=|a_{n_0}|_p=|\lambda|_p>\epsilon'$. But, that is a contradiction because  $|a_{n_0+n_{\epsilon'}}-a_{n_0}|_p< \epsilon'$. 

So $\mathfrak{B}(z)$ is transcendental over $E_{0,K}$. Thus, by (ii) of Theorem~\ref{theo_alg_ind}, we get that, for all $s\geq0$, $\mathfrak{B}^{(s)}(z)$ is transcendental over $E_K$. 
That completes the proof.

\end{proof}

\subsection{The Exponential}

Recall that  $$\exp(\pi_pz)=\sum_{n\geq0}\frac{\pi^n_p}{n!}z^n.$$
This power series is a solution of $\delta -\pi_pz$. It is clear that this differential operator is MOM at zero and, according to Dwork~\cite{DworksFf}, it also has a strong Frobenius structure. Then, $\exp(\pi_pz)$ belongs to $\mathcal{MF}(K)$ (recall that $K=\mathbb{Q}_p(\pi_p)$).

Now, the power series $\exp(\pi_pz)$ is algebraic over $E_K$ because $\exp(\pi_pz)^p$ belongs to $E_{K}$. Indeed, we have $$\exp(\pi_pz)^p=\exp(p\pi_pz)=\sum_{n\geq0}\frac{p^n\pi_p^n}{n!}z^n.$$
  It follows from (i) of Lemma~\ref{lem_3_adic} that $$\left|\frac{p^n\pi_p^n}{n!}\right|_p=(1/p)^{n+\frac{\Sigma_n}{p-1}}.$$
  So $\lim_{n\rightarrow\infty}\left|\frac{p^n\pi_p^n}{n!}\right|_p=0$ and thus  $\exp(\pi_pz)^p$ is the limit of the polynomials $\left\{\sum_{j=0}^{n}\frac{p^j\pi_p^j}{j!}z^j\right\}_{n\geq0}$.\medskip
  
Then, as consequence of Theorem~\ref{theo_bessel} we get that, for any $s\geq0$, $\mathfrak{B}^{(s)}(z)$ is transcendental over $E_K[\exp(\pi_pz)]$ because $\mathfrak{B}^{(s)}(z)$ is transcendental over $E_{K}$ and $\exp(\pi_pz)$ is algebraic over $E_K$.

%Another consequence is that, for any $(r,s)\in\mathbb{N}^2$ the $\exp(\pi_p z)^{(s)}$ and $(\mathfrak{B}(z)\exp(\pi_p z))^{(r)}$ are algebraically independent over $E_K$. Since $\mathfrak{B}(z)$ and $\exp(\pi_pz)$ belong to $\mathcal{MF}(K)$, it follows from Remark~\ref{gg} of Part I that $\mathfrak{B}(z)\exp(\pi_p z)$. Further, $\mathfrak{B}(z)\exp(\pi_p z)$ is trascendental over $E_K$ because $\mathfrak{B}(z)$ is trascendental over $E_K$ and $\exp(\pi_p)$ is algebraic over $E_K$. Finally, if $\exp(\pi_p z)$ and $\mathfrak{B}(z)\exp(\pi_p z)$ are algebraically dependent over $E_K$ then, by Theorem~\ref{theo_alg_ind}, there are integers $a,b$, not all zero, such that $\exp(\pi_pz)^a(\mathfrak{B}(z)\exp(\pi_p z))^b\in E_{K}$. Given that $\mathfrak{B}(z)$ and $\mathfrak{B}(z)\exp(\pi_p z)$ are trascendental over $E_K$, we have $a\geq0$ and $b\neq0$. Thus, it follows that $\mathfrak{B}(z)^{p(a+b)}\exp(\pi_pz)^{pb}\in E_K$

\subsection{Algebraic independence of $G$-functions}

In this section, we prove the algebraic independence of certain $G$-functions by applying Theorem~\ref{theo_criterion_alg_ind} and Corollary~\ref{theo_radius}. To this end, we need to compute the poles, residues, and the principal part of $f'/f$ with $f(z)\in\mathcal{MF}(\mathbb{Q}_p)$. For instance, when $f(z)$ satisfies Dwork-congruences, we manage to get some usefeul information about the poles, residues and the Gauss norm of the principal part of $f'/f$.

\begin{defi}[Dwork-congruences]
Let $K$ be a finite extension of $\mathbb{Q}_p$, $\pi$ be a uniformizer of $\mathcal{O}_K$, and let $f(z)=\sum_{n\geq0}a_nz^n$ be in $1+z\mathcal{O}_K[[z]]$. We say that $f(z)$ satisfies Dwork-congruences if, for every $r\geq1$, $$\frac{f(z)}{f(z^p)}\equiv\frac{F_r(z)}{F_{r-1}(z)}\bmod \pi^r\mathcal{O}_K[[z]],\text{ where for all } r\geq1,\quad F_r(z)=\sum_{n=0}^{p^r-1}a_nz^n.$$
\end{defi}

We have the following lemma

\begin{lemm}\label{lem_dc}
Let $K$ be a Frobenius field, $\pi$ be a uniformizer of $\mathcal{O}_K$, and let $f(z)=\sum_{n\geq0}a_nz^n \in 1+z\mathcal{O}_K[[z]]$. If $f(z)$ satisfies Dwork congruences, then $f(z)\in\mathcal{MF}(K)$.
\end{lemm}

\begin{proof}
It is clear that $f(z)$ is a solution of $\mathcal{L}=\delta-a_0(z)$, where $a_0(z)=\frac{\delta f(z)}{f(z)}$. We first prove that $\mathcal{L}$ belongs to $E_{0,K}[\delta]$. Indeed, we know that $H(z)=\frac{f(z)}{f(z^p)}$ belongs to $E_{p}$ because, by assumption, $f(z)$ satisfies Dwork congruences. Furthermore, it is easy to check that for all integers $m\geq0$, 
$$\delta H_m(z)=H_m(z)\left[\frac{\delta f(z)}{f(z)}-p^{m+1}\frac{(\delta f)(z^{p^{m+1}})}{f(z^{p^{m+1}})}\right], \quad \text{where } H_m(z)=H(z)H(z^p)\cdots H(z^{p^{m}}).$$
Consequently, $H_m\in E_{p}$ and 
$$a_0(z)=\frac{\delta f(z)}{f(z)}=\lim\limits_{m\rightarrow\infty}\frac{\delta H_m(z)}{H_m(z)}.$$
Since $\frac{\delta H_m(z)}{H_m(z)}\in E_p$ and $E_p$ is complete, we conclude that $a_0(z)\in E_p$. Moreover, for all $x\in D(0,1^{-})$, $|f(x)|=1$ because the norm is non-Archimedean and $f(0)=1$; thus, $a_0(z)\in E_{0,K}$. Hence, $\mathcal{L}$ belongs to $E_{0,K}[\delta]$. 

Now, it is clear that $\mathcal{L}$ is MOM at zero because $a_0(0)=0$. Finally, $\mathcal{L}$ is equipped with a strong Frobenius structure because $H\in E_p$ and 
$$\delta H=H\left[\frac{\delta f(z)}{f(z)}-p\frac{(\delta f)(z^{p})}{f(z^p)}\right].$$
Therefore, $f(z)\in\mathcal{MF}(K)$.
\end{proof}

We have the following lemma concerning the computation of the poles, residues, and principal part of $f'/f$ when $f(z)$ satisfies Dwork-congruences.

\begin{lemm}\label{lemm_DC}
Let $K$ be a Frobenius field, let $\pi$ be a uniformizer of $\mathcal{O}_K$, and let $f(z)=\sum_{n\geq0}a_nz^n\in 1+z\mathcal{O}_K[[z]]$ be a power series satisfying Dwork congruences. Set $d=\deg(F_1(z)\bmod\pi)$ and $\bm{F}_1(z)=\sum_{n=0}^{d}a_nz^n$, where $F_1=\sum_{n=0}^{p-1}a_nz^n$. Then, we have
\begin{enumerate}[label=(\roman*)]
\item $\bm{Pol}(f'/f)=\bm{Pol}(\bm{F}'_1/\bm{F}_1)$.
\item If, for all $1\leq i\leq r$, $\deg(F_i(z)\bmod\pi^i)=d(1+p+p^2+\cdots+p^{i-1})$, then the Gauss norm of the principal part of $f'/f$ is less than or equal to $|\pi|^r$, where $F_i(z)=\sum_{n=0}^{p^i-1}a_nz^n$.
\item If $d=p-1$, then the principal part of $f'/f$ is zero.
\end{enumerate}
\end{lemm}

\begin{proof}
(i). Set $\Omega=\{x\in\mathcal{O}_{\mathbb{C}_p}: |F_1(x)|=1\}$ and $H(z)=f(z)/f(z^p)$. First, note that $\Omega=\{x\in\mathcal{O}_{\mathbb{C}_p}: |\bm{F}_1(x)|=1\}$. Since $f(z)$ satisfies Dwork congruences, it follows from (iii) of Lemma~2 in \cite{masha} that $H(z)\in E(\Omega)$ and $|H(x)|=1$ for all $x\in\Omega$. Hence, $H$ is a unit of the ring $E(\Omega)$. Therefore, if for every $m\geq0$ we consider $H_m(z)=H(z)H(z^{p})\cdots H(z^{p^{m}})$, then $H_m\in E(\Omega)$ and $|H_m(x)|=1$ for all $x\in\Omega$. Thus, $H_m$ is a unit of $E(\Omega)$, and we also have the identity
\begin{equation*}
H'_m(z)=H_m(z)\left[\frac{f'(z)}{f(z)}-\pi^{m+1}z^{p^{m+1}-1}\frac{f'(z^{p^{m+1}})}{f(z^{p^{m+1}})}\right].
\end{equation*}
Thus, for all integers $m\geq0$,
\begin{equation*}
\frac{f'(z)}{f(z)}\equiv\frac{H_m'(z)}{H_m(z)}\bmod{\pi^{m+1}\mathcal{O}_K[[z]]}.
\end{equation*}
Since $H_m$ is a unit of $E(\Omega)$ for all $m\geq0$, the previous congruences imply that $\frac{f'(z)}{f(z)}$ belongs to $E(\Omega)$. 

Let 
$$\frac{f'}{f}=\sum_{n\geq0}f_n$$ 
be the decomposition given by Mittag-Leffler's Theorem, where $f_n\in E^{\alpha_n}$ with $\alpha_n\in\mathcal{O}_{\mathbb{C}_p}$ for every $n>0$. Then, $\alpha_n\notin\Omega$ for all $n>0$. In particular, for every $n>0$, we have $\alpha_n\in\mathcal{O}_{\mathbb{C}_p}$ and $|\bm{F}_1(\alpha_n)|<1$; hence, the reduction $\overline{\alpha_n}$ is a root of $\bm{F}_1(z)\bmod\pi$ for every $n>0$.
 
Let $D_{\alpha}$ be a pole of $f'/f$. Then $\alpha\in D_{\alpha_n}$ for some $n>0$, and according to Remark~\ref{rem_discos}, $\overline{\alpha_n}=\overline{\alpha}$. Thus, $\overline{\alpha}$ is a root of $\bm{F}_1(z)\bmod\pi$. In addition, by Theorem~6.1 of~\cite[Chapter I]{Dworkgfunciones}, every root of $\bm{F}_1$ has norm $1$. Consequently, there exists a root $\beta\in\mathbb{C}_p$ of $\bm{F}_1$ such that $D_{\beta}=D_{\alpha}$. Since $D_{\beta}$ is clearly a pole of $\bm{F}'_1/\bm{F}_1$ and $D_{\beta}=D_{\alpha}$, we conclude that $D_{\alpha}$ is a pole of $\bm{F}'_1/\bm{F}_1$.
 
Let $\gamma_1,\ldots, \gamma_l$ be the distinct roots of $\bm{F}_1$. It is clear that 
$$\frac{\bm{F}'_1}{\bm{F}_1}=\sum_{i=1}^l\frac{m_{\gamma_i}}{z-\gamma_i},$$ 
where $m_{\gamma_i}$ is the multiplicity of $\gamma_i$. We have $\sum_{i=1}^l m_{\gamma_i} < p$ because $d=\sum_{i=1}^l m_{\gamma_i}$ and $d\leq p-1<p$. Now, let $\tau_1,\ldots,\tau_s$ be the roots of $\bm{F}_1$ such that $D_{\tau_i}\cap D_{\tau_j}=\emptyset$ for all $i\neq j$ and $\gamma_1,\ldots, \gamma_l\in D_{\tau_1}\cup\cdots\cup D_{\tau_s}$. Thus, by (i) of Lemma~\ref{lem_conv}, the decomposition obtained by applying Mittag-Leffler's Theorem to $\bm{F}'_1/\bm{F}_1$ is given by
$$\frac{\bm{F}'_1}{\bm{F}_1}=\frac{\sum_{\gamma_i\in D_{\tau_1}}m_{\gamma_i}}{z-\tau_1}+\sum_{\gamma_i\in D_{\tau_1}}m_{\gamma_{i}}\sum_{k\geq1}\frac{(\gamma_i-\tau_1)^k}{(z-\tau_1)^{k+1}}+\cdots+\frac{\sum_{\gamma_i\in D_{\tau_s}}m_{\gamma_i}}{z-\tau_s}+\sum_{\gamma_i\in D_{\tau_s}}m_{\gamma_{i}}\sum_{k\geq1}\frac{(\gamma_i-\tau_s)^k}{(z-\tau_s)^{k+1}}.$$
This shows that $\bm{Pol}(\bm{F}'_1/\bm{F}_1)=\{D_{\tau_1},\ldots, D_{\tau_s}\}$.
 
Clearly, $\bm{F}_1'/\bm{F}_1\equiv F'_1/F_1\bmod\pi\mathcal{O}_K[[z]]$. Since $f'/f\equiv F'_1/F_1\bmod\pi\mathcal{O}_{K}[[z]]$, we obtain $f'/f\equiv\bm{F}_1'/\bm{F}_1\bmod\pi\mathcal{O}_K[[z]]$. Moreover, for any $1\leq j\leq s$, we have $\sum_{\gamma_i\in D_{\tau_j}}m_{\gamma_i}\leq d<p$, which implies that the norm of $\sum_{\gamma_i\in D_{\tau_j}}m_{\gamma_i}$ is $1$. Thus, 
$$\frac{f'}{f}\equiv\frac{\bm{F}'_1}{\bm{F}_1}\equiv\frac{\sum_{\gamma_i\in D_{\tau_1}}m_{\gamma_i}}{z-\tau_1}+\cdots+\frac{\sum_{\gamma_i\in D_{\tau_s}}m_{\gamma_i}}{z-\tau_s}\bmod{p\mathcal{O}_K[[z]]}.$$
Consequently, $D_{\tau_1},\ldots, D_{\tau_s}\in\bm{Pol}(f'/f)$.
 
(ii). For every $1\leq i\leq r$, we set $\bm{F}_i(z)=\sum_{n=0}^{d(1+p+\cdots+p^{i-1})}a_nz^n$. We first show that every root of $\bm{F}_i$ has norm 1 for all $1\leq i\leq r$. According to Theorem~6.1 of~\cite[Chapter I]{Dworkgfunciones}, it is sufficient to show that the $p$-adic valuation of $a_{d(1+p+\cdots+p^{i-1})}$ is zero. Since $f(z)$ satisfies Dwork congruences, we have $F_k(z)\equiv F_1(z)F_1(z^{p})\cdots F_1(z^{p^{k-1}})\bmod\pi\mathcal{O}_K[z]$ for all $k\geq1$. By assumption, we know that for every $1\leq i\leq r$, $\deg(F_i(z)\bmod\pi^i)=d(1+p+p^2+\cdots+p^{i-1})$, which implies that $\bm{F}_i(z)\equiv F_i(z)\bmod\pi^i$ for all $1\leq i\leq r$. In particular, $\bm{F}_i(z)\equiv F_i(z)\bmod\pi$ for every $1\leq i\leq r$. Consequently, $\bm{F}_i(z)\equiv F_1(z)F_1(z^{p})\cdots F_1(z^{p^{i-1}})\bmod\pi\mathcal{O}_K[z]$ for every $1\leq i\leq r$. Since it is clear that $\deg(F_1(z)F_1(z^{p})\cdots F_1(z^{p^{i-1}})\bmod\pi)=d(1+p+\cdots+p^{i-1})$, it follows that $\deg(\bm{F}_i\bmod p)=d(1+p+\cdots+p^{i-1})$. Hence, for every $1\leq i\leq r$, the $p$-adic valuation of $a_{d(1+p+\cdots+p^{i-1})}$ is zero. Thus, every root of $\bm{F}_i$ has norm 1 for all $1\leq i\leq r$. 

Now, since $f(z)$ satisfies Dwork congruences, we have 
$$\frac{f(z)}{f(z^p)}\equiv \frac{F_r(z)}{F_{r-1}(z^p)}\bmod{\pi^{r}\mathcal{O}_K[[z]]}.$$
Thus, 
$$\frac{f(z)}{f(z^{p^{r}})}\equiv \frac{F_r(z)F_r(z^{p})\cdots F_r(z^{p^{r-1}})}{F_{r-1}(z^p)\cdots F_{r-1}(z^{p^{r}})}\bmod{\pi^{r}\mathcal{O}_K[[z]]}.$$
This yields 
$$\frac{f'(z)}{f(z)}\equiv\sum_{j=0}^{r-1}p^jz^{p^{j}-1}\frac{F'_r(z^{p^j})}{F_r(z^{p^j})}-\sum_{j=1}^{r-1}p^jz^{p^{j}-1}\frac{F'_{r-1}(z^{p^{j}})}{F_{r-1}(z^{p^{j}})}\bmod{\pi^r\mathcal{O}_K[[z]]}.$$
We know that $F_{r-1}(z)\equiv\bm{F}_{r-1}(z)\bmod\pi^{r-1}$ and hence $\frac{F_{r-1}'(z)}{F_{r-1}(z)}\equiv\frac{\bm{F}'_{r-1}(z)}{\bm{F}_{r-1}(z)}\bmod\pi^{r-1}$. Thus, for any integer $j\geq1$, we have $p^{j}\frac{F_{r-1}'(z)}{F_{r-1}(z)}\equiv p^j\frac{\bm{F}'_{r-1}(z)}{\bm{F}_{r-1}(z)}\bmod\pi^r$. Since we also have $\frac{F_{r}'(z)}{F_{r}(z)}\equiv\frac{\bm{F}'_{r}(z)}{\bm{F}_{r}(z)}\bmod\pi^{r}$, we get
\begin{equation}\label{eq_f'_f}
\frac{f'(z)}{f(z)}\equiv\sum_{j=0}^{r-1}p^jz^{p^{j}-1}\frac{\bm{F}'_r(z^{p^j})}{\bm{F}_r(z^{p^j})}-\sum_{j=1}^{r-1}p^jz^{p^{j}-1}\frac{\bm{F}'_{r-1}(z^{p^{j}})}{\bm{F}_{r-1}(z^{p^{j}})}\bmod{\pi^r\mathcal{O}_K[[z]]}.
\end{equation}
Because every root of the polynomials $\bm{F}_r$ and $\bm{F}_{r-1}$ has norm 1, the principal part of the rational function 
$$\sum_{j=0}^{r-1}p^jz^{p^{j}-1}\frac{\bm{F}'_r(z^{p^j})}{\bm{F}_r(z^{p^j})}-\sum_{j=1}^{r-1}p^jz^{p^{j}-1}\frac{\bm{F}'_{r-1}(z^{p^{j}})}{\bm{F}_{r-1}(z^{p^{j}})}$$ 
is zero. Therefore, Equation~\eqref{eq_f'_f} implies that the Gauss norm of the principal part of $f'(z)/f(z)$ is less than or equal to $|\pi|^r$.

(iii). Since $f(z)$ satisfies Dwork congruences, we have $F_r(z)\equiv F_1(z)F_1(z^{p})\cdots F_1(z^{p^{r-1}})\bmod\pi\mathcal{O}_K[[z]]$ for all $r\geq1$. Given that $d=p-1$, it follows that $\deg(F_r(z)\bmod\pi)=(p-1)(1+p+\cdots+p^{r-1})=p^r-1$. Consequently, the $p$-adic valuation of $a_{p^r-1}$ is zero for any $r\geq1$; thus, $\deg(F_r\bmod\pi^r)=p^r-1=(p-1)(1+p+\cdots+p^{r-1})$ for any $r\geq1$. According to part (ii), the Gauss norm of the principal part of $f'/f$ is less than or equal to $|\pi|^r$ for all $r\geq1$. Taking the limit as $r \to \infty$, we conclude that the principal part of $f'/f$ is zero.
\end{proof}

 As illustration of Theorem~\ref{theo_criterion_alg_ind}, in this section we will prove the algebraic independence of some $E$ and $G$-functions. Let us consider the following power series $$\mathfrak{h}(z)=\sum_{n\geq0}\frac{1}{64^n}\binom{2n}{n}^3z^n,\quad \mathfrak{f}(z)=\sum_{n\geq0}\frac{-1}{(2n-1)64^n}\binom{2n}{n}^3z^n,{ and }\quad\mathfrak{A}(z)=\sum_{n\geq0}\left(\sum_{k=0}^n\binom{n}{k}^2\binom{n+k}{k}^2\right)z^n.$$
Note that $\mathfrak{h}(z)$ is the hypergeometric series ${}_3F_2((1/2,1/2,1/2),(1,1),z)$ and $\mathfrak{f}(z)$ is the hypergeometric series ${}_3F_2((-1/2,1/2,1/2),(1,1),z)$.
 \begin{theo}\label{theo_h_A} 
 Let $K$ be a finite extension of  $\mathbb{Q}_3$ that is a Frobenius field.
 \begin{enumerate}[label=(\roman*)]
 \item The power series $\mathfrak{h}(z)$ and $\mathfrak{f}(z)$ are algebraically dependent over $E_{K}$. Moreover, $\mathfrak{f}(z)/\mathfrak{h}(z)\in E_K$.
 \item The power series $\mathfrak{h}(z)$ and $\mathfrak{A}(z)$ are algebraically independent over $E_K$.
 \item The power series $\mathfrak{f}(z)$ and $\mathfrak{A}(z)$ are algebraically independent over $E_K$.
 \end{enumerate}
 \end{theo}
 \begin{proof}
 The power series $\mathfrak{h}(z)$, $\mathfrak{f}(z)$ are respectively  solution of the following hypergeometric operators 
 \begin{equation*}\label{eq_diff}
 \mathcal{L}:=\delta^3-z(\delta+1/2)^3,\quad \mathcal{D}=\delta^3-z(\delta-1/2)(\delta+1/2)^2.
 \end{equation*}
 It is clear that these differential operators are MOM at zero and, it follows from \cite[Theorem~6.2]{vargas1} that they have strong Frobenius structure for $p=3$. For this reason $\mathfrak{h}(z)$ and $\mathfrak{f}(z)$ belong to $\mathcal{M}\mathcal{F}(\mathbb{Q}_3)$ and thus, they also belong to $\mathcal{M}\mathcal{F}(K)$.
  
 (i). By \cite[Theorem~6.2]{vargas1} again, we know that  $\delta^3-z(\delta+1/2)^3$ has a strong Frobenius structure for $p=3$ of period $1$. Thus, according to (ii) of Theorem~\ref{theo_analytic_element},  we get that $\mathfrak{h}(z)/\mathfrak{h}(z^3)\in E_{\mathbb{Q}_3}$. But, according to Corollary 2 of Section 1 of \cite{Dworkpciclos} and Theorem~2 of \cite{Dworkpciclos}, $\mathfrak{f}(z)/\mathfrak{h}(z^3)\in E_{\mathbb{Q}_3}$. Given that $E_{\mathbb{Q}_3}$ is a field we obtain $\mathfrak{f}(z)/\mathfrak{h}(z)\in E_{\mathbb{Q}_3}$. 
 
 %Finally, as $\mathfrak{h}(z)$ and $\mathfrak{f}(z)$ belong to $\mathcal{M}\mathcal{F}(\mathbb{Q}_3)$ and $\mathbb{Q}_3$ is a Frobenius field then, from  Lemma~\ref{lemm_derivatives}, we deduce that $\mathfrak{f}^{(s)}(z)/\mathfrak{f}(z)$ and $\mathfrak{h}(z)/\mathfrak{h}^{(r)}(z)$ belong to $E_{\mathbb{Q}_3}$ for all integers $r,s\geq0$. Consequently, $$\frac{\mathfrak{f}^{(s)}(z)}{\mathfrak{h}^{(r)}(z)}=\frac{\mathfrak{f}^{(s)}(z)}{\mathfrak{f}(z)}\cdot\frac{\mathfrak{f}(z)}{\mathfrak{h}(z)}\cdot\frac{\mathfrak{h}(z)}{\mathfrak{h}^{(r)}(z)}$$ belongs to $E_{\mathbb{Q}_3}$ for all integers $r,s\geq0$. 
 
 %Given that $K$ is an extension of $\mathbb{Q}_3$, we get $E_{\mathbb{Q}_3}\subset E_K$ and thus, for all integers $r,s\geq0$, $\frac{\mathfrak{f}^{(s)}(z)}{\mathfrak{h}^{(r)}(z)}$ belongs to $E_K$.
 
 (ii). We have already seen that $\mathfrak{h}(z)\in\mathcal{M}\mathcal{F}(K)$. Further, it follows from \cite[Theorem~1]{masha} that $\mathfrak{A}(z)$ satisfies Dwork-congruences for $p=3$. Then, according to Lemma~\ref{lem_dc}, $\mathfrak{A}(z)$ belongs to $\mathcal{M}\mathcal{F}(K)$.  In addition, thanks to to Corollary 2 of Section 1 of \cite{Dworkpciclos} and Theorem~2 of \cite{Dworkpciclos}, we know that $\mathfrak{h}(z)$ also satisfies Dwork-congruences for $p=3$. Then, by (i) of Lemma~\ref{lemm_DC}, we conclude that $\bm{Pol}(\mathfrak{h}'/\mathfrak{h})=\{D_{1}\}$ and $\bm{Pol}(\mathfrak{A}'/\mathfrak{A})=\{D_{-1}\}$.

We now put
\[
\lambda_1=res\left(\frac{\mathfrak{h}'(z)}{\mathfrak{h}(z)},D_1\right)\text{ and }\lambda_2=res\left(\frac{\mathfrak{A}'(z)}{\mathfrak{A}(z)},D_{-1}\right).
\]
 We proceed to show that $\lambda_1\bmod 9=4$ and $\lambda_2=8\bmod 9.$
 
Since $\mathfrak{h}(z)$ satisfied Dwork-congruences for $p=3$, we have
\begin{equation*}%\label{eq_h}
\frac{\mathfrak{h}(z)}{\mathfrak{h}(z^9)}=\frac{\mathfrak{h}_2(z)}{\mathfrak{h}_1(z^3)}\frac{\mathfrak{h}_2(z^3)}{\mathfrak{h}_1(z^9)}\bmod 9.
\end{equation*}
Therefore, 
\begin{equation*}%\label{eq_h_mod9}
\frac{\mathfrak{h}'(z)}{\mathfrak{h}(z)}=\frac{8+8z+5z^2}{1-z^3}=\frac{2z+1}{z^2+z+1}-\frac{7}{z-1}\bmod 9.
\end{equation*}
%Note that $$\frac{8+8z+5z^2}{1-z^3}=\frac{2z+1}{z^2+z+1}-\frac{7}{z-1}$$
 Let $z_0,z_1\in\overline{\mathbb{Q}_3}$ be the roots of $1+z+z^2$. Then $|z_0|=1=|z_1|$, $z_0\neq z_1$, and  $z_0, z_1,\in D_{1}$ because $1+z+z^2=(1-z)^2\bmod 3$. Furthermore, it is clear that $$\frac{2z+1}{z^2+z+1}-\frac{7}{z-1}=\frac{1}{z-z_0}+\frac{1}{z-z_1}-\frac{7}{z-1}.$$
 Since $1,z_0, z_1\in D_{1}$, it follows from (1) of Lemma~\ref{lem_poles} and (ii) of Lemma~\ref{lem_res} that $$\lambda_1\bmod 9=(1+1-7)\bmod 9=-5\bmod 9=4.$$
 
On the other hand, we also know that $\mathfrak{A}(z)$ satisfies Dwork-congruences for $p=3$ and thus,
\begin{equation*}%\label{eq_A}
\frac{\mathfrak{A}(z)}{\mathfrak{A}(z^9)}=\frac{\mathfrak{A}_2(z)}{\mathfrak{A}_1(z^3)}\frac{\mathfrak{A}_2(z^3)}{\mathfrak{A}_1(z^9)}\bmod 9\mathbb{Z}[[z]]
\end{equation*}
So
\begin{equation*}%\label{eq_A_mod9}
\frac{\mathfrak{A}'(z)}{\mathfrak{A}(z)}=\frac{5+4z+8z^2+4z^3+5z^4+z^5}{1-z^6}\bmod 9\mathbb{Z}.
\end{equation*}
Note that
\begin{equation}\label{eq_A_mod9_1}
\frac{5+4z+8z^2+4z^3+5z^4+z^5}{1-z^6}=\frac{2z-1}{z^2-z+1}+\frac{3}{2(z+1)}\bmod 9\mathbb{Z}.
\end{equation}
Let $S$ be the roots of $z^2-z+1$. Note that $S\subset D_{-1}$. Further, we have $\sum_{\tau\in S}\left(\frac{2z-1}{z^2-z+1},\tau\right)=2.$ Thus, by (1) of Lemma~\ref{lem_poles} and (ii) of Lemma~\ref{lem_res}, we deduce from Equation~\eqref{eq_A_mod9_1} that $$\lambda_2\bmod 9=(2+3/2)\bmod 9=8.$$

Finally, we have $\mathcal{L}(\mathfrak{h})=0$ and let us write $\mathcal{L}$ in terms of $d/dz$. Then, $$\mathcal{L}=\frac{d}{dz^3}+\left(\frac{3}{z}-\frac{3}{2(1-z)}\right)\frac{d}{dz^2}+\left(\frac{1}{z^2}-\frac{3}{2z(1-z)}-\frac{1}{z(1-z)}\right)\frac{d}{dz}-\frac{1}{8z^2(z-1)}\in\mathcal{O}_{\mathbb{Q}_3(z)}[d/dz].$$
So, it is clear that $1$ is a regular singular point of $\mathcal{L}_{\mid\mathscr{M}^2}$. The indicial polynomial of $\mathcal{L}_{1\mid\mathscr{M}^2}$ at $1$ is $P=X(X-1)(X-2)+\frac{3}{2}X(X-1)\in\frac{\mathbb{Z}}{9\mathbb{Z}}[X]$.

Note that $\lambda_1\bmod 9$ is not an exponent of $\mathcal{L}_{\mid\mathscr{M}^2}$ at $1$ because $\lambda_1\bmod 9=4$ and $P(4)=6$.

 Further, we also know that $\mathfrak{A}(z)$ is solution of the differential operator $$\mathcal{H}=\frac{d}{dz^3}+\frac{3-153z+6z^2}{(1-34z+z^2)z}\frac{d}{dz^2}+\frac{1-112z+7z^2}{(1-34z+z^2)z^2}\frac{d}{dz}+\frac{z-5}{(1-34z+z^2)z^2}\in\mathcal{O}_{\mathbb{Q}_3(z)}[d/dz].$$
 We are going to see that $-1$ is a regular singular point of $\mathcal{H}_{\mid\mathscr{M}^2}$. We first observe that $1-34z+z^2\equiv(1+z)^2\bmod 9\mathbb{Z}[[z]]$ and that $3-153z+6z^2\equiv(1+z)(3+6z)\bmod 9\mathbb{Z}[[z]]$. Whence, $$(z+1)\frac{3-153z+6z^2}{z(1-34z+z^2)}\equiv\frac{3+6z}{z}\bmod\mathscr{M}^2\mathcal{O}_{\mathbb{Q}_3(z)},$$ $$(1+z)^2\frac{1-112z+7z^2}{z^2(1-34z+z^2)}\equiv\frac{1+5z+7z^2}{z^2}\bmod\mathscr{M}^2\mathcal{O}_{\mathbb{Q}_3(z)},$$
 $$(z+1)^3\frac{z-5}{z^2(1-34z+z^2)}\equiv\frac{(z+1)(z-5)}{z^2}\bmod\mathscr{M}^2\mathcal{O}_{\mathbb{Q}_3(z)}.$$
 Further, it is clear that $\frac{3+6z}{z}$, $\frac{1+5z+7z^2}{z^2}$ and that $\frac{(z+1)(z-5)}{z^2}$ belong to $\mathbb{Z}_3[[z+1]]$. So $-1$ is a regular singular point of $\mathcal{H}_{\mid\mathscr{M}^2}$. In addition, the indicial polynomial of $\mathcal{H}_{\mid\mathscr{M}^2}$ at $-1$ is $$R=X(X-1)(X-2)+3X(X-1)+3X=X(X^2+2)\in\frac{\mathbb{Z}}{9\mathbb{Z}}[[z]].$$
 But $\lambda_2\bmod9$ is not an exponent of $\mathcal{H}_{\mid\mathscr{M}^2}$ at $-1$ because $\lambda_2\bmod9=8$ and $R(8)=6$.

 Consequently, by Theorem~\ref{theo_criterion_alg_ind}, we conclude that $\mathfrak{h}(z)$ and $\mathfrak{A}(z)$ are algebraically independent over $E_{\mathbb{Q}_3}$.

(iii) Let us suppose that $\mathfrak{f}(z)$ and $\mathfrak{A}(z)$ are algebraically dependent over $E_{K}$. Then, by Theorem~\ref{theo_alg_ind}, there are integers $a,b$ not all zero such that $\mathfrak{f}^{a}(z)\mathfrak{A}^b(z)\in E_{K}.$ According to (i), $\mathfrak{f}/\mathfrak{h}\in  E_{K}$. Thus, $$\mathfrak{h}^{a}(z)\mathfrak{A}^b(z)=\frac{\mathfrak{h}^a(z)}{\mathfrak{f}^a(z)}\mathfrak{f}^{a}(z)\mathfrak{A}^b(z)\in E_{K}.$$
Whence, $\mathfrak{A}(z)$ and $\mathfrak{h}(z)$ are algebraically dependent over $E_{K}$.That is a contradiction because, we have already seen that  $\mathfrak{A}(z)$ and $\mathfrak{h}(z)$ are algebraically independent over $E_{K}$. Consequently, $\mathfrak{f}(z)$ and $\mathfrak{A}(z)$ are algebraically independent over $E_{K}$.
\end{proof}

\textbf{Changing the prime $p$}

The aim of this section is to prove that the power series
$$\mathfrak{F}(z)=\sum_{n\geq0}\left(\sum_{k=0}^n\binom{n}{k}^3\right)z^n\quad\text{and}\quad\mathfrak{A}(z)=\sum_{n\geq0}\left(\sum_{k=0}^n\binom{n}{k}^2\binom{n+k}{k}^2\right)z^n$$
are algebraically independent over $\mathbb{Q}(z)$ by showing that they are algebraically independent over $E_{\mathbb{Q}_p}$ for some prime $p$. Thus, the strategy is to test the algebraic independence of these power series for small primes.

By \cite[Theorem 3]{masha}, we know that $\mathfrak{F}(z)$ and $\mathfrak{A}(z)$ satisfy Dwork congruences for every prime $p$. In the case $p=2$, we have 
\[
\frac{\mathfrak{F}(z)}{\mathfrak{F}(z^2)} \equiv \mathfrak{F}_1(z) \bmod{2\mathbb{Z}_2[[z]]}, \quad \text{where } \mathfrak{F}_1(z) = 1 + 2z
\]
and 
\[
\frac{\mathfrak{A}(z)}{\mathfrak{A}(z^2)} \equiv \mathfrak{A}_1(z) \bmod{2\mathbb{Z}_2[[z]]}, \quad \text{where } \mathfrak{A}_1(z) = 1 + 5z.
\]
It is clear that $\mathfrak{F}_1\equiv1\bmod 2$; hence, Lemma~\ref{lem_dc} implies that $\bm{Pol}(\mathfrak{F}'/\mathfrak{F})=\emptyset$. Therefore, in this case, we cannot apply Theorem~\ref{theo_criterion_alg_ind} or Theorem~\ref{theo_rational_function} to determine whether $\mathfrak{F}(z)$ and $\mathfrak{A}(z)$ are algebraically independent over $E_{\mathbb{Q}_2}$.

Now, for $p=3$, we have the following situation:
\[
\frac{\mathfrak{F}(z)}{\mathfrak{F}(z^3)} \equiv \mathfrak{F}_1(z) \bmod{3\mathbb{Z}_3[[z]]}, \quad \text{where } \mathfrak{F}_1(z) = 1 + 2z + 10z^2
\]
and 
\[
\frac{\mathfrak{A}(z)}{\mathfrak{A}(z^3)} \equiv \mathfrak{A}_1(z) \bmod{3\mathbb{Z}_3[[z]]}, \quad \text{where } \mathfrak{A}_1(z) = 1 + 5z + 73z^2.
\]
Consequently, $\deg(\mathfrak{F}_1(z)\bmod 3)=2=\deg(\mathfrak{A}_1(z)\bmod 3)$. It then follows from Lemma~\ref{lem_dc} that the principal parts of both $\mathfrak{F}'(z)/\mathfrak{F}(z)$ and $\mathfrak{A}'(z)/\mathfrak{A}(z)$ are zero. 

Thus, an initial attempt to study the algebraic independence of $\mathfrak{F}(z)$ and $\mathfrak{A}(z)$ over $E_{\mathbb{Q}_3}$ would be to apply Corollary~\ref{theo_radius}. However, we cannot use this result because condition (1) of Theorem~\ref{theo_criterion_alg_ind} is not satisfied. More precisely, we have $\bm{Pol}(\mathfrak{F}'(z)/\mathfrak{F}(z))=\bm{Pol}(\mathfrak{A}'(z)/\mathfrak{A}(z))$ because $\mathfrak{F}_1(z)\bmod 3 = 1+2z+z^2 = \mathfrak{A}_1(z)\bmod 3$. Nevertheless, we can establish the algebraic independence of $\mathfrak{F}(z)$ and $\mathfrak{A}(z)$ over $E_{\mathbb{Q}_5}$ by using Corollary~\ref{theo_radius}. In other words, by changing the prime from $p=3$ to $p=5$, we can conclude that these power series are algebraically independent over $\mathbb{Q}(z)$.

\begin{theo}\label{theo_FA}
Let $K$ be a finite extension of $\mathbb{Q}_5$. If $K$ is a Frobenius field then the power series $$\mathfrak{F}(z)=\sum_{n\geq0}\left(\sum_{k=0}^n\binom{n}{k}^3\right)z^n\text{ and }\mathfrak{A}(z)=\sum_{n\geq0}\left(\sum_{k=0}^n\binom{n}{k}^2\binom{n+k}{k}^2\right)z^n$$
are algebraically independent over $E_K$.
\end{theo}

\begin{proof}
According to \cite[Theorem 3]{masha}, $\mathfrak{F}(z)$ and $\mathfrak{A}(z)$ satisfy Dwork congruences for $p=5$,. Then, by Lemma~\ref{lem_dc}, we have $\mathfrak{F}(z), \mathfrak{A}(z) \in \mathcal{M}\mathcal{F}(\mathbb{Q}_5)$. Since $\mathcal{M}\mathcal{F}(\mathbb{Q}_5) \subset \mathcal{M}\mathcal{F}(K)$, it follows that $\mathfrak{F}(z), \mathfrak{A}(z) \in \mathcal{M}\mathcal{F}(K)$. Furthermore, 
\[
\frac{\mathfrak{F}(z)}{\mathfrak{F}(z^5)} \equiv \mathfrak{F}_1(z) \bmod{5\mathbb{Z}_5[[z]]}, \quad \text{where } \mathfrak{F}_1(z) = 1 + 2z + 10z^2 + 56z^3 + 346z^4
\]
and 
\[
\frac{\mathfrak{A}(z)}{\mathfrak{A}(z^5)} \equiv \mathfrak{A}_1(z) \bmod{5\mathbb{Z}_5[[z]]}, \quad \text{where } \mathfrak{A}_1(z) = 1 + 5z + 73z^2 + 1445z^3 + 33001z^4.
\]
Thus, 
\[
\deg(\mathfrak{F}_1(z) \bmod 5) = 4 = \deg(\mathfrak{A}_1(z) \bmod 5).
\]
Therefore, by Lemma~\ref{lemm_DC}, the principal parts of $\mathfrak{F}'(z)/\mathfrak{F}(z)$ and $\mathfrak{A}'(z)/\mathfrak{A}(z)$ are zero.

Furthermore, we have 
\[
\mathfrak{F}_1(z) \equiv 1 + 2z + z^3 + z^4\equiv(z - 1)(z - 3)(z^2 + 3) \bmod 5
\]
and 
\[
\mathfrak{A}_1(z) \equiv 1 + 3z^2 + z^4\equiv(z - 1)^2(z - 4)^2 \bmod 5.
\]
It is then clear that $\bm{Pol}\left(\frac{\mathfrak{F}'_1(z)}{\mathfrak{F}_1(z)}\right) = \{D_{3}, D_1, D_{-\alpha}, D_{\alpha}\}$, where $\alpha \in \mathbb{C}_5$ is such that $|\alpha^2 - 3| < 1$, and $\bm{Pol}\left(\frac{\mathfrak{A}'_1(z)}{\mathfrak{A}_1(z)}\right) = \{D_4, D_1\}$. Then, by Lemma~\ref{lemm_DC}, we obtain $\bm{Pol}\left(\frac{\mathfrak{F}'(z)}{\mathfrak{F}(z)}\right) = \{D_{3}, D_1, D_{-\alpha}, D_{\alpha}\}$ and $\bm{Pol}\left(\frac{\mathfrak{A}'(z)}{\mathfrak{A}(z)}\right) = \{D_4, D_1\}$. Consequently, $D_4 \in \bm{Pol}\left(\frac{\mathfrak{A}'(z)}{\mathfrak{A}(z)}\right)$ but $D_4 \notin \bm{Pol}\left(\frac{\mathfrak{F}'(z)}{\mathfrak{F}(z)}\right)$, whereas $D_3 \in \bm{Pol}\left(\frac{\mathfrak{F}'(z)}{\mathfrak{F}(z)}\right)$ but $D_3 \notin \bm{Pol}\left(\frac{\mathfrak{A}'(z)}{\mathfrak{A}(z)}\right)$.

Hence, to apply Corollary~\ref{theo_radius}, it suffices to verify that $\mathfrak{F}(z)$ and $\mathfrak{A}(z)$ belong to $\mathfrak{M}$ and that $\rho(\mathfrak{F}(z)) \neq \rho(\mathfrak{A}(z))$. On the one hand, it follows from Remark~7.5 of \cite{ABD19} that $\mathfrak{F}(z)$ and $\mathfrak{A}(z)$ belong to $\mathfrak{M}$. On the other hand, $\mathfrak{F}(z)$ and $\mathfrak{A}(z)$ are solutions, respectively, of the differential equations
\[
\frac{d^2}{dz^2} + \frac{24z^2+14z-1}{z(z+1)(8z-1)}\frac{d}{dz} + \frac{2(4z+1)}{z(z+1)(8z-1)} = 0
\]
\[
\frac{d^3}{dz^3} + \frac{6z^2-153z+3}{(z^2-34z+1)z}\frac{d^2}{dz^2} + \frac{7z^2-112z+1}{(z^2-34z+1)z^2}\frac{d}{dz} + \frac{z-5}{(z^2-34z+1)z^2} = 0.
\]
Thus, $\rho(\mathfrak{F}(z)) = 1/8$ and $\rho(\mathfrak{A}(z)) = (\sqrt{2} - 1)^4$. Consequently, $\rho(\mathfrak{F}(z)) \neq \rho(\mathfrak{A}(z))$.

Therefore, by Corollary~\ref{theo_radius}, we conclude that $\mathfrak{F}(z)$ and $\mathfrak{A}(z)$ are algebraically independent over $E_K$.
\end{proof}

\begin{rema}\label{rem_res_FA}
We retain the notation of the previous theorem. Thus, we have 
\[
\frac{\mathfrak{F}'(z)}{\mathfrak{F}(z)} \equiv \frac{\mathfrak{F}'_1(z)}{\mathfrak{F}_1(z)} \equiv \frac{1}{z-1} + \frac{1}{z-3} + \frac{2z}{z^2+3} \bmod 5
\]
and 
\[
\frac{\mathfrak{A}'(z)}{\mathfrak{A}(z)} \equiv \frac{\mathfrak{A}'_1(z)}{\mathfrak{A}_1(z)} \equiv \frac{2}{z-1} + \frac{2}{z-4} \bmod 5.
\]
It then follows from Lemma~\ref{lem_res}~(ii) that 
\[
\operatorname{res}\left(\frac{\mathfrak{F}'(z)}{\mathfrak{F}(z)}, D_{3}\right) \equiv 1 \bmod 5 \quad \text{and} \quad \operatorname{res}\left(\frac{\mathfrak{A}'(z)}{\mathfrak{A}(z)}, D_4\right) \equiv 2 \bmod 5.
\]
In particular, $\operatorname{res}\left(\frac{\mathfrak{F}'(z)}{\mathfrak{F}(z)}, D_{3}\right)\neq0$ and $\operatorname{res}\left(\frac{\mathfrak{A}'(z)}{\mathfrak{A}(z)}, D_4\right)\neq0$. 
\end{rema}

%The strategy used to prove Theorem~\ref{theo_ind_hyper} can also be implemented to prove the algebraic independence of other $G$\nobreakdash-functions in $\mathcal{M}\mathcal{F}(\mathbb{Q}_p)$. Let us again illustrate this strategy by proving Theorem~\ref{theo_ind_hyper_apery}.

  \subsection{Proof of Theorem~\ref{theo_appli}}

\begin{proof} 
Let $\mathfrak{B}(z)=J_0(\pi_5z)$ and let $K=\mathbb{Q}_5(\pi_5)$. In the proof of Theorem~\ref{theo_FA}, we established that $\mathfrak{F}(z)$ and $\mathfrak{A}(z)$ belong to $\mathcal{M}\mathcal{F}(K)$. Furthermore, we showed in the proof of Theorem~\ref{theo_bessel} that $\mathfrak{B}(z)\in\mathcal{M}\mathcal{F}(K)$. Moreover, it follows from Remark~\ref{rem_no_poles} that for any $\alpha\in\mathcal{O}_{\mathbb{C}_5}$, $D_{\alpha}$ is not a pole of $\mathfrak{B}'(z)/\mathfrak{B}(z)$. 

We first prove that $\mathfrak{F}(z)$, $\mathfrak{A}(z)$, and $\mathfrak{B}(z)$ are algebraically independent over $E_K$.

Aiming for a contradiction, suppose that $\mathfrak{F}(z)$, $\mathfrak{A}(z)$, and $\mathfrak{B}(z)$ are algebraically dependent over $E_K$. Then, according to Theorem~\ref{theo_alg_ind}, there exist integers $a$, $b$, and $c$, not all zero, such that $\mathfrak{F}(z)^a\mathfrak{A}(z)^b\mathfrak{B}(z)^c\in E_{0,K}$. Note first that either $a\neq0$ or $b\neq0$ because, by Theorem~\ref{theo_bessel}, $\mathfrak{B}(z)$ is transcendental over $E_{0,K}$. Furthermore, $c\neq0$ since, by Theorem~\ref{theo_FA}, we know that $\mathfrak{F}(z)$ and $\mathfrak{A}(z)$ are algebraically independent over $E_K$. Without loss of generality, we assume that $a\neq0$. We have seen in the proof of Theorem~\ref{theo_FA} that $D_3\in\bm{Pol}\left(\frac{\mathfrak{F}'(z)}{\mathfrak{F}(z)}\right)$ but $D_3\notin\bm{Pol}\left(\frac{\mathfrak{A}'(z)}{\mathfrak{A}(z)}\right)$, and we also know that $D_3\notin\bm{Pol}\left(\frac{\mathfrak{B}'(z)}{\mathfrak{B}(z)}\right)$. Hence, $D_3\in\bm{Pol}\left(a\frac{\mathfrak{F}'(z)}{\mathfrak{F}(z)}+b\frac{\mathfrak{A}'(z)}{\mathfrak{A}(z)}+c\frac{\mathfrak{B}'(z)}{\mathfrak{B}(z)}\right)$ and 
$$\mu=\operatorname{res}\left(a\frac{\mathfrak{F}'(z)}{\mathfrak{F}(z)}+b\frac{\mathfrak{A}'(z)}{\mathfrak{A}(z)}+c\frac{\mathfrak{B}'(z)}{\mathfrak{B}(z)}, D_{3}\right)=a\cdot\operatorname{res}\left(\frac{\mathfrak{F}'(z)}{\mathfrak{F}(z)}, D_3\right).$$
Since $a\neq0$ and, by Remark~\ref{rem_res_FA}, $\operatorname{res}\left(\frac{\mathfrak{F}'(z)}{\mathfrak{F}(z)}, D_3\right)\neq0$, we obtain that 
$$\mu=\operatorname{res}\left(a\frac{\mathfrak{F}'(z)}{\mathfrak{F}(z)}+b\frac{\mathfrak{A}'(z)}{\mathfrak{A}(z)}+c\frac{\mathfrak{B}'(z)}{\mathfrak{B}(z)}, D_{3}\right)\neq0.$$

Note that $f(z)=\mathfrak{F}(z)^a\mathfrak{A}(z)^b\mathfrak{B}(z)^c\in\mathcal{MF}(K)$. Indeed, first note that $f(z)$ is a solution to the differential operator 
$$\mathcal{L}=\delta-a\frac{\delta\mathfrak{F}(z)}{\mathfrak{F}(z)}-b\frac{\delta\mathfrak{A}(z)}{\mathfrak{A}(z)}-c\frac{\delta\mathfrak{B}(z)}{\mathfrak{B}(z)}.$$
Since $\mathfrak{F}(z)$, $\mathfrak{A}(z)$, and $\mathfrak{B}(z)$ belong to $\mathcal{MF}(K)$, it follows from Theorem~\ref{theo_analytic_element} that $\frac{\delta\mathfrak{F}(z)}{\mathfrak{F}(z)}$, $\frac{\delta\mathfrak{A}(z)}{\mathfrak{A}(z)}$, and $\frac{\delta\mathfrak{B}(z)}{\mathfrak{B}(z)}$ are in $E_{0,K}$. Consequently, $\mathcal{L}\in E_{0,K}[\delta]$, and it is clear that $\mathcal{L}$ is MOM. Furthermore, we already know that $\mathfrak{F}(z)/\mathfrak{F}(z^5)$, $\mathfrak{A}(z)/\mathfrak{A}(z^5)$, and $\mathfrak{B}(z)/\mathfrak{B}(z^5)$ belong to $E_{0,K}$; hence, $f(z)/f(z^5)\in E_{0,K}$. Therefore, $\mathcal{L}$ is equipped with a strong Frobenius structure and consequently, $f(z)\in\mathcal{MF}(K)$.

Given that $f(z)\in E_{0,K}$, Lemma~\ref{lem_inicio} implies that
\begin{equation}\label{eq_rela_alg}
\mathfrak{F}(z)^a\mathfrak{A}(z)^b\mathfrak{B}(z)^c=(z-3)^{\mu}\mathfrak{a}(z)\mathfrak{j}(z),
\end{equation}
where $\mathfrak{a}(z)=\prod (z-\beta_i)^{\mu_i}$, with $\beta_i\in\overline{\mathbb{Q}_p}$, $|\beta_i|=1$, $\mu_i=\operatorname{res}\left(a\frac{\mathfrak{F}'(z)}{\mathfrak{F}(z)}+b\frac{\mathfrak{A}'(z)}{\mathfrak{A}(z)}+c\frac{\mathfrak{B}'(z)}{\mathfrak{B}(z)}, D_{\beta_i}\right)\in\mathbb{Z}^{*}$, and $\mathfrak{j}(z)=(\mathfrak{o}(z))^{1/p^e}$ with $\mathfrak{o}(z)=\prod_{n=1}^{\infty}\left(\frac{z-a_n}{z-b_n}\right)$, where $(a_n,b_n)_{n\geq1}$ is a strongly copiercing sequence associated with $D_0$ and $e\geq0$ is an integer. 

Furthermore, $\mathfrak{o}(z)=\prod_{i=1}^2\mathfrak{o}_i(z)$, where $\mathfrak{o}_1(z)=\prod_{b_n\in\mathcal{O}_{\mathbb{C}_5}}\left(\frac{z-a_n}{z-b_n}\right)$ and $\mathfrak{o}_2(z)=\prod_{b_n\notin\mathcal{O}_{\mathbb{C}_5}}\left(\frac{z-a_n}{z-b_n}\right).$ Since the sequence $(a_n,b_n)_{n\geq1}$ is a strongly copiercing sequence associated with $D_0$, the sequences used to define $\mathfrak{o}_1(z)$ and $\mathfrak{o}_2(z)$ are also strongly copiercing sequences associated with $D_0$.

Thus, we obtain 
\begin{multline*}
a\frac{\mathfrak{F}'(z)}{\mathfrak{F}(z)}+b\frac{\mathfrak{A}'(z)}{\mathfrak{A}(z)}+c\frac{\mathfrak{B}'(z)}{\mathfrak{B}(z)}=\frac{\mu}{z-3}+\sum\frac{\mu_i}{z-\beta_i}\\
+\frac{1}{p^e}\sum_{b_n\in\mathcal{O}_{\mathbb{C}_5}}\left(\frac{1}{z-a_n}-\frac{1}{z-b_n}\right)+\frac{1}{p^e}\sum_{b_n\notin\mathcal{O}_{\mathbb{C}_5}}\left(\frac{1}{z-a_n}-\frac{1}{z-b_n}\right).
\end{multline*}
Consequently, 
\begin{multline}\label{eq_deco_b'}
c\frac{\mathfrak{B}'(z)}{\mathfrak{B}(z)}=\frac{\mu}{z-3}+\sum\frac{\mu_i}{z-\beta_i}-a\frac{\mathfrak{F}'(z)}{\mathfrak{F}(z)}-b\frac{\mathfrak{A}'(z)}{\mathfrak{A}(z)}\\
+\frac{1}{p^e}\sum_{b_n\in\mathcal{O}_{\mathbb{C}_5}}\left(\frac{1}{z-a_n}-\frac{1}{z-b_n}\right)+\frac{1}{p^e}\sum_{b_n\notin\mathcal{O}_{\mathbb{C}_5}}\left(\frac{1}{z-a_n}-\frac{1}{z-b_n}\right).
\end{multline}

Using Mittag-Leffler's Theorem (see Equation~\eqref{eq_mittag}), we know that there exists a unique sequence $\{\alpha_n\}_{n\geq1}$ in $\mathcal{O}_{\mathbb{C}_5}$ such that $|\alpha_n|=1$ for every $n$, $|\alpha_n-\alpha_m|=1$ for all $n\neq m$, and there exists a unique sequence $\{f_n\}_{n\geq0}$ in $E_{0,K}$ such that $f_0\in E(\mathcal{O}_{\mathbb{C}_5})$, $f_n\in E^{\alpha_n}$ for $n>0$, and 
$$c\frac{\mathfrak{B}'(z)}{\mathfrak{B}(z)}=\sum_{n\geq0}f_n.$$
However, from Remark~\ref{rem_no_poles}, we know that the principal part of $\mathfrak{B}'(z)/\mathfrak{B}(z)$ is the function itself, meaning that $\mathfrak{B}'(z)/\mathfrak{B}(z)=f_0$. In particular, $\sum_{n>0}f_n=0$. Now, as we showed in the proof of Theorem~\ref{theo_FA}, the principal parts of $\mathfrak{F}'(z)/\mathfrak{F}(z)$ and $\mathfrak{A}'(z)/\mathfrak{A}(z)$ are zero, and it is clear that the principal parts of the terms $\frac{\mu}{z-3}$, $\sum\frac{\mu_i}{z-\beta_i}$, and $\frac{1}{p^e}\sum_{b_n\in\mathcal{O}_{\mathbb{C}_5}}\left(\frac{1}{z-a_n}-\frac{1}{z-b_n}\right)$ are also zero. Thus, it follows from Equation~\eqref{eq_deco_b'} that 
$$c\frac{\mathfrak{B}'(z)}{\mathfrak{B}(z)}=\frac{1}{p^e}\sum_{b_n\notin\mathcal{O}_{\mathbb{C}_5}}\left(\frac{1}{z-a_n}-\frac{1}{z-b_n}\right).$$

Given that $\frac{\mathfrak{o}'_2(z)}{\mathfrak{o}_2(z)}=\sum_{b_n\notin\mathcal{O}_{\mathbb{C}_5}}\left(\frac{1}{z-a_n}-\frac{1}{z-b_n}\right)$ and $c\neq0$, we have 
$$\frac{\mathfrak{B}'(z)}{\mathfrak{B}(z)}=\frac{1}{cp^e}\frac{\mathfrak{o}'_2(z)}{\mathfrak{o}_2(z)},$$ 
which implies that there exists $\lambda\in\mathbb{C}_5\setminus\{0\}$ such that 
$$\mathfrak{B}(z)=\lambda \mathfrak{o}_2(z)^{1/cp^e}.$$
Therefore, $\mathfrak{B}(z)^{cp^e}=\lambda^{cp^e}\mathfrak{o}_2(z)$. Since $(a_n,b_n)_{b_n\notin\mathcal{O}_{\mathbb{C}_5}}$ is a strongly copiercing sequence associated with $D_0$, it follows from Remark~\ref{rem_copiercing} that $\mathfrak{o}_2(z)$ belongs to $E_{0,p}$. Consequently, $\mathfrak{B}(z)^{cp^e}\in E_{0,p}$. Since $c\neq0$ and $p^e\neq0$, this implies that $\mathfrak{B}(z)$ is algebraic over $E_{0,p}$, which is a contradiction.

Therefore, $\mathfrak{F}(z)$, $\mathfrak{A}(z)$, and $\mathfrak{B}(z)$ are algebraically independent over $E_K$. Consequently, by Theorem~\ref{theo_alg_ind}, for all integers $r$, $s$, $k\geq0$, the derivatives $\mathfrak{B}^{(r)}(z)$, $\mathfrak{A}^{(s)}(z)$, and $\mathfrak{F}^{(k)}(z)$ are algebraically independent over $E_K$.
\end{proof}

As a consequence of Theorem~\ref{theo_appli} we the following result

\begin{coro}\label{coro_Q(z)}
For all integers $r,s, k\geq0$, the power series $\exp(\pi_5z)$, $J_0^{(r)}(\pi_5z)$, $\mathfrak{A}^{(s)}(z)$, and $\mathfrak{F}^{(k)}(z)$ are algebraically independent over $\mathbb{Q}(z)$.
\end{coro}

\begin{proof}
We first show that for any integers $r,s, k\geq0$, the power series $\mathfrak{B}^{(r)}(z) = J_0^{(r)}(\pi_5z)$, $\mathfrak{A}^{(s)}(z)$, and $\mathfrak{F}^{(k)}(z)$ are algebraically independent over $E_K(\exp(\pi_5z))$, where $K=\mathbb{Q}_5(\pi_5)$. To this end, we argue by contradiction: suppose there exist integers $r,s, k\geq0$ such that $\mathfrak{B}^{(r)}(z)$, $\mathfrak{A}^{(s)}(z)$, and $\mathfrak{F}^{(k)}(z)$ are algebraically dependent over $E_K(\exp(\pi_5z))$. 

We already know that $\mathfrak{B}(z)$, $\mathfrak{A}(z)$, and $\mathfrak{F}(z)$ belong to $\mathcal{M}\mathcal{F}(K)$. Thus, Theorem~\ref{theo_analytic_element} implies that the logarithmic derivatives $\mathfrak{B}'(z)/\mathfrak{B}(z)$, $\mathfrak{A}'(z)/\mathfrak{A}(z)$, and $\mathfrak{F}'(z)/\mathfrak{F}(z)$ belong to $E_K$. Consequently, it is straightforward to see that $\mathfrak{B}^{(r+1)}(z)/\mathfrak{B}^{(r)}(z)$, $\mathfrak{A}^{(s+1)}(z)/\mathfrak{A}^{(s)}(z)$, and $\mathfrak{F}^{(k+1)}(z)/\mathfrak{F}^{(k)}(z)$ belong to $E_K$. In other words, $\mathfrak{B}^{(r)}(z)$, $\mathfrak{A}^{(s)}(z)$, and $\mathfrak{F}^{(k)}(z)$ are solutions to first-order differential operators with coefficients in $E_{K}$.  Since they are assumed to be algebraically dependent over $E_K(\exp(\pi_5z))$, and noting that $E_K\subset E_K(\exp(\pi_5z))$, we deduce by Kolchin's Theorem (see, e.g., Theorem 3.3 of Part I) that there exist integers $a, b, c$, not all zero, such that $(\mathfrak{B}^{(r)}(z))^a(\mathfrak{A}^{(s)}(z))^b(\mathfrak{F}^{(k)}(z))^c\in E_K(\exp(\pi_5z))$. However, as observed in Section~\ref{sec_trans}, $\exp(\pi_5z)$ is algebraic over $E_K$; thus, the extension $E_K(\exp(\pi_5z))$ is algebraic over $E_K$. In particular, the product $(\mathfrak{B}^{(r)}(z))^a(\mathfrak{A}^{(s)}(z))^b(\mathfrak{F}^{(k)}(z))^c$ must also be algebraic over $E_K$. This yields a contradiction because, according to Theorem~\ref{theo_appli}, the series $\mathfrak{B}^{(r)}(z)$, $\mathfrak{A}^{(s)}(z)$, and $\mathfrak{F}^{(k)}(z)$ are algebraically independent over $E_K$. 

In conclusion, for any integers $r,s, k\geq0$, the power series $\mathfrak{B}^{(r)}(z)$, $\mathfrak{A}^{(s)}(z)$, and $\mathfrak{F}^{(k)}(z)$ are algebraically independent over $E_K(\exp(\pi_5z))$. Since $\mathbb{Q}(z)(\exp(\pi_5z))\subset E_K(\exp(\pi_5z))$, it follows that $\mathfrak{B}^{(r)}(z)$, $\mathfrak{A}^{(s)}(z)$, and $\mathfrak{F}^{(k)}(z)$ are algebraically independent over $\mathbb{Q}(z)(\exp(\pi_5z))$. Finally, since $\exp(\pi_5z)$ is transcendental over $\mathbb{Q}(z)$, we conclude that the power series $\exp(\pi_5z)$, $\mathfrak{B}^{(r)}(z)$, $\mathfrak{A}^{(s)}(z)$, and $\mathfrak{F}^{(k)}(z)$ are algebraically independent over $\mathbb{Q}(z)$.
\end{proof}

\end{document}